%% file: Homomesies_on_permutations.tex
\documentclass{amsart}
\usepackage{amsmath,amssymb,fullpage}
\usepackage{amsthm}
\usepackage{amsfonts,newlfont,url,xspace}
\usepackage{ stmaryrd }
\usepackage[dvipsnames]{xcolor}
\usepackage{graphics,graphicx,verbatim}
\usepackage{float}
\usepackage{hyperref}
\usepackage{soul}
\usepackage[foot]{amsaddr}  
\usepackage{tikz}
\usetikzlibrary {arrows.meta} 

\include{pythonlisting}  

\newtheorem{thm}{Theorem}[section]
\newtheorem{lem}[thm]{Lemma}
\newtheorem{cor}[thm]{Corollary}
\newtheorem{prop}[thm]{Proposition}

\theoremstyle{definition}
\newtheorem{definition}[thm]{Definition}
\newtheorem{example}[thm]{Example}
\newtheorem{remark}[thm]{Remark}
\newtheorem{prob}[thm]{Problem}

\DeclareMathOperator{\lcm}{lcm}

\usepackage{ifxetex,ifluatex}
\if\ifxetex T\else\ifluatex T\else F\fi\fi T%
  \usepackage{fontspec}
\else
  \usepackage[T1]{fontenc}
  \usepackage[utf8]{inputenc}
  \usepackage{lmodern}
\fi

\newcommand{\K}{\mathcal{K}} 
\renewcommand\L{\mathcal{L}} 
\newcommand{\F}{\mathcal{F}} 
\newcommand{\M}{\mathcal{M}} 
\newcommand{\C}{\mathcal{C}} 
\newcommand{\R}{\mathcal{R}} 

\newcommand{\sage}{{SageMath}\xspace}
\newcommand{\Findstat}{{FindStat}\xspace}

\newcommand{\sinv}{\sigma^{-1}}
\newcommand{\inv}{\textnormal{inv}}
\newcommand{\Inv}{\textnormal{Inv}}
\newcommand{\maj}{\textnormal{maj}}
\newcommand{\exc}{\textnormal{exc}}
\newcommand{\fp}{\textnormal{fp}(\sigma)}
\newcommand{\Des}{\textnormal{Des}}
\newcommand{\des}{\textnormal{des}}

\newcommand{\rank}{\textnormal{rank}}
\newcommand{\Stat}{\textnormal{Stat}}
\newcommand{\fix}{\textnormal{fix}}
\newcommand{\bb}{\textbf} 

\title[Homomesies on permutations]{Homomesies on permutations: an analysis of maps and statistics in the FindStat database}
\author[1]{Jennifer Elder$^1$}
\address[1]{Rockhurst University. \href{mailto:flattenedparkingfunctions@outlook.com}{flattenedparkingfunctions@outlook.com}}
\author[2]{Nadia Lafreni\`ere$^2$}
\address[2]{Corresponding author. Dartmouth College, 6188 Kemeny Hall, 27 N. Main Street, Hanover, NH, 03755. \href{mailto:nadia.lafreniere@dartmouth.edu}{nadia.lafreniere@dartmouth.edu}}
\author[3]{Erin McNicholas$^3$}
\address[3]{Willamette University. \href{mailto:emcnicho@willamette.edu}{emcnicho@willamette.edu}}
\author[4]{Jessica Striker$^4$}
\address[4]{North Dakota State University. \href{mailto:jessica.striker@ndsu.edu}{jessica.striker@ndsu.edu}}
\author[5]{Amanda Welch$^5$}
\address[5]{Eastern Illinois University. \href{mailto:arwelch@eiu.edu}{arwelch@eiu.edu}}

\begin{document}
\maketitle

\begin{abstract}
  In this paper, we perform a systematic study of permutation statistics and bijective maps on permutations in which we identify and prove 122 instances of the homomesy phenomenon. Homomesy occurs when the average value of a statistic is the same on each orbit of a given map. The maps we investigate include the Lehmer code rotation, the reverse, the complement, the Foata bijection, and the Kreweras complement. The statistics studied relate to familiar notions such as inversions, descents, and permutation patterns, and also more obscure constructs.  Beside the many new homomesy results, we discuss our research method, in which we used SageMath to search the FindStat combinatorial statistics database to identify potential homomesies.
\end{abstract}



\section{Introduction}
Dynamical algebraic combinatorics is the study of objects important in algebra and combinatorics through the lens of dynamics.
In this paper, we focus on permutations, which are fundamental objects in algebra, combinatorics, representation theory, geometry, probability, and many other areas of mathematics. They are intrinsically dynamical, acting on sets by permuting their components. Here, we study bijections on permutations $f:S_n\rightarrow S_n$,
so the $n!$ elements being permuted are themselves permutations.  In particular, we find and prove many instances of \emph{homomesy}~\cite{PR2015}, an important phenomenon in dynamical algebraic combinatorics that occurs when the average value of some \emph{statistic} (a map $g:S_n\rightarrow \mathbb{Z}$) is the same over each orbit of the action. Homomesy occurs in many contexts, notably that of {rowmotion} on order ideals of certain families of posets and {promotion} on various sets of tableaux. See Subsection~\ref{sec:homomesy} for more specifics on homomesy and~\cite{Roby2016,Striker2017,SW2012} for further discussion.

A prototypical example of our homomesy results is as follows. Consider the Kreweras complement map $\K:S_n\rightarrow S_n$ (from Definition \ref{def:krew}). We show in Proposition \ref{Khom lastentry} that the last entry statistic of a permutation exhibits homomesy with respect to the Kreweras complement. See Figure~\ref{fig:ex} for an example of this result in the case $n=3$.

\begin{figure}[ht]
  \begin{tikzpicture}


    \node  [anchor=west] at (1,10) {\text{orbit of size $1$:}};

    \node[anchor=east] at (5,10) {$31{\color{red}2}$};

    \draw [->] (4.8,10.2) .. controls (5.5,10.7) and (5.5,9.3) .. (4.8,9.8);

    \node  [anchor=west] at (5.3,10) {$\mathcal{K}$};

    \node [anchor =west] at (10,10) {\text{average of last entry} $={\color{red}2}$};


    \node[anchor=west] at (1,8) {\text{orbit of size $2$:}};

    \node[anchor=east] at (5,8) {$12{\color{red}3}$};

    \draw [->] (5,8) -- (6,8);

    \node[anchor=south] at (5.5,8) {$\mathcal{K}$};

    \node [anchor=west] at (6,8) {$23{\color{red}1}$};

    \draw [<-] (4.8,7.8) .. controls (5.25,7.4) and (5.75,7.4) .. (6.2,7.8);

    \node[anchor=west] at (10,8) {\text{average of last entries} $=\frac{{\color{red}3+1}}{2}$};


    \node[anchor=west] at (1,6) {\text{orbit of size $3$:}};

    \node[anchor=east] at (5,6) {$13{\color{red}2}$};

    \draw [->] (5,6) -- (6,6);

    \node[anchor=center] at (6.5,6) {$21{\color{red} 3}$};

    \draw [->] (7,6) -- (8,6);

    \node[anchor=west] at (8,6) {$32{\color{red} 1}$};

    \draw [<-] (4.8,5.8) .. controls (6,5) and (7.2,5) .. (8.4,5.8);

    \node[anchor=south] at (5.5,6) {$\mathcal{K}$};

    \node[anchor=south] at (7.5,6) {$\mathcal{K}$};

    \node[anchor=north] at (6.5,5.2) {$\mathcal{K}$};

    \node[anchor=west] at (10,6) {\text{average of last entries} $=\frac{{\color{red} 2+3+1}}{3}$};

  \end{tikzpicture}
  \caption{Orbit decomposition of $S_3$ under the action of the Kreweras complement.  The last entry of each permutation is highlighted.  Calculating the averages of these last entries over each orbit, we observe an instance of homomesy.}\label{fig:ex}
\end{figure}
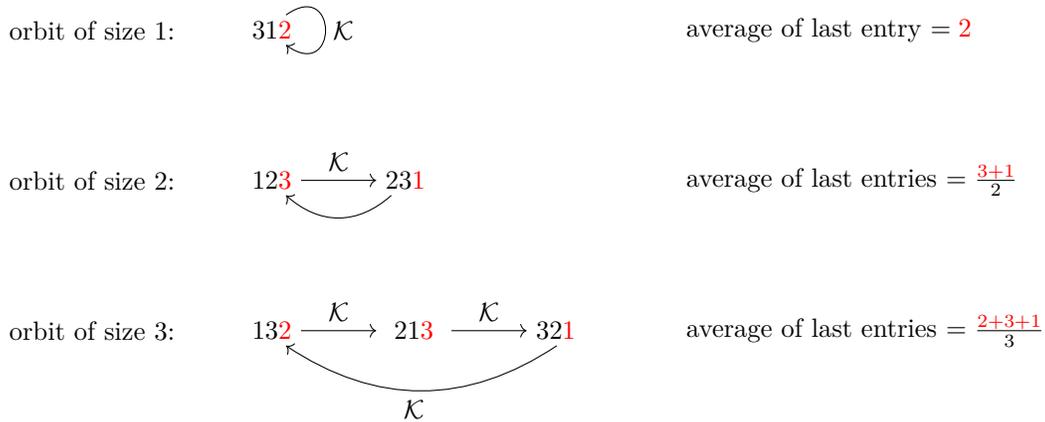

Rather than pick actions and statistics at random to test for homomesy, we used FindStat~\cite{FindStat}, the combinatorial statistics database, which (at the time of writing) included 387 permutation statistics and 19 bijective maps on permutations. Using the interface with SageMath~\cite{sage} computational software, we tested all combinations of these maps and statistics, finding 117 potential instances of homomesy, involving 68 statistics.

We highlight here some of the most interesting results.
One initial finding was that
homomesies occurred in only nine of the 19 examined maps. Among the maps that do not have any homomesic statistics, we find the well-known inverse map, as well as the first fundamental transform and cactus evacuation.  Of the nine maps exhibiting homomesy, four maps (all related to the \textbf{Foata bijection} map) have only one homomesic statistic.

Even more intriguing is the large number of homomesies found for the \textbf{Lehmer code rotation} map.
Despite its presence in FindStat, we could not find any occurrence of the Lehmer code rotation in the literature on combinatorial actions. The study in this paper suggests that this map is worthy of further investigation. Many of the homomesic statistics are related to inversions and descents, but other notable statistics
include several \emph{permutation patterns} as well as the \emph{rank} of a permutation.

As we worked through the proofs for the homomesic statistics for the \textbf{reverse} and \textbf{complement} maps, we found that the global averages are often the same. Using the relationship between the two maps (see Lemma \ref{lem:C&R_relation} ), we were able to prove many of the shared homomesies. Given this strong relationship, it is also of interest that there are several statistics that are only homomesic for one of the two maps.

In addition to exhibiting
homomesies, the action of the \textbf{Kreweras complement} map generates an interesting orbit structure on $S_n$.  Examining this orbit structure, we were able to characterize the distribution of all orbits.

Our main results are Theorems~\ref{thm:LC}, \ref{thmboth}, \ref{onlycomp}, \ref{onlyrev}, \ref{thm:foata} and \ref{Thm:Kreweras}, in which we prove all 117 of these homomesies.
In addition, we proved homomesy for 5 statistics not in the database, in Theorems \ref{thm:LC_inversions_at_entry} \ref{thm:descents_at_i_LC}, \ref{thm:inversion_positions_RC},\ref{Thm:Kreweras} and Proposition \ref{thm:ith_entry_comp}, for a grand total of 122 homomesic statistics. Furthermore, we found theorems on the orbit structure of the maps, chiefly  Theorems \ref{Thm: L-Orbit cardinality}, \ref{Prop:K even orbs} and \ref{thm:Orbit_generators_K}. We also give one open problem (Problem~\ref{prob:pp_LRC}).

This paper is organized as follows. In Section~\ref{sec:method}, we describe in detail our method of searching for potential homomesies. Section~\ref{sec:background} contains background material on homomesy and permutations. Sections~\ref{sec:lehmer} through \ref{sec:krew} contain our main results, namely, homomesies involving one or more related maps.
Each section begins by defining the map(s), followed by any additional results on properties of the maps. It then states as a theorem all the homomesic statistics, organized by theme. Finally, many propositions proving specific homomesies are given, which together, prove the main theorem(s) for the section. Below is a list of the map(s) for each section, along with the number of homomesic statistics from the FindStat database included in the corresponding theorem:

\begin{enumerate}
  \item[\ref{sec:lehmer}] Lehmer code rotation (45 homomesic statistics)
  \item[\ref{sec:comp_rev}] Complement and reverse (22 statistics homomesic for both maps, 5 statistics homomesic for reverse but not complement, and 13 statistics homomesic for complement but not reverse)
  \item[\ref{sec:foata}] Foata bijection and variations (4 maps all having the same single homomesic statistic)
  \item[\ref{sec:krew}] Kreweras complement and inverse Kreweras complement (3 homomesic statistics)
\end{enumerate}

\subsection*{Acknowledgements} The genesis for this project was the Research Community in Algebraic Combinatorics workshop, hosted by ICERM and funded by the NSF.  In addition to thanking ICERM and the organizers of this workshop, we wish to thank the developers of FindStat~\cite{FindStat}, especially moderators Christian Stump and Martin Rubey for their helpful and timely responses to our questions. We also thank the developers of SageMath~\cite{sage} software, which was useful in this research, and the CoCalc~\cite{SMC} collaboration platform. The anonymous referees suggested several improvements to this paper, for which we are grateful to them. We thank Joel Brewster Lewis for the description of the orbits of odd sizes under the Kreweras complement (Proposition~\ref{prop:orbits_of_odd_length}, Proposition~\ref{prop:orbits_of_size_d_odd}, and Theorem ~\ref{thm:orbit_count_odd_sizes}), and we thank Sergi Elizalde for suggesting examples of statistics that exhibit homomesy under the inverse map. JS was supported by a grant from the Simons Foundation/SFARI (527204, JS). 

\section{Summary of methods}
\label{sec:method}
\Findstat~\cite{FindStat} is an online database of combinatorial statistics developed by Chris Berg and Christian Stump in 2011 and   highlighted
as an example of a {fingerprint database} in the Notices of the American Mathematical Society~\cite{fingerprint}. \Findstat is
not only a searchable database which collects information, but also yields dynamic information about connections between combinatorial objects.  \Findstat takes statistics input by a user (via the website \cite{FindStat} or the \sage interface), uses \sage to apply combinatorial maps, and outputs corresponding statistics on other combinatorial objects.
\Findstat has grown expansively to a total of 1787 statistics on 23 combinatorial collections with 249 maps among them (as of April 27, 2022).

For this project, we analyzed all combinations of bijective maps and statistics on one combinatorial collection: permutations.
At the time of this investigation, there were 387 statistics and 19 bijective maps on permutations in \Findstat.  For each map/statistic pair, our empirical investigations either suggested a possible homomesy or provided a counterexample in the form of two orbits with differing averages. We then set about finding proofs for the experimentally identified potential homomesies.  These homomesy results are the main theorems of this paper: Theorems \ref{thm:LC}, \ref{thmboth}, \ref{onlycomp}, \ref{onlyrev}, \ref{thm:foata} and \ref{Thm:Kreweras}.

Thanks to the already existing interface between SageMath \cite{sage} and FindStat, we were able to automatically search for pairs of maps and statistics that exhibited homomesic behavior. For each value of $2 \leq n \leq 6$, we ran the following code to identify potential homomesies:

\begin{figure}[H]
\begin{python}
  sage: from sage.databases.findstat import FindStatMaps, FindStatStatistics # Access to the FindStat methods
  ....: findstat()._allow_execution = True  # To run all the code from Findstat
  ....: for map in FindStatMaps(domain="Cc0001", codomain= "Cc0001"):  # Cc0001 is the Permutation Collection
  ....:     if map.properties_raw().find('bijective') >= 0:  # The map is bijective
  ....:         F = DiscreteDynamicalSystem(Permutations(n), map)   # Fix n ahead of time
  ....:         for stat in FindStatStatistics("Permutations"):
  ....:             if F.is_homomesic(stat):
  ....:                 print(map.id(), stat.id())
\end{python}
\end{figure}

Note that the choice of running the verification of homomesies for permutations of $2$ to $6$ elements is not arbitrary: at $n=6$, the computational results stabilize. We did not find any false positives when using the data from $n=6$.
On the other hand, testing only smaller values of $n$ would have given us many false positives. For example, several statistics in the FindStat database involve the number of occurrences of some permutation patterns of length $5$. These statistics evaluate to $0$ for any permutation of fewer than $5$ elements, which misleadingly makes them appear $0$-mesic if not tested for values of $n$ at least $5$.

In the statements of our results, we refer to statistics and maps by their name as well as their FindStat identifier (ID). There is a webpage on FindStat associated to each statistic or map; for example, the URL for the inversion number, which has FindStat ID 18, is: \url{http://www.findstat.org/StatisticsDatabase/St000018}. We often refer to statistics as ``Statistic 18'' (or simply ``Stat 18'') and similarly for maps.
The FindStat database attributes IDs to statistics and maps sequentially; when we did the investigation, the maximum statistic ID in the FindStat database for a statistic on permutations was 1778, and the maximum map ID for a bijective map on permutations was 241. There were four statistics for which we could not disprove homomesy, because the database did not provide values for them on permutations of at least $5$ items, nor code to evaluate these statistics. Those are Statistics 1168, 1171, 1582 and 1583; they all correspond to the dimension of some vector spaces.

A visual summary of our results is given in Figure~\ref{fig:table}.

\begin{figure}[ht]
  \includegraphics[width=10cm]{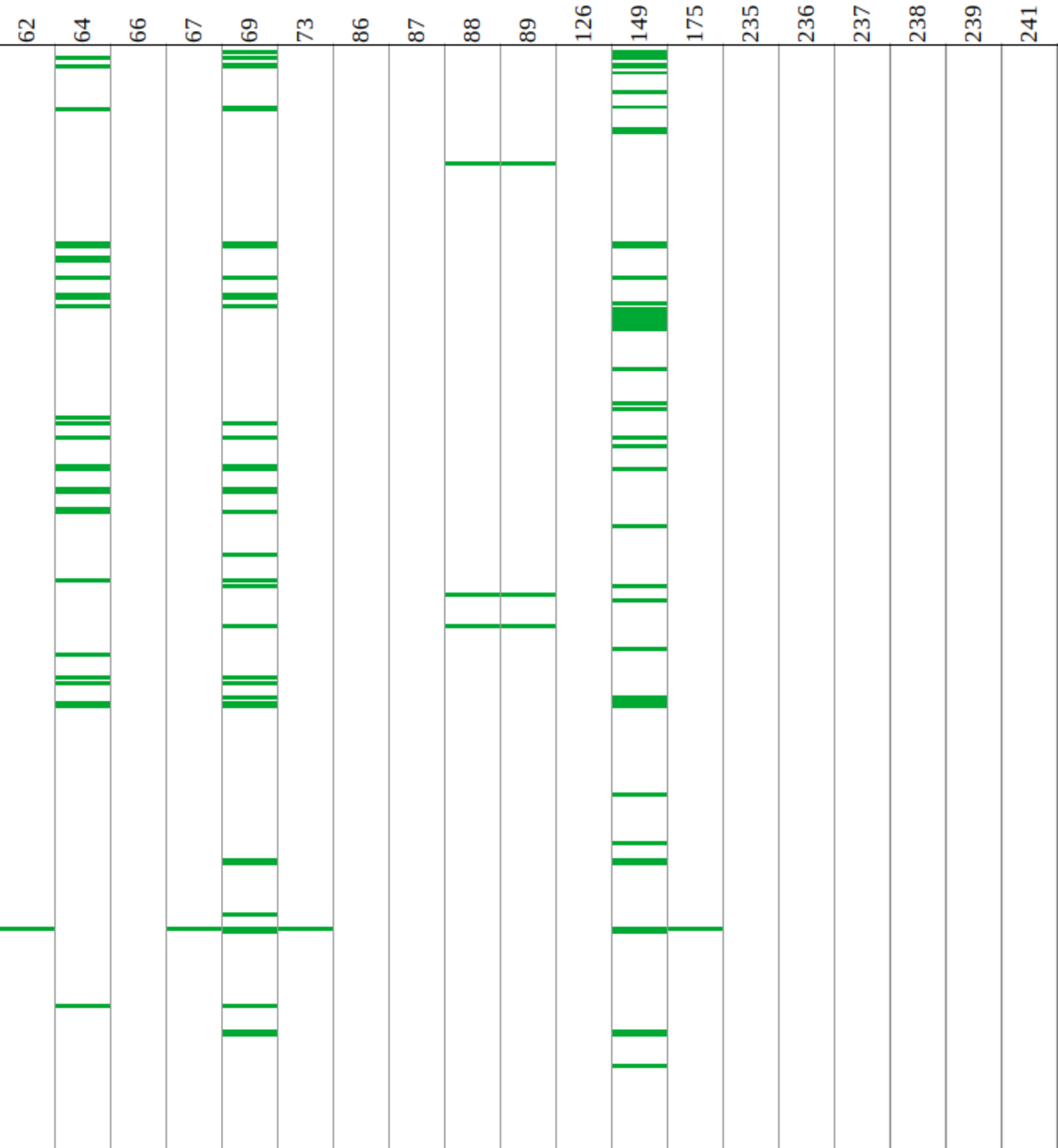}
  \caption{Each column corresponds to one of the 19 bijective maps on permutations stored in FindStat, and is labeled using the map's FindStat identifier.  The rows correspond to the 387 statistics on permutations. Green boxes correspond to the $117$ proven homomesies. The maps that have several homomesic statistics are the reverse map (64), the complement map (69), the Kreweras complement (88) and its inverse (89), and the Lehmer code rotation (149). The Lehmer-code to major-code bijection (62) and its inverse (73), and the Foata bijection (67) and its inverse (175) are also homomesic when paired with the statistic that is computed as the major index minus the number of inversions (1377). }\label{fig:table}
\end{figure}

\section{Background}
\label{sec:background}
This section gives background definitions and properties regarding the two main topics in our title: permutations (Subsection~\ref{sec:permutations}) and homomesy (Subsection~\ref{sec:homomesy}). It also discusses prior work on homomesy for maps on permutations in Subsection~\ref{sec:prior}.

\subsection{Permutations}
\label{sec:permutations}
Permutations are a central object in combinatorics, and many statistics on them are well-studied. We define here a few classical ones. Readers familiar with statistics on permutations may skip this subsection without loss of continuity.
\begin{definition}
  \label{def:basic_stats}
  Let  $[n] = \{1,2,\ldots, n\}$. A \textbf{permutation} $\sigma$ of $[n]$ is a bijection from $[n]$ to $[n]$
  in which the image of $i \in [n]$ is $\sigma_i$. We use the one-line notation, which means we write $\sigma = \sigma_1\sigma_2\ldots\sigma_n$. Permutations form a group called the \textbf{symmetric group}; we write $S_n$ for the set of all permutations of $[n]$.

  For a permutation $\sigma=\sigma_1\sigma_2\ldots\sigma_n$, we say that $(i,j)$ is an \textbf{inversion} of $\sigma$ if $i<j$ and $\sigma_j < \sigma_i$. We write $\Inv(\sigma)$ for the set of inversions in $\sigma$. We say that $(\sigma_i, \sigma_j)$ is an \textbf{inversion pair} of $\sigma$ if $(i, j)$ is an inversion of $\sigma$. This also corresponds to pairs $(\sigma_i, \sigma_j)$ with $\sigma_i > \sigma_j$ and $\sigma_i$ positioned to the left of $\sigma_j$ in $\sigma = \sigma_1\sigma_2 \ldots\sigma_n$. The \textbf{inversion number} of a permutation $\sigma$, denoted $\inv(\sigma)$, is the number of inversions.

  We say that $i$ is a \textbf{descent} exactly when $(i, i+1)$ is an inversion. At times we will instead say $\sigma$ has a descent at $i$. This also corresponds to the indices $i \in [n-1]$ such that $\sigma_i > \sigma_{i+1}$. We write $\Des(\sigma)$ for the set of descents in $\sigma$, and $\des(\sigma)=\#\Des(\sigma)$ for the number of descents.  If $i \in [n-1]$ is not a descent, we say that it is an \textbf{ascent}. We call a \textbf{peak} an ascent that is followed by a descent, and a \textbf{valley} a descent that is followed by an ascent.

  The \textbf{major index} of a permutation is the sum of its descents. We write $\maj(\sigma)$ to denote the major index of the permutation $\sigma$.

  A \textbf{run} in a permutation is a contiguous increasing sequence. Runs in permutations are separated by descents, so the number of runs is one more than the number of descents.
\end{definition}

\begin{example}
  For the permutation $\sigma =  216354$, the inversions are $\{(1,2), (3,4), (3,5), (3,6), (5,6)\}$, and the descents are $\{1,3,5\}$. There are two peaks ($2$ and $4$), and two valleys ($1$ and $3$), and the major index is $1+3+5=9$. The permutation has four runs, here separated with vertical bars: $2|16|35|4$.
\end{example}

\begin{definition}[Patterns in permutations]\label{def:patterns}
  We say that a permutation $\sigma$ of $[n]$ contains the \textbf{pattern} $abc$, with $\{a,b,c\} = \{1,2,3\}$, if there is a triplet $\{i_1 < i_2 < i_3\} \subseteq [n]$ such that $abc$ and $\sigma_{i_1}\sigma_{i_2}\sigma_{i_3}$ are in the same relative order. We call each such triplet an \textbf{occurrence} of the pattern $abc$ in $\sigma$.
\end{definition}

\begin{example}
  The permutation $\sigma=1324$ contains one occurrence of the pattern $213$, since $\sigma_2\sigma_3\sigma_4=324$ is in the same relative order as $213$.

  The permutation $415236$ contains four occurrences of the pattern $312$, because $412$, $413$, $423$ and $523$ all appear in the relative order $312$ in $415236$.
\end{example}

More generally, patterns of any length can be defined as subsequences of a permutation that appear in the same relative order as the pattern.

The definition above is for what we call \textit{classical} patterns. We can refine this notion by putting additional constraints on the triplets of positions that form the occurrences.
\begin{definition}\label{def:consecutive_patterns}
  The \textbf{consecutive pattern} (or \textbf{vincular pattern}) $a-bc$ (resp. $ab-c$) is a pattern in which $c$ occurs right after $b$ (resp. in which $b$ occurs right after $a$). The \textit{classical} pattern $abc$ corresponds to the pattern $a-b-c$.
\end{definition}
For example, the pattern $13-2$ means that we need to find three entries in the permutation such that
\begin{itemize}
  \item The smallest is immediately followed by the largest;
  \item The median entry comes after the smallest and the largest, but not necessarily immediately after.
\end{itemize}

\begin{example}
  The permutation $415236$ contains two occurrences of the pattern $3-12$, because $423$ and $523$ appear in the relative order $312$ in $415236$, with the last entries being adjacent in the permutation. On the other hand, $\{1,2,5\}$ is an occurrence of the pattern $312$  (since $413$ appears in the right order) that is not an occurrence of $3-12$.

  Inversions correspond to the classical pattern $21$ and to the consecutive pattern $2-1$, whereas descents correspond to the consecutive pattern $21$.
\end{example}

Unless otherwise specified, ``patterns'' refer to classical patterns.

\subsection{Homomesy}
\label{sec:homomesy}
First defined in 2015 by James~Propp and Tom~Roby~\cite{PR2015}, homomesy
relates the average of a given statistic over some set, to the averages over orbits formed by a bijective map. Note that in this paper, we use the word \textbf{map} instead of function or action, to match with the terminology in FindStat.

\begin{definition}
  Given a finite set $S$, an element $x\in S$, and an invertible map $\mathcal{X}:S \rightarrow S$, the \textbf{orbit} $\mathcal{O}(x)$ is the sequence consisting of $y_i\in S$ such that $y_i=\mathcal{X}^i(x) \textrm{ for some } i\in\mathbb{Z}$. That is, $\mathcal{O}(x)$ contains the elements of $S$ reachable from $x$ by applying $\mathcal{X}$ or $\mathcal{X}^{-1}$ any number of times. The \textbf{size} of an orbit is the number of unique elements in the sequence, denoted $|\mathcal{O}(x)|$. The \textbf{order} of $\mathcal{X}$ is the least common multiple of the sizes of the orbits.
\end{definition}

\begin{definition}[\cite{PR2015}]
  Given a finite set $S$, a bijective map $\mathcal{X}:S \rightarrow S$, and a statistic $f:S \rightarrow \mathbb{Z}$, we say that $(S, \mathcal{X}, f)$ exhibits \textbf{homomesy} if there exists $c \in \mathbb{Q}$  such that for every orbit $\mathcal{O}$
  \begin{center}
    $\displaystyle\frac{1}{|\mathcal{O}|} \sum_{x \in \mathcal{O}} f(x) = c$
  \end{center}
  where $|\mathcal{O}|$ denotes the size of $\mathcal{O}$. If such a $c$ exists, we say the triple is \textbf{$c$-mesic}.
\end{definition}

When the set $S$ is clear from context, we may say a statistic is \textbf{homomesic with respect to $\mathcal{X}$} rather than explicitly stating the triple. When the map $\mathcal{X}$ is also implicit, we may simply say a statistic is \textbf{homomesic}.
Homomesy may be generalized beyond the realms of bijective actions and integer statistics, but we will not address these generalizations in this paper.

\begin{remark}\label{global_avg} Note that whenever a statistic is homomesic, the orbit-average value is indeed the global average.
\end{remark}

We end this subsection with two general lemmas about homomesy that will be used later. In the interest of making this paper self-contained, we include proofs, though the results are well-known.

\begin{lem} \label{lem:inverse}
  If a triple $(S,\mathcal{X},f)$ is $c$-mesic, then so is $(S,\mathcal{X}^{-1},f). $
\end{lem}
\begin{proof}
  A bijective map and its inverse have exactly the same elements in their orbits, thus the orbit-averages for a given statistic are also equal.
\end{proof}

\begin{lem}
  \label{lem:sum_diff_homomesies}
  For a given action, linear combinations of homomesic statistics are also homomesic.
\end{lem}
\begin{proof}
  Suppose $f,g$ are homomesic statistics with respect to a bijective map $\mathcal{X}:S \rightarrow S$, where $S$ is a finite set. So $\displaystyle\frac{1}{|\mathcal{O}|} \sum_{x \in \mathcal{O}} f(x) = c$ and $\displaystyle\frac{1}{|\mathcal{O}|} \sum_{x \in \mathcal{O}} g(x) = d$ for some $c,d\in\mathbb{C}$. Let $a,b\in\mathbb{C}$. Then
  \[\displaystyle\frac{1}{|\mathcal{O}|} \sum_{x \in \mathcal{O}} (af+bg)(x) = a\displaystyle\frac{1}{|\mathcal{O}|} \sum_{x \in \mathcal{O}} f(x) + b\displaystyle\frac{1}{|\mathcal{O}|}  \sum_{x \in \mathcal{O}} g(x)=ac+bd.\]
  Thus, $af+bg$ is homomesic with respect to $\mathcal{X}$ with average value $ac+bd$.
\end{proof}

\subsection{Prior work on homomesy on permutations}
\label{sec:prior}
Since the homomesy phenomenon was defined, mathematicians have looked for it on natural combinatorial objects. Permutations indeed arose as such a structure, and some recent work initiated the study of homomesic statistics on permutations. Michael La\,Croix and Tom Roby \cite{LaCroixRoby} focused on the statistic counting the number of fixed points in a permutation, while they looked at compositions of the first fundamental transform of Foata with what they call ``dihedral actions'', a few maps that include the complement, the inverse and the reverse, all discussed in Section \ref{sec:comp_rev}. Simultaneously, Elizabeth Sheridan-Rossi considered a wider range of statistics for the same maps, as well as for the compositions of dihedral actions with the Foata bijection \cite{Sheridan-Rossi}.

Our approach differs from the previous studies by being more systematic. As previously mentioned, we proved or disproved homomesy for all the 7,345 combinations of a bijective map and a statistic on permutations that were in the FindStat database. It is worth noting that the interesting maps described in \cite{LaCroixRoby} and \cite{Sheridan-Rossi} are compositions of FindStat maps, but they are not listed as single maps in FindStat. We did not consider compositions of FindStat maps; this would be an interesting avenue for further study.

\section{Lehmer code rotation}
\label{sec:lehmer}
An important way to describe a permutation is through its inversions. The Lehmer code of a permutation (defined below) completely characterizes it, so we have a bijection between Lehmer codes and permutations. 

In this section, after describing the Lehmer code, we define the Lehmer code rotation map. We then state Theorem~\ref{thm:LC} which lists
the 45 statistics in FindStat that are homomesic for this map. Before proving this theorem, in Subsection~\ref{subsec:lehmer_orbit} we describe the orbits of the Lehmer code rotation and make connections with actions on other combinatorial objects in Remark~\ref{remark:lehmer_connections}. The homomesies are then proved, starting with statistics related to inversions (Subsection \ref{subsec:lehmer_inv}), then those related to descents (Subsection \ref{subsec:lehmer_des}), to permutation patterns (Subsection \ref{subsec:lehmer_pp}), and finishing with a few other statistics (Subsection \ref{subsec:lehmer_misc}). We also give one open problem related to homomesic permutation patterns for the Lehmer code rotation (Problem~\ref{prob:pp_LRC}).

\begin{definition}
  \label{def:lehmercode}
  The \textbf{Lehmer code} of a permutation $\sigma\in S_n$ is:
  \[L(\sigma )=(L(\sigma )_{1},\ldots, L(\sigma )_{n})\quad {\text{where}}\quad L(\sigma )_{i}=\#\{j>i\mid \sigma _{j}<\sigma _{i}\}.\]
\end{definition}
It is well known (see for example \cite[p.12]{Knuth_AOCP3}), that there is a bijection between tuples of length $n$ whose entries are integers between $0$ and $n-i$ at position $i$ and permutations. Hence, the Lehmer code uniquely defines a permutation.

\begin{example}
  The Lehmer code of the permutation 31452 is  $L(31452)= (2,0,1,1,0)$, whereas the Lehmer code of $42513$ is $L(42513) = (3,1,2,0,0)$.
\end{example}

Since the entries of the Lehmer code count the number of inversions that start at each entry of the permutation, the following observation follows:
\begin{prop}\label{prop:LC_num_inv_is_sum}
  The number of inversions in the permutation $\sigma$ is given by $\sum_{i=1}^n L(\sigma)_i$.
\end{prop}

\begin{definition}
  The \textbf{Lehmer code rotation} (FindStat map 149) is a map that sends a permutation $\sigma$ to the unique permutation $\tau$ (of the same set) such that every entry in the Lehmer code of $\tau$ is cyclically (modulo $n+1-i$) one larger than the Lehmer code of $\sigma$. In symbols:
  \begin{align*}
    \L :  \sigma & \mapsto \tau                                     \\
    L(\sigma)_i  & \mapsto L(\tau)_i =  L(\sigma)_i+1 \mod (n-i+1).
  \end{align*}
\end{definition}
An example is illustrated in Figure \ref{fig:LCR}.

\begin{figure}
  \centering
  \includegraphics{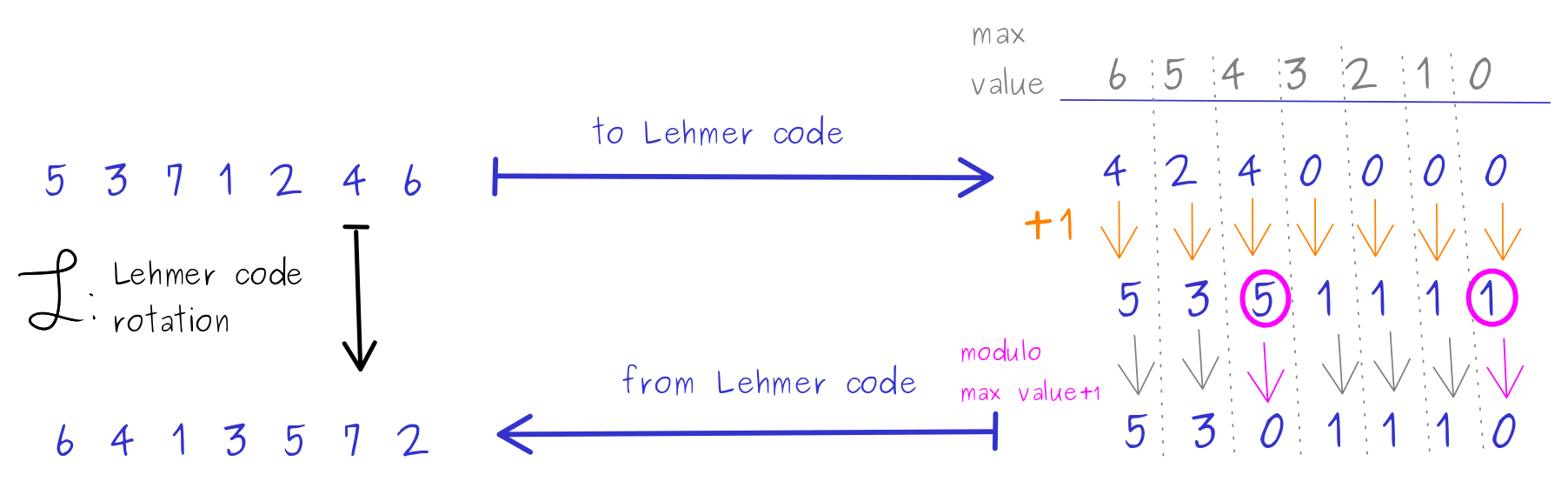}
  \caption{The Lehmer code rotation applied on the permutation $5371246$ yields the permutation $6413572$. The step-by-step process is illustrated by this picture.}\label{fig:LCR}
\end{figure}

\begin{example}
  The permutation $\sigma = 31452$ has Lehmer code $L(\sigma) = (2,0,1,1,0)$. Hence, \[L(\L(\sigma)) = (2+1 \mod 5,\ 0+1 \mod 4,\ 1+1\mod 3,\ 1+1\mod 2,\ 0+1 \mod 1) = (3,1,2,0,0).\] Because $(3,1,2,0,0)$ is the Lehmer code of the permutation $42513$, $\L(31452) = 42513$. \end{example}

\begin{remark}
\label{remark:lehmer_connections}
Despite its presence in FindStat, we could not find the Lehmer code rotation map in the literature. However, we did find six similar maps in a paper by Vincent Vajnovszki \cite{vajnovszki}. Although those were not in the FindStat database, we tested them and found they did not exhibit interesting homomesies.

Also, the Lehmer code rotation on permutations is equivalent to the toggle group action of \emph{rowmotion}~\cite{SW2012} on the poset constructed as the disjoint union of chains of $i$ elements for $1\leq i\leq n-1$. As noted by Martin Rubey (personal communication), the distributive lattice of order ideals of this poset forms a $\omega$-sorting order in the sense of \cite{Armstrong2009} with sorting word $\omega=[1,\ldots,n,1,\ldots,n-1,1,\ldots,n-2,\ldots,1,2,1]$.
\end{remark}

The main theorem of this section is the following.
\begin{thm}\label{thm:LC}
  The Lehmer code rotation map exhibits homomesy for the following $45$ statistics found in the FindStat database:
  \begin{itemize}
    \item Statistics related to inversions:
          \begin{itemize} \rm
            \item \hyperref[18_246_LC]{\Stat ~$18$}: The number of inversions of a permutation $(${\small average: $\frac{n(n-1)}{4}$}$)$
            \item \hyperref[18_246_LC]{\Stat ~$246$}: The number of non-inversions of a permutation $(${\small average: $\frac{n(n-1)}{4}$}$)$
            \item \hyperref[495_836_837_LC]{\Stat ~$495$}: The number of inversions of distance at most $2$ of a permutation $(${\small average: $\frac{2n-3}{2}$}$)$
            \item \hyperref[54_1556_1557_LC]{\Stat ~$1556$}: The number of inversions of the third entry of a permutation $(${\small average: $\frac{n-3}{2}$}$)$
            \item \hyperref[54_1556_1557_LC]{\Stat ~$1557$}: The number of inversions of the second entry of a permutation $(${\small average: $\frac{n-2}{2}$}$)$
          \end{itemize}
    \item Statistics related to descents:
          \begin{itemize} \rm
            \item \hyperref[4_21_245_833_LC]{\Stat ~$4$}: The major index of a permutation $(${\small average: $\frac{n(n-1)}{4}$}$)$
            \item \hyperref[4_21_245_833_LC]{\Stat ~$21$}: The number of descents of a permutation$(${\small average: $\frac{n-1}{2}$}$)$
            \item \hyperref[23_353_365_366_LC]{\Stat ~$23$}: The number of inner peaks of a permutation $(${\small average: $\frac{n-2}{3}$}$)$
            \item \hyperref[35_92_99_483_834_LC]{\Stat ~$35$}: The number of left outer peaks of a permutation $(${\small average: $\frac{2n-1}{6}$}$)$
            \item \hyperref[35_92_99_483_834_LC]{\Stat ~$92$}: The number of outer peaks of a permutation $(${\small average: $\frac{n+1}{3}$}$)$
            \item \hyperref[35_92_99_483_834_LC]{\Stat ~$99$}: The number of valleys of a permutation, including the boundary $(${\small average: $\frac{n+1}{3}$}$)$
            \item \hyperref[4_21_245_833_LC]{\Stat ~$245$}: The number of ascents of a permutation $(${\small average: $\frac{n-1}{2}$}$)$
            \item \hyperref[23_353_365_366_LC]{\Stat ~$353$}: The number of inner valleys of a permutation $(${\small average: $\frac{n-2}{3}$}$)$
            \item \hyperref[23_353_365_366_LC]{\Stat ~$365$}: The number of double ascents of a permutation $(${\small average: $\frac{n-2}{6}$}$)$
            \item \hyperref[23_353_365_366_LC]{\Stat ~$366$}: The number of double descents of a permutation$(${\small average: $\frac{n-2}{6}$}$)$
            \item \hyperref[325_470_LC]{\Stat ~$470$:} The number of runs in a permutation $(${\small average: $\frac{n+1}{2}$}$)$
            \item \hyperref[35_92_99_483_834_LC]{\Stat ~$483$}: The number of times a permutation switches from increasing to decreasing or decreasing to increasing $(${\small average: $\frac{2n-4}{3}$}$)$
            \item \hyperref[638_LC]{\Stat ~$638$}: The number of up-down runs of a permutation $(${\small average: $\frac{4n+1}{6}$}$)$
            \item \hyperref[4_21_245_833_LC]{\Stat ~$833$}: The comajor index of a permutation $(${\small average: $\frac{n(n-1)}{4}$}$)$
            \item \hyperref[35_92_99_483_834_LC]{\Stat ~$834$}: The number of right outer peaks of a permutation$(${\small average: $\frac{2n-1}{6}$}$)$
            \item \hyperref[495_836_837_LC]{\Stat ~$836$}: The number of descents of distance $2$ of a permutation $(${\small average: $\frac{n-2}{2}$}$)$
            \item \hyperref[495_836_837_LC]{\Stat ~$837$}: The number of ascents of distance $2$ of a permutation $(${\small average: $\frac{n-2}{2}$}$)$
            \item \hyperref[1114_1115_LC]{\Stat ~$1114$}: The number of odd descents of a permutation $(${\small average: $ \frac{1}{2}\lceil\frac{n-1}{2}\rceil$}$)$
            \item \hyperref[1114_1115_LC]{\Stat ~$1115$}: The number of even descents of a permutation $(${\small average: $\frac{1}{2}\lfloor\frac{n-1}{2} \rfloor$}$)$
          \end{itemize}

    \item Statistics related to permutation patterns:
          \begin{itemize} \rm
            \item \hyperref[355_to_360_LC]{\Stat ~$355$}: The number of occurrences of the pattern $21-3$  $(${\small average: $\frac{(n-1)(n-2)}{12}$ }$)$ 
            \item \hyperref[355_to_360_LC]{\Stat ~$356$}: The number of occurrences of the pattern $13-2$ $(${\small average: $\frac{(n-1)(n-2)}{12}$ }$)$ 
            \item \hyperref[355_to_360_LC]{\Stat ~$357$}: The number of occurrences of the pattern $12-3$ $(${\small average: $\frac{(n-1)(n-2)}{12}$ }$)$ 
            \item \hyperref[355_to_360_LC]{\Stat ~$358$}: The number of occurrences of the pattern $31-2$ $(${\small average: $\frac{(n-1)(n-2)}{12}$ }$)$ 
            \item \hyperref[355_to_360_LC]{\Stat ~$359$}: The number of occurrences of the pattern $23-1$ $(${\small average: $\frac{(n-1)(n-2)}{12}$ }$)$ 
            \item \hyperref[355_to_360_LC]{\Stat ~$360$}: The number of occurrences of the pattern $32-1$ $(${\small average: $\frac{(n-1)(n-2)}{12}$ }$)$ 
            \item \hyperref[423_435_437_LC]{\Stat ~$423$}: The number of occurrences of the pattern $123$ or of the pattern $132$ in a permutation  $(${\small average: $\frac{1}{3}\binom{n}{3}$ }$)$ 
            \item \hyperref[423_435_437_LC]{\Stat ~$435$}: The number of occurrences of the pattern $213$ or of the pattern $231$ in a permutation  $(${\small average: $\frac{1}{3}\binom{n}{3}$ }$)$ 
            \item \hyperref[423_435_437_LC]{\Stat ~$437$}: The number of occurrences of the pattern $312$ or of the pattern $321$ in a permutation  $(${\small average: $\frac{1}{3}\binom{n}{3}$ }$)$ 
            \item \hyperref[709_LC]{\Stat ~$709$}: The number of occurrences of $14-2-3$ or $14-3-2$  $(${\small average: $\frac{1}{12}\binom{n-1}{3}$ }$)$ 
            \item \hyperref[1084_LC]{\Stat ~$1084$}: The number of occurrences of the vincular pattern $|1-23$ in a permutation   $(${\small average: $\frac{n-2}{6}$ }$)$ 
          \end{itemize}
    \item Other statistics:
          \begin{itemize} \rm
            \item \hyperref[7_991_LC]{\Stat ~$7$}: The number of saliances (right-to-left maxima) of the permutation  $(${\small average: $H_n = \sum_{i=1}^n \frac{1}{i}$}~$)$ 
            \item \hyperref[20_LC]{\Stat ~$20$}: The rank of the permutation (among the permutations, in lexicographic order)  $(${\small average: $\frac{n!+1}{2}$ }$)$ 
            \item \hyperref[54_1556_1557_LC]{\Stat ~$54$}: The first entry of the permutation $(${\small average: $\frac{n+1}{2}$ }$)$ 
            \item \hyperref[325_470_LC]{\Stat ~$325$}: The width of a tree associated to a permutation $(${\small average: $\frac{n+1}{2}$ }$)$ 
            \item \hyperref[692_796_LC]{\Stat ~$692$}: Babson and Steingrímsson's statistic stat of a permutation  $(${\small average: $\frac{n(n-1)}{4}$ }$)$ 
            \item \hyperref[692_796_LC]{\Stat ~$796$}: Babson and Steingrímsson's statistic stat' of a permutation  $(${\small average: $\frac{n(n-1)}{4}$ }$)$ 
            \item \hyperref[7_991_LC]{\Stat ~$991$}: The number of right-to-left minima of a permutation $(${\small average: $H_n = \sum_{i=1}^n \frac{1}{i}$ }$)$ 
            \item \hyperref[1377_1379_LC]{\Stat ~$1377$}: The major index minus the number of inversions of a permutation $(${\small average: $0$ }$)$ 
            \item \hyperref[1377_1379_LC]{\Stat ~$1379$}: The number of inversions plus the major index of a permutation $(${\small average: $\frac{n(n-1)}{2}$ }$)$ 
            \item \hyperref[1640_LC]{\Stat ~$1640$}: The number of ascent tops in the permutation such that all smaller elements appear before $(${\small average: $1-\frac{1}{n}$ }$)$ 
          \end{itemize}
  \end{itemize}
\end{thm}

\subsection{Orbit structure}
\label{subsec:lehmer_orbit}
Before beginning the proof of Theorem~\ref{thm:LC}, we show a few results on the orbit structure of the Lehmer code rotation.

\begin{thm}[Orbit cardinality]\label{Thm: L-Orbit cardinality} All orbits of the Lehmer code rotation have size $\lcm(1,2,\ldots, n)$.
\end{thm}

\begin{proof}
  Since there is a bijection between permutations and their Lehmer codes, one can look at the orbit of the map directly on the Lehmer code. We know that $L(\L(\sigma))_i = L(\sigma)_i + 1 \mod (n+1-i)$. Therefore, the minimum $k >0$ such that $L(\sigma)_i = L(\L^k(\sigma))_i$ is $k = n+1-i$. Looking at all values $1 \leq i \leq n$, we need $\lcm(1,2,3,\ldots, n)$ iterations of $\L$ to get back to the original Lehmer code.
\end{proof}

The following four lemmas about entries of the Lehmer code over an orbit will be useful to prove homomesies related to inversions and descents.
\begin{lem}\label{lem:equioccurrences_Lehmer_code}
  Over one orbit of the Lehmer code rotation, the numbers $\{0, 1, 2, \ldots, n-i\}$ all appear equally often as $L(\sigma)_i$, the $i$-th entry of the Lehmer code.
\end{lem}

\begin{proof}
  We know that $L(\L(\sigma))_i = L(\sigma)_i + 1 \mod (n+1-i)$, which also means that  $L(\L^j(\sigma))_i = L(\sigma)_i + j \mod (n+1-i)$. Using the fact that each orbit has size $\lcm(1, \ldots, n)$, the $i$-th entry of the Lehmer code has each value in $\{0, \ldots, n-i\}$ appearing exactly $\frac{\lcm(1,\ldots,n)}{n+1-i}$ times.
\end{proof}

\begin{lem}\label{lem:equioccurrences_of_pairs_Lehmer_code}
  Over one orbit of the Lehmer code rotation, the pairs $\{(a,b) \mid a \in \{0,\ldots, n-i\}, b \in \{0,\ldots, n-i-1\}\}$ all appear equally often as $(L(\sigma)_i, L(\sigma)_{i+1})$.  That is, pairs of possible adjacent entries are all equally likely over any given orbit.
\end{lem}

\begin{proof}
  Let $(L(\sigma)_i,L(\sigma)_{i+1}) = (a,b)$. Then $\Big(L(\L(\sigma))_i, L(\L(\sigma))_{i+1}\Big) = \Big(a+1 \mod (n-i+1), b+1 \mod (n-i)\Big)$.  Since $n-i$ and $n-i+1$ are coprime, the successive application of $\L$  spans all the possibilities for $(a,b)$ exactly once before returning to $(a,b)$ in $(n-i)(n-i+1)$ steps.
\end{proof}

The latter statement can be expanded to qualify the independence of non-adjacent entries, as explained by the next two lemmas.
\begin{lem}\label{lem:equioccurrences_of_distant_pairs_Lehmer_code}
  Over one orbit of the Lehmer code rotation, the pairs $\{(a,b) \mid a \in \{0,\ldots, n-i\}, b \in \{0,\ldots, n-j\}\}$ all appear equally often as $(L(\sigma)_i, L(\sigma)_{j})$ if $n-i+1$ and $n-j+1 $ are coprime.
\end{lem}

\begin{proof}
  Let $(L(\sigma)_i,L(\sigma)_{j}) = (a,b)$. Then $\Big(L(\L(\sigma))_i, L(\L(\sigma))_{j}\Big) = \Big(a+1 \mod (n-i+1), b+1 \mod (n-j+1)\Big)$.  Since $n-i+1$ and $n-j+1$ are coprime, the successive application of $\L$  spans all the possibilities for $(a,b)$ exactly once before returning to $(a,b)$ in $(n-i+1)(n-j+1)$ steps.
\end{proof}

The lemma above applies when $n-i-1$ and $n-j-1$ are coprime. The case where they are not is considered in the following two lemmas.

\begin{lem}\label{lem:pairs_in_Lehmer_code_with_same_parities}
  Over one orbit of the Lehmer code rotation, the number $L(\sigma)_i-L(\sigma)_j \mod k$ has a constant value for all the values $i$ and $j$ for which $k$ is a divisor of both $n-i+1$ and $n-j+1$.
\end{lem}

\begin{proof}
  For any $m$, it suffices to show that $$L(\sigma)_i - L(\sigma)_j \mod k = L(\L^m(\sigma))_i - L(\L^m(\sigma))_j \mod k.$$ We know that $$L(\L^m(\sigma))_i - L(\L^m(\sigma))_j \mod k = \big(L(\sigma)_i +m \mod (n-i+1)\big)- \big(L(\sigma)_j+m \mod (n-j+1)\big) \mod k .$$ Since $k$ divides both $n-i+1$ and $n-j+1$, the latter is equal to $$L(\sigma)_i  +m - L(\sigma)_j-m  \mod k = L(\sigma)_i  - L(\sigma)_j \mod k.$$ Putting the pieces together, we conclude that $$L(\sigma)_i - L(\sigma)_j \mod k = L(\L^m(\sigma))_i - L(\L^m(\sigma))_j \mod k.$$
\end{proof}

\begin{lem}\label{lem:equioccurrences_pairs_distance_2}
  If $n-i$ is even, there exists  for each orbit of the Lehmer code rotation a value $r \in \{0,1\}$ such that the pairs $\{(a,b) \mid a \in \{0, \ldots, n-i+1\}, b \in \{0, \ldots, n-i-1\}, \text{ with }a - b = r \mod 2\}$ all appear equally often as $(L(\sigma)_{i-1}, L(\sigma)_{i+1})$.
\end{lem}

\begin{proof} When $n-i$ is even, both $n-i+2$ and $n-i$ are even, meaning that their greatest common divisor is $2$.
  Thanks to Lemma \ref{lem:pairs_in_Lehmer_code_with_same_parities}, we know that $L(\sigma)_{i-1}-L(\sigma)_{i+1}$ has a constant value modulo $2$ over each orbit. What is left to prove is that all pairs $(a,b)$ that satisfy this constraint are equally likely to happen as $(L(\sigma)_{i-1}, L(\sigma)_{i+1})$.

  Acting $m$ times with the Lehmer code rotation,  we look at the evolution of the pair $\big(L(\sigma)_{i-1},L(\sigma)_{i+1}\big)$:
  \[ \big(L(\L^m(\sigma))_{i-1}, L(\L^m(\sigma))_{i+1}\big) = \big(a + m \mod (n-i+2), b + m \mod (n-i) \big), \]
  and the orbit spans all possible combinations satisfying the parity condition before returning back to $(a,b)$ in $m = \frac{(n-i+2)(n-i)}{2}$ steps.
\end{proof}

\subsection{Statistics related to inversions}
\label{subsec:lehmer_inv}
In this subsection, we state and prove propositions giving the homomesies related to inversions of Theorem~\ref{thm:LC}, as well as homomesy for a family of statistics that do not appear in FindStat (see Theorem \ref{thm:LC_inversions_at_entry}).

Recall inversions of a permutation from Definition \ref{def:basic_stats}.

\begin{prop}[Statistics 18, 246]\label{18_246_LC}
  The number of inversions is $\frac{n(n-1)}{4}$-mesic with respect to the
  Lehmer code rotation for permutations of $[n]$.
  Similarly, the number of noninversions is also $\frac{n(n-1)}{4}$-mesic.
\end{prop}

\begin{proof}
  Following Proposition \ref{prop:LC_num_inv_is_sum}, the number of inversions of $\sigma$ is the sum of the numbers of the Lehmer code $L(\sigma)$. It suffices to observe, as we did in Lemma \ref{lem:equioccurrences_Lehmer_code}, that  each number in $\{0, \ldots, n-i\}$ occurs equally often as $L(\sigma)_i$. Therefore, the average value of $L(\sigma)_i$ is $\frac{n-i}{2}$. Thus, the average number of inversions is the sum of the average values at each position of the Lehmer code, which is:
  \[\sum_{i=1}^{n} \frac{n-i}{2} = \sum_{k=0}^{n-1}\frac{k}{2} = \frac{n(n-1)}{4}. \]
  Since the noninversions are exactly the pairs $(i,j)$ that are not inversions, there are $\frac{n(n-1)}{2}-\inv(\sigma)$ noninversions in permutation $\sigma \in S_n$, so the number of non-inversions is also $\frac{n(n-1)}{4}$-mesic.
\end{proof}

\begin{thm}\label{thm:LC_inversions_at_entry}
  The number of inversions starting at the $i$-th entry of a permutation is $\frac{n-i}{2}$-mesic for the Lehmer Code rotation on permutations of $[n]$.
\end{thm}

\begin{proof}
  This follows from Lemma \ref{lem:equioccurrences_Lehmer_code}, since, for each $i$, the number of inversions starting at the $i$-th entry of a permutation $\sigma$ is exactly $L(\sigma)_i$. Hence, $L(\sigma)_i$ is $\frac{n-i}{2}$-mesic.
\end{proof}

\begin{cor}[Statistics 54, 1556, 1557]
  \label{54_1556_1557_LC}
  The following FindStat statistics are homomesic:
  \begin{itemize}
    \item The first entry of the permutation is $\frac{n+1}{2}$-mesic.
    \item The number of inversions of the second entry of a permutation is $\frac{n-2}{2}$-mesic.
    \item The number of inversions of the third entry of a permutation is $\frac{n-3}{2}$-mesic.
  \end{itemize}
\end{cor}

\begin{proof}
  The second and third statements are special cases of Theorem \ref{thm:LC_inversions_at_entry}.
  As for the first entry of a permutation $\sigma$, its value is $L(\sigma)_1+1$; following Theorem \ref{thm:LC_inversions_at_entry}, this is $\frac{n+1}{2}$-mesic.
\end{proof}

\begin{remark}
  While the above corollary shows that the first entry of the permutation is a homomesic statistic, this does not hold for the other entries of a permutation. For example, statistic 740 (last entry of a permutation) is not homomesic under Lehmer code rotation; the number of inversions starting with the last entry of a permutation is always indeed $0$ and is unrelated to the value of the last entry.
\end{remark}

\subsection{Statistics related to descents}
\label{subsec:lehmer_des}
In this subsection, we state and prove propositions giving the homomesies of Theorem~\ref{thm:LC} related to descents. Furthermore, Theorem \ref{thm:descents_at_i_LC} proves some homomesies that do not correspond to FindStat statistics, but are helpful in proving some of them.

Recall that descents of a permutation, as well as affiliated concepts and notations, are given in Definition~\ref{def:basic_stats}.
We begin with the following lemma, which will be helpful in proving Theorem~\ref{thm:descents_at_i_LC}.
\begin{lem}\label{lem:descents_correspondence_in_Lehmer_code}
  Descents in a permutation correspond exactly to strict descents of the Lehmer code; formally, $i$ is a descent of $\sigma$ if and only if $L(\sigma)_{i} > L(\sigma)_{i+1}$.
\end{lem}

\begin{proof}
  If $i$ is a descent of $\sigma$, then $\sigma_i > \sigma_{i+1}$, and
  \[ L(\sigma)_i  = \#\{ j>i\mid \sigma_i > \sigma_j \} = \#\{ j>i+1 \mid  \sigma_i > \sigma_j \} + 1 \geq \#\{ j>i+1 \mid \sigma_{i+1} > \sigma_j \} + 1 = L(\sigma)_{i+1}+1.  \]
  Therefore, $L(\sigma)_i > L(\sigma)_{i+1}$.

  We prove the converse by contrapositive: we assume $\sigma_i <\sigma_{i+1}$ (so $i$ is an ascent). Then,
  \[ L(\sigma)_i  = \#\{ j>i \mid \sigma_{i} > \sigma_{j} \} = \#\{ j>i+1 \mid \sigma_i > \sigma_j \}  \leq \#\{ j>i+1 \mid \sigma_{i+1} > \sigma_j \}  = L(\sigma)_{i+1}.  \]
\end{proof}

\begin{thm}\label{thm:descents_at_i_LC}
  The number of descents at position $i$ is $\frac{1}{2}$-mesic under the Lehmer code rotation, for $1\leq i< n$.
\end{thm}

\begin{proof}
  The key to this result is Lemma \ref{lem:descents_correspondence_in_Lehmer_code}, saying that $i$ is a descent of $\sigma$ if and only if $L(\sigma)_{i+1} < L(\sigma)_i$.
  Therefore, counting descents at position $i$ in permutations corresponds to counting strict descents at position $i$ in the Lehmer code.
  We look at all possible adjacent pairs of entries in the Lehmer code. The entry at position $i$ can take any value in $\{0, \ldots, n-i\}$, so there are $n-i+1$ options for the $i$-th entry  and $n-i$ for the $(i+1)$-st. Since $n-i+1$ and $n-i$ are coprime, all possible pairs of adjacent entries are equally likely to occur in each orbit of the Lehmer code rotation (this is the result of Lemma \ref{lem:equioccurrences_of_pairs_Lehmer_code}). Hence, the proportion of strict descents in the Lehmer code at position $i$ is
  \[\frac{1}{(n-i)(n-i+1)}\sum_{k=1}^{n-i+1}(k-1) = \frac{(n-i)(n-i+1)}{2(n-i)(n-i+1)} =  \frac{1}{2}. \]
\end{proof}

\begin{prop}[Statistics 4, 21, 245, 833]\label{4_21_245_833_LC}
  The number of descents in a permutation of $[n]$ is $\frac{n-1}{2}$-mesic under Lehmer code rotation, and the number of ascents is $\frac{n-1}{2}$-mesic.

  The major index of a permutation of $[n]$ is $\frac{n(n-1)}{4}$-mesic under Lehmer code rotation. Similarly, the comajor index is $\frac{n(n-1)}{4}$-mesic.
\end{prop}

\begin{proof}
  Theorem \ref{thm:descents_at_i_LC} states that the number of descents at position $i$ is $\frac{1}{2}$-mesic for $1 \leq i< n$. Hence, the number of descents is the sum of $n-1$ homomesic statistics, each with average $\frac{1}{2}$. Hence, there are in average $\frac{n-1}{2}$ descents in $\sigma$ for each orbit of the Lehmer code rotation.

  Ascents are the pairs that are not descents; there are therefore on average $n-1-\frac{n-1}{2} = \frac{n-1}{2}$ ascents over each orbit of the Lehmer code rotation.

  Recall that the major index  of a permutation $\sigma$ is the sum of its descents, that is $\maj(\sigma)=\sum_{i \in \text{Des}(\sigma)} i$. We already observed that, for each position in $\{1, \ldots, n-1\}$ under the Lehmer code rotation, the number of descents at this position is $\frac{1}{2}$-mesic. Therefore, the average value of the major index over each orbit is
  \[ \sum_{i=1}^{n-1} \frac{1}{2} i = \frac{n(n-1)}{4}. \]

  The comajor index of a permutation $\sigma$ is defined as $\sum_{i \in \text{Des}(\sigma)} (n-i) = n \des(\sigma) - \text{maj}(\sigma)$. Since  the number of descents is $\frac{n-1}{2}$-mesic and the major index is $\frac{n(n-1)}{4}$-mesic, the comajor index is homomesic with average value $n \frac{n-1}{2} - \frac{n(n-1)}{4}=\frac{n(n-1)}{4}$.
\end{proof}

\begin{prop}[Statistics 325, 470]\label{325_470_LC}
  The number of runs in a permutation of $[n]$ is $\frac{n+1}{2}$-mesic under Lehmer code rotation. The width of a tree associated to a permutation of $[n]$ is also $\frac{n+1}{2}$-mesic under Lehmer code rotation.
\end{prop}

\begin{proof}
  Since runs are contiguous increasing sequences, runs are separated by descents, and, in each permutation, there is one more run than there are descents. The result follows from the number of descents being $\frac{n-1}{2}$-mesic.
  
  It is shown in \cite{Luschny}, where permutation trees are defined, that the width of the  permutation tree of $\sigma$ is the number of runs of $\sigma$. Therefore, we do not define permutation trees here, as we can prove the homomesy of the statistics by using the homomesy result regarding runs.
\end{proof}

\begin{prop}[Statistics 1114, 1115]\label{1114_1115_LC}
  The number of odd descents of a permutation of $[n]$ is $\frac{1}{2}\lceil \frac{n-1}{2}\rceil$-mesic under Lehmer code rotation, and the number of even descents is $\frac{1}{2}\lfloor\frac{n-1}{2}\rfloor$-mesic.
\end{prop}

\begin{proof}
  The proof also follows from Theorem \ref{thm:descents_at_i_LC}, which says that the average number of descents at $i$, $1\leq i\leq n-1$, is $\frac{1}{2}$ over each orbit. Therefore, the number of odd descents is, on average,
  \[\sum_{i=1}^{\left\lceil\frac{n-1}{2}\right\rceil} \frac{1}{2} = \frac{1}{2} \lceil\frac{n-1}{2}\rceil.\] Similarly, the average number of even descents is \[\sum_{i=1}^{\lfloor\frac{n-1}{2}\rfloor} \frac{1}{2} = \frac{1}{2} \lfloor\frac{n-1}{2}\rfloor.\]
\end{proof}

\begin{prop}[Statistics 23, 353, 365, 366]\label{23_353_365_366_LC}
  The number of double descents and the number of double ascents are each $\frac{n-2}{6}$-mesic under Lehmer code rotation. The number of (inner) valleys and the number of (inner) peaks are each $\frac{n-2}{3}$-mesic.
\end{prop}

\begin{proof}
  The permutation $\sigma$ has a peak at $i$ if $\sigma_i < \sigma_{i+1} > \sigma_{i+2}$. Following Lemma \ref{lem:descents_correspondence_in_Lehmer_code}, this is exactly when $L(\sigma)_i \leq L(\sigma)_{i+1} > L(\sigma)_{i+2}$. If $n-i$ is even, then $n-i-1$, $n-i$ and $n-i+1$ are two-by-two coprime, meaning that all combinations of values for $L(\sigma)_i$, $L(\sigma)_{i+1}$ and $L(\sigma)_{i+2}$ are possible and equally likely over one single orbit (following Lemmas \ref{lem:equioccurrences_of_pairs_Lehmer_code} and \ref{lem:equioccurrences_of_distant_pairs_Lehmer_code}). The argument is therefore the same as for (single) descents. We first prove that peaks are $\frac{n-2}{3}$-mesic.

  For a given value $i \in \{1,\ldots, n-2\}$, there are $n-i$ choices for $L(\sigma)_{i+1}$. Let $k = L(\sigma)_{i+1}$. Then, the proportion of options for the Lehmer code that represent a peak at position $i$ (i.e.\ $i$ is an ascent and $i+1$ is a descent) is $\frac{k(k+1)}{(n-i+1)(n-i-1)}$. This is because the options $\{0,1, \ldots k\}$ are valid entries for $L(\sigma)_i$, out of the $n-i+1$ options, and the valid entries for $L(\sigma)_{i+2}$ are $\{0,1,\ldots, k-1\}$, when there are a total of $n-i-1$ possible entries.  Averaging over all values of $k$, we obtain:
  \[ \frac{1}{n-i} \sum_{k=1}^{n-i-1} \frac{k(k+1)}{(n-i+1)(n-i-1)} = \frac{1}{3}.\]

  If $n-i$ is odd, $L(\sigma)_i$ and $L(\sigma)_{i+2}$ either always have the same parity, or always have distinct parity, over one given orbit, following Lemma  \ref{lem:pairs_in_Lehmer_code_with_same_parities}.  We recall from Lemma \ref{lem:equioccurrences_pairs_distance_2} that, the parity condition aside, the entries of the Lehmer code $L(\sigma)_{i}$ and $L(\sigma)_{i+2}$ are independent.  We first consider what happens when $L(\sigma)_i$ and $L(\sigma)_{i+2}$ have the same parity. We must treat separately the cases of $L(\sigma)_i$ and $L(\sigma)_{i+2}$ being both odd and both even. The probability of having a peak at $i$ when $L(\sigma)_{i+1} = k$ is:
  \[ \frac{\lceil\frac{k}{2}\rceil\lceil\frac{k+1}{2}\rceil}{\frac{n-i-1}{2}\frac{n-i+1}{2}} + \frac{\lfloor\frac{k}{2}\rfloor\lfloor\frac{k+1}{2}\rfloor}{\frac{n-i-1}{2}\frac{n-i+1}{2}},  \]
  where the first summand corresponds to $L(\sigma)_i$ and $L(\sigma)_{i+2}$ being even, and the second to both of them being odd. The above is always equal to $\frac{k(k+1)}{(n-i+1)(n-i-1)}$, so the proportion of peaks at position $i$ is also $\frac{1}{3}$, which is the same as the proportion of peaks when $n-i$ is even.

  Then, when $L(\sigma)_i$ and $L(\sigma)_{i+2}$ have different parities, the probability  of having a peak at $i$ when $L(\sigma)_{i+1}=k$ is:
  \[ \frac{\lfloor\frac{k}{2}\rfloor\lceil\frac{k+1}{2}\rceil}{\frac{n-i-1}{2}\frac{n-i+1}{2}} + \frac{\lceil\frac{k}{2}\rceil\lfloor\frac{k+1}{2}\rfloor}{\frac{n-i-1}{2}\frac{n-i+1}{2}} = \frac{k(k+1)}{(n-i+1)(n-i-1)}. \]
  Therefore, by what is above, the probability of having a peak at position $i$ is always $\frac{1}{3}$, and the number of peaks is $\frac{n-2}{3}$-mesic.

  An ascent is a double ascent if it is not a peak nor $n-1$. Knowing that peaks are $\frac{n-2}{3}$-mesic and ascents that are not $n-1$ are $\frac{n-2}{2}$-mesic, double ascents are $\frac{n-2}{2}-\frac{n-2}{3}=\frac{n-2}{6}$-mesic (using Lemma \ref{lem:sum_diff_homomesies} and Theorem \ref{thm:descents_at_i_LC}).

  The proof for valleys is the same as the proof for peaks, and the proof for double descents follows from the one for double ascents.
\end{proof}

\begin{prop}[Statistic 638]\label{638_LC}
  The number of up-down runs is $\frac{4n+1}{6}$-mesic under Lehmer code rotation.
\end{prop}

\begin{proof}
  Up-down runs are defined as maximal monotone contiguous subsequences, or the first entry alone if it is a descent. Since peaks and valleys mark the end of a monotone contiguous subsequence (except for the last one that ends at the end of a permutation), the statistic is counted by one more than the number of peaks and valleys (added together), plus one if $1$ is a descent. Since peaks, valleys, and descents at each position are homomesic statistics, by Lemma~\ref{lem:sum_diff_homomesies} their sum is a homomesic statistic.
\end{proof}

\begin{definition}
  We define a few variants of peaks and valleys:
  \begin{enumerate}
    \item A \textbf{left outer peak} is either a peak, or 1 if it is a descent. Similarly, a \textbf{right outer peak} is either a peak, or $n-1$ if it is an ascent. An \textbf{outer peak} is either a left or a right outer peak.
    \item A \textbf{valley of a permutation, including the boundary,} is either a valley, 1 if it is an ascent, or $n-1$ if it is a descent.
  \end{enumerate}
\end{definition}

\begin{prop}[Statistics 35, 92, 99, 483, 834]\label{35_92_99_483_834_LC}
  The following variants of valleys and peaks are homomesic under the Lehmer code rotation:
  \begin{itemize}
    \item The number of left outer peaks of a permutation is $\frac{2n-1}{6}$-mesic;
    \item the number of outer peaks of a permutation is $\frac{n+1}{3}$-mesic;
    \item the number of valleys of a permutation, including the boundary, is $\frac{n+1}{3}$-mesic;
    \item the number of times a permutation switches from increasing to decreasing or decreasing to increasing is $\frac{2n-4}{3}$-mesic;
    \item the number of right outer peaks of a permutation is $\frac{2n-1}{6}$-mesic.
  \end{itemize}
\end{prop}

\begin{proof}
  Recall that the number of ascents (respectively, descents) at each position is $\frac{1}{2}$-mesic. Then, we can express these statistics as sum of homomesic statistics under the Lehmer code rotation. We know that sums of homomesic statistics are also homomesic, following Lemma \ref{lem:sum_diff_homomesies}.
  \begin{itemize}
    \item The number of left outer peaks of a permutation is the number of peaks, plus $1$ if there is a descent at position $1$; it is therefore homomesic with an average of $\frac{n-2}{3} + \frac{1}{2} = \frac{2n-1}{6}$.
    \item The number of outer peaks of a permutation is the number of left outer peaks, plus one if there is an ascent at position $n-1$; it is homomesic with an average of $\frac{2n-1}{6} + \frac{1}{2} = \frac{n+1}{3}$.
    \item Since the number of valleys of a permutation, including the boundary, is the sum of the number of valleys, plus $1$ if there is an ascent at position $1$ and plus $1$ if there is a descent at position $n-1$, the statistic is homomesic with an average of $\frac{n-2}{3}+2\frac{1}{2} = \frac{n+1}{3}$.
    \item The number of times a permutation switches from increasing to decreasing or decreasing to increasing is the sum of the number of valleys and peaks, so it is $\frac{2n-4}{3}$-mesic.
    \item For the same reasons as the number of left out peaks, the  number of right outer peaks is $\frac{2n-1}{6}$-mesic.
  \end{itemize}
\end{proof}

We can also define variants of descents.
\begin{definition}
  A \textbf{descent of distance $2$} is an index $i$ such that $\sigma_i>\sigma_{i+2}$. If $i \in  \{1, \ldots, n-2\}$ is not a descent of distance $2$, it is an ascent of distance $2$.
\end{definition}

\begin{lem}\label{lem:desents_distance_2_in_Lehmer_code}
  The permutation $\sigma$ has a descent of distance $2$ at $i$ if and only if $L(\sigma)_i > L(\sigma)_{i+2}+1$ or $L(\sigma)_{i} = L(\sigma)_{i+2}+1$ and $L(\sigma)_{i}\leq L(\sigma)_{i+1}$.
\end{lem}

\begin{proof}
  First, note that the condition $L(\sigma)_{i}\leq L(\sigma)_{i+1}$ is equivalent to $i$ being an ascent by Lemma \ref{lem:descents_correspondence_in_Lehmer_code}.
  We prove three things:
  \begin{enumerate}
    \item If $L(\sigma)_i > L(\sigma)_{i+2}+1$, then $\sigma$ has a descent of distance $2$ at $i$ (we prove the contrapositive).
    \item If $L(\sigma)_{i} = L(\sigma)_{i+2}+1$ and $i$ is an ascent of $\sigma$, then $\sigma$ has a descent of distance $2$ at $i$ (using contradiction).
    \item If $\sigma$ has a descent of distance $2$ at $i$, then either $L(\sigma)_i > L(\sigma)_{i+2}+1$, or $L(\sigma)_{i} = L(\sigma)_{i+2}+1$ and $L(\sigma)_{i}\leq L(\sigma)_{i+1}$.\\
  \end{enumerate}

  \begin{enumerate}
    \item The contrapositive of the statement we want to prove is that if $\sigma$ has an ascent of distance $2$ at $i$, then $L(\sigma)_i \leq L(\sigma)_{i+2}+1$. We prove this below. The function $\delta_{(i,j) \textit{ is an inversion}}$ takes value $1$, if $(i,j)$ is an inversion, or $0$ otherwise.

          If $\sigma_i < \sigma_{i+2}$, then $\{j > i+2 \mid \sigma_j < \sigma_i \} \subseteq \{j > i+2 \mid \sigma_j < \sigma_{i+2} \}.$
          Also,
          \begin{align*}
            L(\sigma)_i & = \#\{j > i+2 \mid \sigma_j < \sigma_i \} + \delta_{(i,i+1) \textit{ is an inversion}} + \underbrace{\delta_{(i,i+2) \textit{ is an inversion}}}_0 \\
                        & \leq \underbrace{\#\{j > i+2 \mid \sigma_j < \sigma_{i+2} \}}_{L(\sigma)_{i+2}} +\ \delta_{(i,i+1) \textit{ is an inversion}}                      \\
                        & \leq L(\sigma)_{i+2}+1.
          \end{align*}

    \item Let $L(\sigma)_{i} = L(\sigma)_{i+2}+1$ and let $i$ be an ascent of $\sigma$. Assume, for now, that $\sigma_i < \sigma_{i+2}$ (this will lead to a contradiction).
          Then,
          \begin{align*}
            L(\sigma)_i & = \#\{j > i+2 \mid \sigma_j < \sigma_i \} + \underbrace{\delta_{(i,i+1) \textit{ is an inversion}}}_{0} + \underbrace{\delta_{(i,i+2) \textit{ is an inversion}}}_0 \\
                        & \leq \#\{j > i+2 \mid \sigma_j < \sigma_{i+2} \}                                                                                                                    \\
                        & = L(\sigma)_{i+2},
          \end{align*}
          which contradicts the hypothesis that $L(\sigma)_i = L(\sigma)_{i+2}+1$. Therefore, the assumption that $\sigma_i < \sigma_{i+2}$ is false, and $i$ is a descent of distance $2$ of $\sigma$.

    \item If $\sigma$ has a descent of distance $2$ at $i$, then $\sigma_i > \sigma_{i+2}$. Then,
          \[ \{j > i+2 \mid \sigma_j < \sigma_i \} \supseteq \{ j > i+2 \mid \sigma_j < \sigma_{i+2} \}. \]
          Going back to the Lehmer code:
          \begin{align*}
            L(\sigma)_i & = \delta_{(i,i+1) \textit{ is an inversion}}+\underbrace{\delta_{(i,i+2) \textit{ is an inversion}}}_{1} + \#\{ j > i+2 \mid \sigma_j < \sigma_{i} \} \\
                        & \geq \delta_{(i,i+1) \textit{ is an inversion}}+ 1 + \underbrace{\#\{ j > i+2 \mid \sigma_j < \sigma_{i+2}\}}_{L(\sigma)_{i+2}}                       \\
                        & \geq L(\sigma)_{i+2} + 1.
          \end{align*}
          Moreover, if $i$ is not an ascent, then $\delta_{(i,i+1) \textit{ is an inversion}}=1$ and $L(\sigma)_{i} \geq L(\sigma)_{i+2} + 2$.

          Hence, either $L(\sigma)_{i} > L(\sigma)_{i+2} + 1$, or $L(\sigma)_{i} = L(\sigma)_{i+2} + 1$ and $i$ is an ascent.
  \end{enumerate}
\end{proof}

\begin{prop}[Statistics 495, 836, 837]\label{495_836_837_LC}
  The number of descents (respectively, ascents) of distance $2$ are $\frac{n-2}{2}$-mesic under the Lehmer code rotation. The inversions of distance at most $2$ are $\frac{2n-3}{2}$-mesic under the Lehmer code rotation.
\end{prop}

\begin{proof}
  We first prove that, over the course of one orbit, exactly half of the permutations have a descent of distance $2$ at position $i \leq n-2$, which proves these homomesy results for ascents and descents.

  Following Lemma \ref{lem:desents_distance_2_in_Lehmer_code}, we need to count the frequency of $L(\sigma)_i > L(\sigma)_{i+2}+1$ and of $L(\sigma)_i = L(\sigma)_{i+2} +1$ with $L(\sigma)_i \leq L(\sigma)_{i+1}$, over the course of a given orbit. There are three cases, according to Lemma \ref{lem:pairs_in_Lehmer_code_with_same_parities}:
  \begin{enumerate}
    \item $n-i$ is even;
    \item $n-i$ is odd, and, over that orbit, the parity of $L(\sigma)_i$ and $L(\sigma)_{i+2}$ are the same;
    \item $n-i$ is odd, and, over that orbit, the parity of $L(\sigma)_i$ and $L(\sigma)_{i+2}$ are distinct.
  \end{enumerate}

  For the following proofs, let $k = L(\sigma)_i$.
  \begin{enumerate}
    \item When $n-i$ is even. In this case, any triplet $(L(\sigma)_i, L(\sigma)_{i+1}, L(\sigma)_{i+2})$ occurs equally often over the course of one orbit of $\L$, following Lemmas \ref{lem:equioccurrences_of_pairs_Lehmer_code} and \ref{lem:equioccurrences_of_distant_pairs_Lehmer_code}. Then, over one orbit, the probability of having $L(\sigma)_i > L(\sigma)_{i+2}+1$ is
          \[ \frac{1}{n-i+1} \sum_{k=2}^{n-i}\frac{k-1}{n-i-1} = \frac{n-i}{2(n-i+1)}. \]
          The probability of having both  $L(\sigma)_i = L(\sigma)_{i+2} + 1$  and $ L(\sigma)_i \leq L(\sigma)_{i+1}$ is
          \[ \frac{1}{n-i+1}\sum_{k=1}^{n-i-1} \frac{1}{n-i-1}\frac{n-i-k}{n-i}  = \frac{1}{2(n-i+1)}.\]
          The two events being disjoint, we can sum the probabilities of getting them, and we obtain that the probability of getting a descent of distance $2$ at $i$ is $\frac{n-i+1}{2(n-i+1)}=\frac{1}{2}$.

          The next two cases are when $n-i$ is odd, which means that the entries $L(\sigma)_i$ and $L(\sigma)_{i+2}$ are not independent, as exhibited by Lemma \ref{lem:pairs_in_Lehmer_code_with_same_parities}.
    \item When $n-i$ is odd, and, over that orbit, the parity of $L(\sigma)_i$ and $L(\sigma)_{i+2}$ are the same. In this case, one cannot have $L(\sigma)_i = L(\sigma)_{i+2} +1$, and we only need to count how often $L(\sigma)_i > L(\sigma)_{i+2}+1$.  Since the entries of the Lehmer code $L(\sigma)_{i}$ and $L(\sigma)_{i+2}$ are independent except for the parity constraint (as shown in Lemma \ref{lem:equioccurrences_pairs_distance_2}), $L(\sigma)_i > L(\sigma)_{i+2}+1$ occurs with probability
          \begin{align*}
            \frac{1}{n-i+1} \bigg( \sum_{k=2\text{, $k$ even}}^{n-i-1} & \frac{\frac{k}{2}}{\frac{n-i-1}{2}} +  \sum_{k=3\text{, $k$ odd}}^{n-i} \frac{\frac{k-1}{2}}{\frac{n-i-1}{2}} \bigg)                             \\
                                                                       & =  \frac{1}{(n-i+1)(n-i-1)}\left( \sum_{k=2\text{, $k$ even}}^{n-i-1} k +  \sum_{k=3\text{, $k$ odd}}^{n-i} k-1 \right)                          \\
                                                                       & =  \frac{1}{(n-i+1)(n-i-1)}\left( \sum_{m=1\ (m=\frac{k}{2})}^{\frac{n-i-1}{2}} 2m +  \sum_{m=1\ (m=\frac{k-1}{2})}^{\frac{n-i-1}{2}} 2m \right) \\
                                                                       & = \frac{(n-i-1)(n-i+1)}{2(n-i+1)(n-i-1)}
            = \frac{1}{2}.                                                                                                                                                                                                \\
          \end{align*}
    \item When $n-i$ is odd, and, over that orbit, the parity of $L(\sigma)_i$ and $L(\sigma)_{i+2}$ are distinct. We first count how often $L(\sigma)_i > L(\sigma)_{i+2}+1$. This occurs with probability
          \begin{align*}
            \frac{1}{n-i+1} \bigg( \sum_{k=4\text{, $k$ even}}^{n-i-1} & \frac{\frac{k-2}{2}}{\frac{n-i-1}{2}} +  \sum_{k=3\text{, $k$ odd}}^{n-i} \frac{\frac{k-1}{2}}{\frac{n-i-1}{2}} \bigg)                              \\
                                                                       & =  \frac{1}{(n-i+1)(n-i-1)}\left( \sum_{k=4\text{, $k$ even}}^{n-i-1} k-2 +  \sum_{k=3\text{, $k$ odd}}^{n-i} k-1 \right)                           \\
                                                                       & =  \frac{1}{(n-i+1)(n-i-1)}\left( \sum_{m=1\ (m=\frac{k-2}{2}) }^{\frac{n-i-3}{2}} 2m +  \sum_{m=1\ (m=\frac{k-1}{2})}^{\frac{n-i-1}{2}} 2m \right) \\
                                                                       & = \frac{n-i-3+n-i+1}{4(n-i+1)} = \frac{n-i-1}{2(n-i+1)}.                                                                                            \\
          \end{align*}

          We then count the frequency of triplets $(L(\sigma)_{i}, L(\sigma)_{i+1}, L(\sigma)_{i+2})$ such that $L(\sigma)_{i}= L(\sigma)_{i+2}+1$ and $L(\sigma)_{i}\leq L(\sigma)_{i+1}$. We get
          \begin{align*}
            \frac{1}{n-i+1}\bigg( \sum_{k=1,\ k\text{ odd}}^{n-i-2} \frac{1}{\frac{n-i-1}{2}}\frac{n-i-k}{n-i} +  \sum_{k=2,\ k\text{ even}}^{n-i-1} \frac{1}{\frac{n-i-1}{2}}\frac{n-i-k}{n-i} \bigg) = \frac{1}{n-i+1}.
          \end{align*}
          Hence, summing the probability of the two events, we get $\frac{n-i-1+2}{2(n-i+1)} = \frac{1}{2}$.
  \end{enumerate}
  This completes the proof that a descent of distance $2$ occurs at position $i$ ($1\leq i \leq n-2$) in half of the permutations of any given orbit, hence proving that descents of distance $2$ are $\frac{n-2}{2}$-mesic.

  As for inversions of distance at most $2$, they are exactly descents and descents of distance $2$. Therefore, their number is the sum of the number of descents and descents of distance $2$.
\end{proof}

It is worth noting that, unlike inversions of distance at most $2$, inversions of distance at most $3$ (statistic 494 in FindStat) are not homomesic under the Lehmer code rotation. We found a counter-example when $n=6$, where the average over one orbit can be  $\frac{119}{20}$, $6$, or $\frac{121}{20}$.

\subsection{Statistics related to permutation patterns}
\label{subsec:lehmer_pp}
In this subsection, we state and prove propositions giving the homomesies of Theorem~\ref{thm:LC} related to permutation patterns, and ask if one can characterize what patterns are homomesic under the Lehmer code rotation (open problem~\ref{prob:pp_LRC}). Recall that permutation patterns and consecutive patterns were defined in Definitions \ref{def:patterns} and \ref{def:consecutive_patterns}.

\begin{prop}[Statistics 355 to 360]\label{355_to_360_LC}
  The number of occurrences of the pattern $ab-c$, with $\{a,b,c\}=\{1,2,3\}$ is $\frac{(n-1)(n-2)}{12}$-mesic.
\end{prop}

\begin{proof}
  The proof is analogous for all six cases. We do the proof here for $13-2$.
  In the Lehmer code, this corresponds to two adjacent entries: the first needs to be at most as large as the second (because it represents an ascent), and the only inversion in the pattern is its first entry at the second of the two adjacent entries. Therefore, the number of occurrences in the Lehmer code is:
  \[ \sum_{i=2}^n \max(L(\sigma)_i-L(\sigma)_{i-1}, 0). \]
  Since adjacent entries in the Lehmer code are independent, all possibilities appear equally often over each orbit (see Lemma \ref{lem:equioccurrences_of_pairs_Lehmer_code}). Therefore, over one orbit, the average is given by the following (where $j$ is the value of $L(\sigma)_i$, and $k$ is the value of $L(\sigma)_{i-1}$):
  \begin{align*}
    \sum_{i=2}^{n}\Bigg(\frac{1}{n+1-i} & \sum_{j=0}^{n-i} \left(\frac{1}{n+2-i}\sum_{k=0}^j j-k\right)\Bigg)                    \\
                                        & = \sum_{i=2}^{n}\left(\frac{1}{(n+1-i)(n+2-i)}\sum_{j=0}^{n-i} \frac{j(j+1)}{2}\right) \\
                                        & = \frac{(n-2)(n-1)}{12}.
  \end{align*}
\end{proof}

\begin{prop}[Statistics 423, 435, 437]\label{423_435_437_LC}
  The number of total occurrences of the patterns $123$ and $132$ (respectively $213$ and $231$, or $312$ and $321$)
  in a permutation is $\Big(\frac{1}{3}\binom{n}{3}\Big)$-mesic.
\end{prop}

\begin{proof}
  The number of occurrences of the pattern 312 or of the pattern 321 in a permutation $\sigma$ is counted by the number of pairs of inversions that have the same first position. Therefore, in the Lehmer code, this is given by
  \[ \sum_{i=1}^n \binom{L(\sigma)_i}{2}. \]
  Given that every possible number appears equally often in the Lehmer code over each orbit (Lemma \ref{lem:equioccurrences_Lehmer_code}), the average over one orbit is
  \[ \sum_{i=1}^n \frac{1}{n-i+1} \sum_{j=0}^{n-i} \binom{j}{2}  = \sum_{i=1}^{n} \frac{(n-i)(n-i-1)}{6} = \frac{1}{3}\binom{n}{3}.\]
  Similarly, for the patterns 132 and 123, these corresponds to two noninversions starting at the same position. Given a permutation $\sigma$, this is given by
  \[  \sum_{i=1}^n \binom{n-i-L(\sigma)_i}{2}. \]
  Averaging over one orbit, this is
  \[ \sum_{i=1}^n \frac{1}{n-i+1} \sum_{j=0}^{n-i} \binom{n-i-j}{2}  = \sum_{i=1}^{n}  \sum_{k=0}^{n-i} \binom{k}{2} = \frac{1}{3}\binom{n}{3},\]
  where the first equality is obtained by the change of variable $k= n-i-j$.\\

  Finally, to obtain the number of occurrences of the pattern $213$ or of the pattern $231$, we consider all occurrences of patterns of length 3  in a permutation, and we subtract the number of occurrences of the other patterns of length 3: $123$, $132$, $312$ and $321$. Therefore, it is a difference of homomesic statistics, so by Lemma~\ref{lem:sum_diff_homomesies} it is also homomesic.
\end{proof}

\begin{remark}
  Note that classical patterns of length at least $3$ (Statistics 2, 119, 217, 218, 219, 220) are not homomesic under the Lehmer code rotation, nor are the sum of other pairs of classical patterns (statistics 424 to 434, as well as 436).
\end{remark}

\begin{prop}[Statistic 709]\label{709_LC}
  The number of occurrences of the vincular patterns $14-2-3$ or $14-3-2$ is $\Big(\frac{1}{12}\binom{n-1}{3}\Big)$-mesic.
\end{prop}

\begin{proof}
  These two patterns correspond to two consecutive positions $i-1$ and $i$, and two positions $\ell, m > i$ for which $\sigma_{i-1} < \sigma_\ell, \sigma_m < \sigma_i$. In the Lehmer code, $\ell$ and $m$ are counted as inversions at position $i$ that are not inversions at position $i-1$. Therefore, the number of such patterns in the permutation $\sigma$ is \[\sum_{i=2}^n \binom{L(\sigma)_i - L(\sigma)_{i-1}}{2}.\] In average, over one orbit (here, again, using Lemma \ref{lem:equioccurrences_Lehmer_code}), we have (where $j$ is the value of $L(\sigma)_i$, and $k$ is the value of $L(\sigma)_{i-1}$):
  \begin{align*}
    \sum_{i=2}^n \left(\frac{1}{n+1-i}\sum_{j=0}^{n-i}\left(\frac{1}{n+2-i}\sum_{k=0}^j \binom{j-k}{2}\right)\right) & =\sum_{i=2}^n \left(\frac{1}{n+1-i}\sum_{j=0}^{n-i}\frac{1}{n+2-i}\frac{j^3-j }{6}\right) \\
                                                                                                                     & =\sum_{i=2}^n \frac{(n-i)((n-i)(n-i+1)-2)}{24(n+2-i)}                                     \\
                                                                                                                     & = \sum_{i=2}^n \frac{(n-i)(n-i-1)}{24}                                                    \\
                                                                                                                     & = \frac{1}{12}\binom{n-1}{3}.
  \end{align*}
\end{proof}

\begin{definition}
  The \textbf{vincular pattern $|1-23$} is the number of occurrences of the pattern $123$, where the first two matched entries are the first two entries of the permutation.
\end{definition}

\begin{prop}[Statistics 1084]\label{1084_LC}
  The number of occurrences of the vincular pattern $|1-23$ in a permutation is $\frac{n-2}{6}$-mesic under the Lehmer code rotation. \end{prop}

\begin{proof}
  The condition that the first two entries must be the first two entries of the permutations means that we only need to consider $L(\sigma)_1$ and $L(\sigma)_2$ to count the number of occurrences of this pattern.

  More specifically, we need $L(\sigma)_1 \leq L(\sigma)_2$ (so that they form an ascent), and then we multiply it by the number of non-inversions starting at position $2$ (which is $n-2-L(\sigma)_2$). Since all combinations of the first two entries of the Lehmer code appear equally often over each orbit of the Lehmer code rotation (see Lemma \ref{lem:equioccurrences_of_pairs_Lehmer_code}), the average number of occurrences of the pattern over one orbit is (with $j = L(\sigma)_2$)
  \[ \frac{1}{n-1}\sum_{j=0}^{n-2} \frac{j+1}{n} (n-2-j) = \frac{n-2}{6}, \]
  where $\frac{j+1}{n}$ is the likelihood of the first entry in the Lehmer code being at most $L(\sigma)_2 = j$.
\end{proof}

Despite the evidence that the number of occurrences of many permutation patterns are homomesic for the Lehmer code rotation, we have found permutation patterns listed in FindStat that are not homomesic, including patterns as simple as $123$ (i.e.\ increasing subsequences of length $3$). This suggests the following problem:
\begin{prob}\label{prob:pp_LRC}
Characterize the permutation patterns that are homomesic for the Lehmer code rotation.
\end{prob}

\subsection{Miscellaneous statistics}
\label{subsec:lehmer_misc}
A few statistics not directly related to descents, inversions or permutation patterns are also homomesic for the Lehmer code rotation. They appear in this subsection.

\begin{definition}
  A \textbf{left-to-right maximum} in a permutation is the maximum of the entries seen so far in the permutation when we read from left to right: this is $\sigma_i$ such that $\sigma_j < \sigma_i$ for all $j<i$. Similarly, a \textbf{left-to-right minimum} is an entry that is the smallest to be read so far: this is $\sigma_i$ such that $\sigma_j > \sigma_i$ for all $j<i$. We define a \textbf{right-to-left maximum (resp. minimum)} analogously: this is $\sigma_i$ such that $\sigma_j < \sigma_i$ (resp. $\sigma_j > \sigma_i$) for all $j>i$.
\end{definition}

\begin{prop}[Statistics 7, 991]\label{7_991_LC}
  The number of right-to-left maxima and the number of right-to-left minima are each $H_n$-mesic, where $H_n =  \sum_{i=1}^n \frac{1}{i}$ is the $n$-th harmonic number.
\end{prop}

\begin{proof}
  Right-to-left minima are represented with zeros in the Lehmer code, since there is no inversion starting at that position.  The average number of zeros at position $i$ is $\frac{1}{n+1-i}$, following Lemma \ref{lem:equioccurrences_Lehmer_code}. Therefore, the average number of right-to-left minima is $\sum_{i=1}^n \frac{1}{n+1-i} = \sum_{k=1}^n \frac{1}{k} = H_n$.

  Similarly, a right-to-left maximum at entry $i$ corresponds to entry $n-i$ in the Lehmer code, which means that $(i,j)$ is an inversion for all $j > i$. Hence, the average number of entries $n-i$ at position $i$ is also $\frac{1}{n+1-i}$. We therefore obtain the same result as for left-to-right minima.
\end{proof}

Note that the number of left-to-right minima (and maxima) are not homomesic. Counter-examples for the number of left-to-right minima can be found at $n=6$, where the orbit average ranges from $\frac{71}{30}$ to $\frac{5}{2}$. Note that, unlike right-to-left extrema, left-to-right extrema do not correspond to a specific value of given entries in the Lehmer code.

\begin{definition}
  The \textbf{rank} of a permutation of $[n]$ is its position among the $n!$ permutations, ordered lexicographically. This is an integer between $1$ and $n!$.
\end{definition}

Before we prove homomesy for the rank under the Lehmer code rotation, we give a lemma describing the connection between the rank and the Lehmer code. This seems to be a known fact, but we could not find a proof in the literature.
\begin{lem}\label{lem:rank_and_lehmer_code}
  For a permutation $\sigma$ of $[n]$, the rank of $\sigma$ is given directly by the Lehmer code $L(\sigma)$ as:
  \begin{equation}
    \rank(\sigma)=1+\sum_{i=1}^{n-1}L(\sigma)_i(n-i)!.\label{eqn:rank}
  \end{equation}
\end{lem}

\begin{proof}
  We prove this lemma by induction on $n$. The base case is when $n=1$: the only permutation has rank $1$, which satisfies Equation \eqref{eqn:rank}. Assuming Equation \eqref{eqn:rank} holds for permutations of $[n]$, we prove it works for permutations of $[n+1]$ in the following way.

  The key is to notice that the first entry of the permutation gives a range for the rank. The rank of  a permutation $\sigma$ of $[n+1]$ is between $(\sigma_1-1)n!+1$ and $\sigma_1n!$. More specifically, it is given by $(\sigma_1-1)n!+\rank(\sigma_2\ldots\sigma_{n+1})$. Using the induction hypothesis,
  \[ \rank(\sigma) = (\sigma_1-1)n!+\rank(\sigma_2\ldots\sigma_{n+1}) = L(\sigma)_1 n! + 1 + \sum_{i=1}^{n-1}L(\sigma)_{i+1}(n+1-(i+1))! = 1 + \sum_{i=1}^{n}L(\sigma)_{i}(n+1-i)!, \]
  which proves Equation \eqref{eqn:rank}.
\end{proof}

We now have the tools to prove Proposition \ref{20_LC}.
\begin{prop}[Statistics 20]\label{20_LC}
  The rank of the permutation is $\frac{n!+1}{2}$-mesic under the Lehmer code rotation.
\end{prop}
\begin{proof}
  We use Lemma \ref{lem:rank_and_lehmer_code} to compute the rank directly from the Lehmer code.

  Let $m$ be the orbit size under the Lehmer code rotation. By Theorem \ref{Thm: L-Orbit cardinality}, $m=\lcm(1,2,\ldots,n).$  Acting on $\sigma$ by the Lehmer code rotation we get $L(\sigma)+\textbf{1}$ where addition in the $i$-th component is done modulo $n-i+1$.  Thus if we act on $\sigma$ by the Lehmer code rotation $k$ times, the resulting Lehmer code has rank $$\rank(\L^k(\sigma))=1+\sum_{i=1}^{n-1}[L(\sigma)_i+k]_{n-i+1}(n-i)!.$$  Calculating the average over an orbit of the Lehmer code rotation we find
  $$\frac{1}{m}\sum_{k=1}^{m} \rank(\L^k(\sigma))=1+\frac{1}{m}\sum_{i=1}^{n-1}(n-i)!\sum_{k=1}^m\left[L(\sigma)_i+k\right]_{n-i+1}.$$
  Using the fact that $n-i+1$ is a divisor of $m$, it follows that $\sum_{k=1}^m\left[L(\sigma)_i+k\right]_{n-i+1}$ is the sum of the equivalence classes modulo $n-i+1$ repeated $\frac{m}{n-i+1}$ times.  Thus,
  \begin{eqnarray*}
    \frac{1}{m}\sum_{k=1}^{m} \rank(\L^k(\sigma))&=&1+\frac{1}{m}\sum_{i=1}^{n-1}(n-i)!\frac{m}{n-i+1}\sum_{j=0}^{n-i}j\\
    &=&1+\frac{1}{2}\sum_{i=1}^{n-1}(n-i)!(n-i)=\frac{n!+1}{2}.
  \end{eqnarray*}
  This shows that the rank is $\frac{n!+1}{2}$-mesic for the Lehmer code rotation.
\end{proof}

\begin{definition}\label{BabsonSteingr\'imsson}
  Eric Babson and Einar Steingr\'imsson defined a few statistics in terms of occurrences of permutation patterns, including the statistics that they name \textbf{stat} and \textbf{stat$'$} \cite[Proposition 9]{BabsonSteingrimsson}. The statistic stat is the sum of the number of occurrences of the consecutive permutation patterns $13-2$, $21-3$, $32-1$ and $21$, while stat$'$ is the sum of the number of occurrences of $13-2$, $31-2$. $32-1$ and $21$.
\end{definition}

\begin{prop}[Statistics 692, 796]\label{692_796_LC}
  The Babson--Steingr\'imsson statistics stat and stat$'$ are $\frac{n(n-1)}{4}$-mesic.
\end{prop}

\begin{proof}
  We showed in Proposition \ref{355_to_360_LC} that the consecutive patterns of the form $ab-c$ for $\{a,b,c\} = \{1,2,3\}$ are $\frac{(n-1)(n-2)}{12}$-mesic, and we also showed that the number of descents is $\frac{n-1}{2}$-mesic (Proposition \ref{4_21_245_833_LC}).

  Following Lemma \ref{lem:sum_diff_homomesies}, the sum of homomesic statistics is homomesic, and the average of both stat and stat$'$ over one orbit of the Lehmer code is $3\frac{(n-1)(n-2)}{12}+\frac{n-1}{2} = \frac{n(n-1)}{4}.$
\end{proof}

\begin{prop}[Statistics 1377, 1379]\label{1377_1379_LC}
  The major index minus the number of inversions of a permutation is $0$-mesic. The number of inversions plus the major index of a permutation is $\frac{n(n-1)}{2}$-mesic.
\end{prop}
\begin{proof}
  Recall from Lemma \ref{lem:sum_diff_homomesies} that linear combinations of homomesic statistics are homomesic. Both the major index and the number of inversions are $\frac{n(n-1)}{4}$-mesic. Therefore, their difference is $0$-mesic and their sum is $\frac{n(n-1)}{2}$-mesic.
\end{proof}

\begin{definition}
  An \textbf{ascent top} is a position $i$ for which $\sigma_{i-1} < \sigma_{i}$. In other words, $i$ is an ascent top exactly when $i-1$ is an ascent.
\end{definition}

\begin{prop}[Statistic 1640]\label{1640_LC}
  The number of ascent tops in the permutation such that all smaller elements appear before is $\big(1-\frac{1}{n}\big)$-mesic under the Lehmer code rotation.
\end{prop}

\begin{proof}
  Given an index $i$, if all smaller elements appear before position $i$, $L(\sigma)_i = 0$. Using the proof of Proposition \ref{4_21_245_833_LC}, that means that $L(\sigma)_{i-1} \leq L(\sigma)_i = 0$. Therefore, this happens whenever we have two consecutive zero entries in the Lehmer code.

  We then use the fact that all possible choices for adjacent entries of the Lehmer code occur with the same frequency in any given orbit (Lemma \ref{lem:equioccurrences_of_pairs_Lehmer_code}). Hence, the average number of occurrences of $i$ being an ascent top in the permutation such that all smaller elements appear before, in each orbit, is $\frac{1}{n+1-i}\frac{1}{n+2-i}$. Hence, the total number is
  \[ \sum_{i=2}^{n} \frac{1}{n+1-i}\frac{1}{n+2-i} = \sum_{j=1}^{n-1}\frac{1}{j}\frac{1}{j+1} = \sum_{j=1}^{n-1}\left(\frac{1}{j}-\frac{1}{j+1}\right) = 1-\frac{1}{n}. \]
\end{proof}

This concludes the proof of Theorem \ref{thm:LC}, showing that the 45 statistics listed are homomesic under the Lehmer code rotation.

\section{Complement and Reverse Maps}
\label{sec:comp_rev}
In this section, we prove homomesies for the reverse and complement maps. Because these maps behave similarly, there are many statistics that exhibit the same behavior on the orbits of both maps. For that reason, we have divided this section into four parts. Subsection~\ref{sec:comp} discusses the differences and similarities of the two maps and includes lemmas that will be helpful in our later proofs. In Subsection~\ref{sec:both}, we prove homomesies for both the complement and the reverse map. In Subsection~\ref{sec:complement}, we prove homomesies for the complement map, and provide examples to show that they are not homomesic for the reverse. In Subsection~\ref{sec:rev}, we prove homomesies for the reverse map, and provide examples to show that they are not homomesic for the complement.

Many of the statistics that are homomesic under both maps are proven using one of two methods. The first method is to count all possibilities of the statistic and then divide by two, as either it will occur in the permutation or its reverse (or complement). The other method is to use the relationship between the reverse and complement maps seen in Lemma \ref{lem:C&R_relation}. However, there are a few statistics that use different proof techniques. While these statistics are not themselves of more interest than our other results, the proofs are noteworthy for being distinct. For the reverse, these are the \hyperref[prop:R_446]{disorder} of a permutation and the \hyperref[prop:R_304]{load} of a permutation. And for the complement, these are the \hyperref[prop:C_1114_1115]{number of odd descents}, the \hyperref[prop:C_1114_1115]{number of even descents}, and \hyperref[prop:C_692]{the Babson and Steingr\'imsson statistic stat}.

First, we introduce the maps and main theorems.

\begin{definition}
  If $\sigma = \sigma_1 \ldots \sigma_n$, then the \textbf{reverse} of $\sigma$ is $\R(\sigma) = \sigma_n\ldots \sigma_1$. That is, $\R(\sigma)_i=\sigma_{n+1-i}$.
\end{definition}

\begin{definition}
  If $\sigma = \sigma_1 \ldots \sigma_n$, then the \textbf{complement} of $\sigma$ is $\C(\sigma) = (n+1-\sigma_1)\ldots(n+1-\sigma_n)$. That is, $\C(\sigma)_i =n+1 - \sigma_i$.
\end{definition}

\begin{remark}
It is useful to note that when viewing the reverse or complement as actions on permutation matrices, they are seen as horizontal and vertical reflections respectively.
\end{remark}

\begin{example}\label{revcompex}
  Let $\sigma = 52134$. Then $\R(\sigma) = 43125$ and $\C(\sigma) = 14532$.
\end{example}

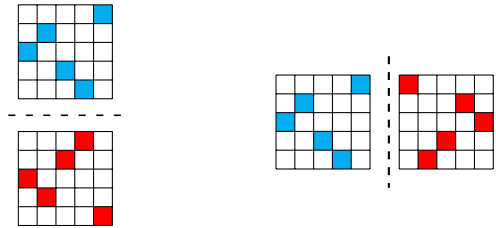
\begin{figure}[h!]

\centering
\begin{minipage}[c]{0.25\linewidth}
\begin{center}
	\begin{tabular}{c}
		$\begin{tikzpicture}[fill=cyan, scale=0.25,
		baseline={([yshift=-.5ex]current bounding box.center)},
		cell30/.style={fill}, cell21/.style={fill}, cell02/.style={fill}, 
		cell13/.style={fill},
		cell44/.style={fill},
		]
		\foreach \i in {0,...,4}
		\foreach \j in {0,...,4}
		\path[cell\i\j/.try] (\i,\j) rectangle +(1,1);
		\draw grid (5,5);
		\end{tikzpicture}$\\
		 - - - - - - - \\	
		
		$\begin{tikzpicture}[fill=red, scale=0.25,
		baseline={([yshift=-.5ex]current bounding box.center)},
		cell02/.style={fill}, cell11/.style={fill}, cell23/.style={fill}, 
		cell40/.style={fill},
		cell34/.style={fill},
		]
		\foreach \i in {0,...,4}
		\foreach \j in {0,...,4}
		\path[cell\i\j/.try] (\i,\j) rectangle +(1,1);
		\draw grid (5,5);
		\end{tikzpicture}$
	\end{tabular}
\end{center}
  \end{minipage}
  \begin{minipage}[c]{0.25\linewidth}
	\[
	\begin{tikzpicture}[fill=cyan, scale=0.25,
		baseline={([yshift=-.5ex]current bounding box.center)},
		cell30/.style={fill}, cell21/.style={fill}, cell02/.style={fill}, 
		cell13/.style={fill},
		cell44/.style={fill},
		]
		\foreach \i in {0,...,4}
		\foreach \j in {0,...,4}
		\path[cell\i\j/.try] (\i,\j) rectangle +(1,1);
		\draw grid (5,5);
		\draw[thick,dashed] (6,6) -- (6,-1);
	\end{tikzpicture} \ 
	\begin{tikzpicture}[fill=red, scale=0.25,
		baseline={([yshift=-.5ex]current bounding box.center)},
		cell04/.style={fill}, cell33/.style={fill}, cell42/.style={fill}, 
		cell21/.style={fill},
		cell10/.style={fill}, 
		]
		\foreach \i in {0,...,4}
		\foreach \j in {0,...,4}
		\path[cell\i\j/.try] (\i,\j) rectangle +(1,1);
		\draw grid (5,5);
	\end{tikzpicture}\]
  \end{minipage}\hfill
  \caption{The Reverse and the Complement}
\end{figure}

While the inverse map shares many similarities with the reverse and the complement, it is interesting to note that it does not exhibit homomesy on any of the statistics found in FindStat. We conjecture that this is due to the number of fixed points under the inverse map. For each permutation $\sigma$ fixed under a map, the value of the statistic evaluated at $\sigma$ has to be the global average of the statistic. Thus, each fixed point of a map adds another constraint on a statistic being homomesic under that map.
\begin{example}
Let $\sigma = 52134$. Then the inverse of $\sigma$ is $\mathcal{I}(\sigma) = 32451.$ 
\begin{figure}[h!]
\begin{center}
	\begin{tabular}{c}
$\begin{tikzpicture}[fill=cyan, scale=0.25,
		baseline={([yshift=-.5ex]current bounding box.center)},
		cell30/.style={fill}, cell21/.style={fill}, cell02/.style={fill}, 
		cell13/.style={fill},
		cell44/.style={fill},
		]
		\foreach \i in {0,...,4}
		\foreach \j in {0,...,4}
		\path[cell\i\j/.try] (\i,\j) rectangle +(1,1);
		\draw grid (5,5);
		\draw[thick,dashed] (6,6) -- (12,-1);
	\end{tikzpicture} \ 
	\begin{tikzpicture}[fill=red, scale=0.25,
		baseline={([yshift=-.5ex]current bounding box.center)},
		cell24/.style={fill}, cell13/.style={fill}, cell32/.style={fill}, 
		cell41/.style={fill},
		cell00/.style={fill}, 
		]
		\foreach \i in {0,...,4}
		\foreach \j in {0,...,4}
		\path[cell\i\j/.try] (\i,\j) rectangle +(1,1);
		\draw grid (5,5);
	\end{tikzpicture}$
  	\end{tabular}
\end{center}
  \caption{The Inverse}
\end{figure}
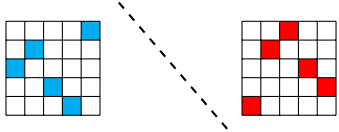\label{fig:inv}
\end{example}

\begin{remark}
One could manufacture statistics where the inverse map does exhibit homomesy. Sergi Elizalde suggested two such examples. The number of exceedances + $\frac{1}{2}$ the number of fixed points and the number of deficiencies + $\frac{1}{2}$ the number of fixed points are both $\frac{n}{2}$-mesic (see Definition \ref{def:exceedance} for the definition of exceedances and deficiencies). We see this as the number of exceedances equals the number of filled boxes above the main diagonal, the number of deficiencies equals the number of filled boxes below the main diagonal, the number of fixed points equals the number of filled boxes in the main diagonal, and the inverse acts on a permutation matrix by reflecting it along the main diagonal.
\end{remark}

The main theorems of this section are as follows.

\begin{thm}\label{thmboth}
  The reverse map and the complement map are both homomesic under the following statistics:
  \begin{itemize} \rm
    \item Statistics related to inversions:
          \begin{itemize}
            \item \hyperref[prop:RC_18_246]{$\Stat~18$}: The number of inversions of a permutation $(${\small average: $\frac{n(n-1)}{4}$ }$)$
            \item \hyperref[prop:RC_55_341]{$\Stat~55$}: The inversion sum of a permutation $(${\small average: $\frac{1}{2}\binom{n+1}{3}$ }$)$
            \item \hyperref[prop:RC_18_246]{$\Stat$ $246$}: The number of non-inversions of a permutation $(${\small average: $\frac{n(n-1)}{4}$ }$)$
            \item \hyperref[prop:RC_55_341]{$\Stat$ $341$}: The non-inversion sum of a permutation $(${\small average: $\frac{1}{2}\binom{n+1}{3}$ }$)$
            \item \hyperref[prop:RC_495]{$\Stat$ $494$}: The number of inversions of distance at most $3$ of a permutation $(${\small average: $\frac{3n-6}{2}$ }$)$
            \item \hyperref[prop:RC_495]{$\Stat$ $495$}: The number of inversions of distance at most $2$ of a permutation $(${\small average: $\frac{2n-3}{2}$ }$)$
            \item \hyperref[prop:RC_538_539]{$\Stat$ $538$}: The number of even inversions of a permutation $(${\small average: $\frac{1}{2}\cdot \lfloor \frac{n}{2}\rfloor\lfloor\frac{n-1}{2}\rfloor$ }$)$
            \item \hyperref[prop:RC_538_539]{$\Stat$ $539$}: The number of odd inversions of a permutation $(${\small average: $\frac{1}{2}\lfloor\frac{n^2}{4}\rfloor$ }$)$
            \item \hyperref[prop:RC_677]{$\Stat$ $677$}: The standardized bi-alternating inversion number of a permutation $(${\small average: $\frac{\lfloor\frac{n}{2}\rfloor^2}{2}$ }$)$
          \end{itemize}
    \item Statistics related to descents:
          \begin{itemize}
            \item \hyperref[prop:RC_21_245_et]{$\Stat$ $21$}: The number of descents of a permutation $(${\small average: $\frac{n-1}{2}$ }$)$
            \item \hyperref[prop:RC_21_245_et]{$\Stat$ $245$}: The number of ascents of a permutation $(${\small average: $\frac{n-1}{2}$ }$)$
            \item \hyperref[prop:RC_21_245_et]{$\Stat$ $470$}: The number of runs in a permutation $(${\small average: $\frac{n+1}{2}$ }$)$
            \item \hyperref[prop:RC_21_245_et]{$\Stat$ $619$}: The number of cyclic descents of a permutation $(${\small average: $\frac{n}{2}$ }$)$
            \item \hyperref[prop:RC_836_837]{$\Stat$ $836$}: The number of descents of distance $2$ of a permutation $(${\small average: $\frac{n-2}{2}$ }$)$
            \item \hyperref[prop:RC_836_837]{$\Stat$ $837$}: The number of ascents of distance $2$ of a permutation $(${\small average: $\frac{n-2}{2}$ }$)$
            \item \hyperref[prop:RC_836_837]{$\Stat$ $1520$}: The number of strict $3$-descents of a permutation $(${\small average: $\frac{n-3}{2}$ }$)$
          \end{itemize}
    \item Statistics related to other permutation properties:
          \begin{itemize}
            \item \hyperref[prop:RC_21_245_et]{$\Stat$ $325$}: The width of the tree associated to a permutation $(${\small average: $\frac{n+1}{2}$ }$)$
            \item \hyperref[prop:RC_342]{$\Stat$ $342$}: The cosine of a permutation $(${\small average: $\frac{(n+1)^2n}{4}$ }$)$
            \item \hyperref[prop:RC_354]{$\Stat$ $354$}: The number of recoils of a permutation $(${\small average: $\frac{n-1}{2}$ }$)$
            \item \hyperref[prop:RC_457]{$\Stat$ $457$}: The number of occurrences of one of the patterns $132$, $213$, or $321$ in a permutation $(${\small average: $\frac{\binom{n}{3}}{2}$ }$)$
            \item \hyperref[prop:RC_21_245_et]{$\Stat$ $824$}: The sum of the number of descents and the number of recoils of a permutation $(${\small average: $n-1$ }$)$
            \item \hyperref[prop:RC_828]{$\Stat$ $828$}: The Spearman’s rho of a permutation and the identity permutation $(${\small average: $\binom{n+1}{3}$ }$)$
          \end{itemize}
  \end{itemize}
\end{thm}

\begin{thm}\label{onlycomp}
  The complement map is homomesic under the following statistics, but the reverse map is not:
  \begin{itemize} \rm
    \item Statistics related to inversions:
          \begin{itemize}
            \item \hyperref[prop:C_1557_1556]{$\Stat$ $1556$}: The number of inversions of the second entry of a permutation $(${\small average: $\frac{n-2}{2}$ }$)$
            \item \hyperref[prop:C_1557_1556]{$\Stat$ $1557$}: The number of inversions of the third entry of a permutation $(${\small average: $\frac{n-3}{2}$ }$)$
          \end{itemize}
    \item Statistics related to descents:
          \begin{itemize}
            \item \hyperref[prop:C_4]{$\Stat$ $4$}: The major index $(${\small average: $\frac{n(n-1)}{4}$ }$)$
            \item \hyperref[prop:C_1114_1115]{$\Stat$ $1114$}: The number of odd descents of a permutation $(${\small average: $\frac{1}{2}\lceil \frac{n-1}{2}\rceil$ }$)$
            \item \hyperref[prop:C_1114_1115]{$\Stat$ $1115$}: The number of even descents of a permutation $(${\small average: $\frac{1}{2}\lfloor\frac{n-1}{2}\rfloor$ }$)$
          \end{itemize}
    \item Statistics related to other permutation properties:
          \begin{itemize}
            \item \hyperref[prop:C_20]{$\Stat$ $20$}: The rank of a permutation $(${\small average: $\frac{n!+1}{2}$ }$)$
            \item \hyperref[prop:C_54_740]{$\Stat$ $54$}: The first entry of the permutation $(${\small average: $\frac{n+1}{2}$ }$)$
            \item \hyperref[prop:C_662]{$\Stat$ $662$}: The staircase size of a permutation $(${\small average: $\frac{n-1}{2}$ }$)$
            \item \hyperref[prop:C_692]{$\Stat$ $692$}: Babson and Steingr\'imsson’s statistic stat of a permutation $(${\small average: $\frac{n(n-1)}{4}$ }$)$
            \item \hyperref[prop:C_54_740]{$\Stat$ $740$}: The last entry of a permutation $(${\small average: $\frac{n+1}{2}$ }$)$
            \item \hyperref[prop:C_1332]{$\Stat$ $1332$}: The number of steps on the non-negative side of the walk associated with a permutation $(${\small average: $\frac{n-1}{2}$ }$)$
            \item \hyperref[prop:C_1377_1379]{$\Stat$ $1377$}: The major index minus the number of inversions of a permutation $(${\small average: $\frac{n(n-1)}{2}$ }$)$
            \item \hyperref[prop:C_1377_1379]{$\Stat$ $1379$}: The number of inversions plus the major index of a permutation $(${\small average: $\frac{n(n-1)}{2}$ }$)$
            \item \hyperref[thm:ith_entry_comp]{$i$-th entry}: The $i$-th entry of a permutation $(${\small average: $\frac{n+1}{2}$ }$)$
          \end{itemize}
  \end{itemize}
\end{thm}

\begin{thm}\label{onlyrev}
  The reverse map is homomesic under the following statistics, but the complement map is not:
  \begin{itemize} \rm
    \item \hyperref[prop:R_304]{$\Stat$ $304$}: The load of a permutation $(${\small average: $\frac{n(n-1)}{4}$ }$)$
    \item \hyperref[prop:R_305]{$\Stat$ $305$}: The inverse major index $(${\small average: $\frac{n(n-1)}{4}$ }$)$
    \item \hyperref[prop:R_446]{$\Stat$ $446$}: The disorder of a permutation $(${\small average: $\frac{n(n-1)}{4}$ }$)$
    \item \hyperref[prop:R_616]{$\Stat$ $616$}: The inversion index of a permutation $(${\small average: $\binom{n+1}{3}$ }$)$
    \item \hyperref[prop:R_798]{$\Stat$ $798$}: The makl of a permutation $(${\small average: $\frac{n(n-1)}{4}$ }$)$
  \end{itemize}
\end{thm}

\subsection{Comparing and contrasting the reverse and complement}
\label{sec:comp}

Before we provide proofs for our main theorems, we introduce some general lemmas which show how the two maps are similar and how they differ. First, we note that both the complement and reverse maps are involutions, and thus their orbits always have size 2.

One of the main differences between these two maps is illustrated by the following lemma.

\begin{lem}
  Whenever $n>2$ is odd, we note the following:
  \begin{enumerate}
    \item $\R(\sigma)$ has a fixed point: $\sigma_{\frac{n+1}{2}} =\R(\sigma)_{\frac{n+1}{2}}.$
    \item $\C(\sigma)$ has a fixed point: If $\sigma_i=\frac{n+1}{2}$, then $\C(\sigma)_i =\frac{n+1}{2} $.
  \end{enumerate}
  When $n$ is even, $\R(\sigma)$ and $\C(\sigma)$ have no fixed points.
\end{lem}
\begin{proof}
  The proof follows directly from the definitions of the reverse and complement map. Let $n>2$ be an odd integer.
  \begin{enumerate}
    \item Since $\R(\sigma)_i = \sigma_{n + 1 - i}, \R(\sigma)_{\frac{n+1}{2}} = \sigma_{n + 1 - \frac{n+1}{2}} = \sigma_{\frac{n+1}{2}}$.
    \item Since $\C(\sigma)_i = n+1-\sigma_i$, $\sigma_i = \frac{n+1}{2}$ implies that $\C(\sigma)_i = n+1-\frac{n+1}{2} = \frac{n+1}{2}.$
  \end{enumerate}
  Let $n>2$ be an even integer. Then $\frac{n+1}{2}$ is not an integer, and there is no elements $\sigma_{\frac{n+1}{2}}$ or $\sigma_i=\frac{n+1}{2}$ as parts of a permutation $\sigma$.
\end{proof}

\begin{example}
  Continuing Example \ref{revcompex}, let $\sigma = 52134$. Then $\R(\sigma) = 43125$, and we see $\R(\sigma)$ has a fixed point $\sigma_3 = 1 = \R(\sigma)_3$. Additionally, $\C(\sigma) = 14532$, and we see $\C(\sigma)$ has a fixed point $\sigma_4 = 3 = \C(\sigma)_4$.
\end{example}

The following lemma exhibits the relationship between the complement and the reverse maps, which will be used in the proofs of our main results.

\begin{lem} \label{lem:C&R_relation}
  Let $\sigma \in S_n$. Then,
  \begin{enumerate}
    \item $\C(\sigma)^{-1} = \R(\sigma^{-1})$ and $\C(\sigma^{-1}) = \R(\sigma)^{-1}.$
    \item $(\R\circ \C)^2=e$, where $e$ is the identity map on permutations.
    \item $(\R\circ \mathcal{I})^4= e$, where $\mathcal{I}$ is the map that sends $\sigma$ to its inverse, $\sigma^{-1}$.
  \end{enumerate}
\end{lem}

\begin{proof}
Each of these equations becomes clear when the maps are viewed as actions on permutation matrices, as the reverse map is equivalent to a horizontal reflection, the complement map is equivalent to a vertical reflection, and the inverse map is equivalent to a reflection along the main diagonal. 
\end{proof}

Now we are ready to prove our main theorems.

\subsection{Statistics homomesic for both the reverse and the complement}
\label{sec:both}

In this subsection, we prove homomesy of the complement and reverse maps for the statistics listed in Theorem \ref{thmboth}.

First, we consider statistics related to inversions. We will use the following lemma and definition.

\begin{lem}\label{inversion_pairs}
  The permutation $\sigma$ has $(a,b)\in \Inv (\sigma)$ if and only if $(a,b) \notin \Inv(\C(\sigma))$ and $(n+1-b, n+1-a) \notin \Inv(\R(\sigma))$.
\end{lem}

\begin{proof}
  Suppose that $(i, j)$ is a pair such that $1 \leq i < j \leq n$. If $\sigma_i>\sigma_j$, then $\C(\sigma)_i = n+1-\sigma_i<n+1-\sigma_j = \C(\sigma)_j$ and
  $\R(\sigma)_{n+1-i} = \sigma_i > \sigma_j = \R(\sigma)_{n+1-j}$. So $(i, j)$ is an inversion of $\sigma$ if and only if $(i, j)$ is not an inversion of $\C(\sigma)$ and $(n + 1 - j, n + 1 - i)$ is not an inversion of $\R(\sigma)$.
\end{proof}

\begin{definition}
  An inversion, where $i<j$, is said to be an \bb{odd inversion} if $i\neq j \mod 2$. An inversion is said to be an \bb{even inversion} if $i=j\mod 2$.
\end{definition}

\begin{prop}[Statistics 538, 539]\label{prop:RC_538_539}
  The number of even inversions of a permutation are $\Big(\frac{1}{2}\cdot \lfloor \frac{n}{2}\rfloor\lfloor\frac{n-1}{2}\rfloor\Big)$-mesic, and the number of odd inversions of a permutation are $\Big(\frac{1}{2}\lfloor\frac{n^2}{4}\rfloor\Big)$-mesic under the complement and reverse maps.
\end{prop}

\begin{proof}
  Using Lemma \ref{inversion_pairs} if  $(i,j)$ is an inversion of $\sigma$, then $(n + 1 - j, n + 1 - i)$ is not an inversion of $\R(\sigma)$.
  \begin{itemize}
    \item If $(i,j)$ is an odd inversion, then so is $(n + 1 - j, n + 1 - i)$.
    \item If $(i,j)$ is an even inversion, then so is $(n + 1 - j, n + 1 - i)$.
  \end{itemize}
  In either case, each odd or even inversion of $\sigma$ is matched with an odd or even inversion that is not present in $\R(\sigma)$.

  Similarly, if $(i,j)$ is an odd or even inversion of $\sigma$, it is not an inversion pair for $\C(\sigma)$.

  There are $\lfloor \frac{n}{2} \rfloor \cdot \lfloor \frac{n+1}{2} \rfloor = \lfloor\frac{n^2}{4}\rfloor$ ways to choose an odd inversion, and $\lfloor \frac{n}{2}\rfloor\lfloor\frac{n-1}{2}\rfloor$ ways to choose an even inversion.

  Therefore, the number odd inversions of a permutation are $\Big(\frac{1}{2}\cdot \lfloor\frac{n^2}{4}\rfloor\Big)$-mesic, and the number of even inversions of a permutation is $\Big(\frac{1}{2}\cdot \lfloor \frac{n}{2}\rfloor\lfloor\frac{n-1}{2}\rfloor\Big)$-mesic.
\end{proof}

\begin{prop}[Statistics 18, 246]\label{prop:RC_18_246}
  The number of inversions and number of non-inversions of a permutation is $\frac{n(n-1)}{4}$-mesic for both the complement and reverse maps.
\end{prop}
\begin{proof}
  First note that the number of inversions of a permutation is the sum of even and odd inversions of that permutation. Using Lemma \ref{lem:sum_diff_homomesies}, we see that the number of inversions is homomesic for both the complement and reverse.

  Similarly, the number of non-inversions of a permutation is given by $\frac{n(n-1)}{2}-\inv(\sigma)$. Since we have proven that $\inv(\sigma)$ is homomesic, we see that the number of non-inversions is homomesic for both maps as well.

  Between $\sigma$ and $\C(\sigma)$, or $\sigma$ and $\R(\sigma)$, we count all the possible inversion or non-inversion pairs: $\frac{n(n-1)}{2}$. Thus the number of inversions or non-inversions is $\frac{n(n-1)}{4}$-mesic for both maps.
\end{proof}

\begin{definition}
  The \bb{inversion sum} of a permutation is given by $\displaystyle \sum_{(a, b) \in \Inv(\sigma)} (b - a)$.
\end{definition}

\begin{prop}[Statistics 55, 341]\label{prop:RC_55_341}
  The inversion sum of a permutation and the non-inversion sum of a permutation are both $\left(\frac{1}{2}\binom{n+1}{3}\right)$-mesic under the complement and reverse maps.
\end{prop}
\begin{proof}
  Using the result from Lemma \ref{inversion_pairs}, when we add the inversion sum for $\sigma $ with that of $\R(\sigma)$, we have:
  \[
    \sum_{(a,b) \in \Inv (\sigma)}(b-a) + \sum_{(a,b) \notin \Inv (\sigma)}\left((n+1-a)-(n+1-b)\right) = \sum_{1\leq a < b \leq n} (b - a).
  \]
  This is also the result from adding the inversion sum for $\sigma $ with that of $\C(\sigma)$:
  \[
    \sum_{(a,b) \in \Inv (\sigma)}(b-a) + \sum_{(a,b) \notin \Inv (\sigma)} (b-a) = \sum_{1\leq a < b \leq n} (b - a).
  \]
  From here, we find
  \[
    \sum_{1\leq a < b \leq n} (b - a) = \sum_{i=1}^{n-1}i(n-i) = \frac{(n-1)n^2}{2} - \frac{(n-1)n(2n-1)}{6} = \frac{n(n-1)(n+1)}{6} = \binom{n+1}{3}.\]

  Hence, the average is $\frac{1}{2}\binom{n+1}{3}$.
\end{proof}

\begin{definition}
  The \bb{sign of an integer} is given by
  \[
    \mbox{sign}(n) =\begin{cases} ~~1 & \mbox{ if } n >0 \\
      ~~0 & \mbox{ if } n=0  \\
      -1  & \mbox{ if } n <0
    \end{cases}.
  \]
\end{definition}

\begin{definition}\cite{Even_Zohar_2016}
  The \textbf{standardized bi-alternating inversion number} of a permutation $\sigma = \sigma_1 \sigma_2 \ldots \sigma_n$ is defined as
  \[
    \frac{j(\sigma) + \lfloor \frac{n}{2} \rfloor^2}{2}
  \]
  where
  \[
    j(\sigma) = \sum_{1\leq y<x \leq n} (-1)^{x+y}\mbox{sign}(\sigma_x-\sigma_y).
  \]
\end{definition}

\begin{prop}[Statistic 677]\label{prop:RC_677}
  The standardized bi-alternating inversion number of a permutation is $\displaystyle \  \frac{\lfloor\frac{n}{2}\rfloor^2}{2}$-mesic under the complement and the reverse maps.
\end{prop}
\begin{proof}

  For the complement, we have:
  \[
    j(\C(\sigma)) = \sum_{1\leq y<x \leq n} (-1)^{x+y}\mbox{sign}((n+1-\sigma_x)-(n+1-\sigma_y) = \sum_{1\leq y<x \leq n} (-1)^{x+y}\mbox{sign}(-\sigma_x+\sigma_y ).
  \]
  Since $\mbox{sign}(\sigma_x-\sigma_y)$ and $\mbox{sign}(-\sigma_x +\sigma_y )$ are always opposites, while $(-1)^{x+y}$ is always the same (as $(-1)^{y+x}$), the average of $j(\sigma)+j(\C(\sigma))$ is 0.

  For the reverse, we have:
  \[
    j(\R(\sigma)) = \sum_{1\leq y<x \leq n} (-1)^{x+y}\mbox{sign}(\sigma_{n+1-x}-\sigma_{n+1-y})\\
    =\sum_{1\leq y<x \leq n} (-1)^{x+y}\mbox{sign}(\sigma_y-\sigma_x).
  \]
  As with the complement, $j(\R(\sigma)) + j(\sigma) = 0$, and the average over the orbit for both maps is given by $\displaystyle \frac{\lfloor\frac{n}{2}\rfloor^2}{2}$.
\end{proof}

Recall Definition \ref{def:basic_stats} for the definition of an inversion pair.
\begin{definition}
  The number of \bb{recoils} of a permutation $\sigma$ is defined as the number of inversion pairs of $\sigma$ of the form $(i+1, i)$. Alternatively, the number of recoils of a permutation $\sigma$ is the number of descents of $\sigma^{-1}$.
\end{definition}

\begin{prop}[Statistic 354]\label{prop:RC_354}
  The number of recoils of a permutation is $\frac{n-1}{2}$-mesic under the complement and the reverse maps.
\end{prop}

\begin{proof}
  Using Lemma \ref{lem:C&R_relation}, note that $\des(\C(\sigma)^{-1}) + \des(\sigma^{-1}) = \des(\R(\sigma^{-1})) + \des(\sigma^{-1}) = n-1$ and $\des(\sigma^{-1}) + \des(\R(\sigma)^{-1})= \des(\sigma^{-1}) + \des(\C(\sigma^{-1}))=n-1$.

  Therefore, the number of recoils is $\frac{n-1}{2}$-mesic for both maps.
\end{proof}

Before we look at statistics related to descents and ascents, we consider Lemmas \ref{lem:positionofdes_C&R} and \ref{lem:numberofdes_C&R} related to the position and number of descents and ascents for both maps.

\begin{lem}\label{lem:positionofdes_C&R}
  If $\sigma$ has a descent at position $i$, $\C(\sigma)$ has an ascent at position $i$ and $\R(\sigma)$ has an ascent at position $n - i$. And if $\sigma$ has an ascent at position $i$, $\C(\sigma)$ has a descent at position $i$ and $\R(\sigma)$ has a descent at position $n - i$
\end{lem}
\begin{proof}
  Let $\sigma\in S_n$. If $\sigma_i> \sigma_{i+1}$, then $\C(\sigma)_i=n+1-\sigma_i< n+1-\sigma_{i+1}=\C(\sigma)_{i+1}$ and $\R(\sigma)_{n-i} = \sigma_{i+1} < \sigma_i = \R(\sigma)_{n-i+1}$. This means that a descent at position $i$ is mapped to an ascent at position $i$ under the complement and an ascent at position $n - i$ under the reverse. Similarly, we can see an ascent at position $i$ is mapped to a descent at position $i$ under the complement and a descent at position $n-i$ under the reverse.
\end{proof}

\begin{lem}\label{lem:numberofdes_C&R}
  If $\sigma$ has $k$ descents and $n-1-k$ ascents, then $\C(\sigma)$ and $\R(\sigma)$ both have $k$ ascents and $n-1-k$ descents.
\end{lem}
\begin{proof}
  In Lemma \ref{lem:positionofdes_C&R} we showed that every descent in $\sigma$ contributed to an ascent in $\C(\sigma)$ and $\R(\sigma)$, and every ascent in $\sigma$ contributed to a descent in $\C(\sigma)$ and $\R(\sigma)$, so the result follows.
\end{proof}

\begin{prop}\label{prop:RC_21_245_et}
  For both the complement and the reverse maps,
  \begin{itemize}
    \item \textnormal{(Statistics 21, 245)} The number of descents of a permutation, and the number of ascents of a permutation, are $\frac{n-1}{2}$-mesic;
    \item \textnormal{(Statistic 619)} The number of cyclic descents of a permutation is $\frac{n}{2}$-mesic;
    \item \textnormal{(Statistic 470)} The number of runs in a permutation is $\frac{n+1}{2}$-mesic;
    \item \textnormal{(Statistic 325)} The width of the tree associated to a permutation is $\frac{n+1}{2}$-mesic;
    \item \textnormal{(Statistic 824)} The sum of the number of descents and the number of recoils of a permutation is $(n-1)$-mesic.
  \end{itemize}
\end{prop}
\begin{proof}
  There are $n-1$ possible ascents or descents in a permutation $\sigma$. From Lemma \ref{lem:positionofdes_C&R}, we note that between $\sigma$ and $\C(\sigma)$, and $\sigma$ and $\R(\sigma)$, we have all possible ascents and descents. Therefore, the number of ascents and descents of a permutation are $\frac{n-1}{2}$-mesic

  Cyclic descents only differ from standard descents by allowing a descent at position $n$ if $\sigma_n<\sigma_1$. Using the same argument from Lemma \ref{lem:numberofdes_C&R}, if $\sigma_n<\sigma_1$, we have $\C(\sigma)_1<\C(\sigma)_n$ and $\R(\sigma)_1< \R(\sigma)_n$. So between $\sigma$ and $\C(\sigma)$, or $\sigma$ and $\R(\sigma)$, we have all possible cyclic descents. Thus the average is $\frac{n}{2}$.

  The width of the tree associated to a permutation is the same as the number of runs, as stated in  \cite{Luschny} (and in the proof of Proposition \ref{325_470_LC}). The number of runs is the same as the number of descents, plus 1. Since we just showed that the number of descents is $\frac{n-1}{2}$-mesic, this gives an average of $\frac{n+1}{2}$.

  The sum of the number of recoils and the number of descents is also homomesic by Proposition \ref{prop:RC_354} and Lemma \ref{lem:sum_diff_homomesies}. Using these results, this sum is $(n-1)$-mesic for both maps.
\end{proof}

\begin{prop}\label{prop:RC_836_837}
  For both the complement and the reverse maps,
  \begin{itemize}
    \item \textnormal{(Statistics 836, 837)} The number of descents of distance $2$ of a permutation, and the number of ascents of distance $2$ of a permutation is $\frac{n-2}{2}$-mesic.
    \item \textnormal{(Statistic 1520)} The number of strict $3$-descents of a permutation is $\frac{n-3}{2}$-mesic.
  \end{itemize}
\end{prop}

\begin{proof}
  If $\sigma$ has a descent of distance $2$, under $\C(\sigma)$ and $\R(\sigma)$ the descent is mapped to an ascent of distance $2$. We have all possible ascents and descents of distance $2$ either present in $\sigma$ or $\R(\sigma)$, and all possible ascents or descents in either $\sigma$ or $\C(\sigma)$.

  There are a total of $n-2$ possible descents or ascents of distance $2$, so the average is $\frac{n-2}{2}$.

  What is called a strict $3$-descent in FindStat is a descent of distance $3$. The argument for descents of distance $3$ follows similarly. Because there are a total of $n-3$ possible descents of distance $3$, the average is $\frac{n-3}{2}$.
\end{proof}

\begin{prop} \label{prop:RC_495} For both the complement and reverse maps,
  \begin{itemize}
    \item \textnormal{(Statistic 495)} The number of inversions of distance at most $2$ of a permutation is $\frac{2n-3}{2}$-mesic;
    \item \textnormal{(Statistic 494)} The number of inversions of distance at most $3$ of a permutation is $\frac{3n-6}{2}$-mesic.
  \end{itemize}
\end{prop}

\begin{proof}
  First note that the number of inversions of distance at most $i$ is the sum of all descents, descents of distance $2$, up to descents of distance $i$.

  In Propositions \ref{prop:RC_21_245_et} and \ref{prop:RC_836_837}, we showed that the number of descents, and the numbers of descents of distance $2$ and $3$ are homomesic for both the reverse and the complement. By Lemma \ref{lem:sum_diff_homomesies}, the sums of these statistics are also homomesic.

  The average number of inversions of distance at most $2$ is $\frac{(n-1)+(n-2)}{2} = \frac{2n-3}{2}$.

  Similarly, the average number of inversions of distance at most $3$ is $\frac{(n-1)+(n-2)+(n-3)}{2 } = \frac{3n-6}{2}$.
\end{proof}

Lastly, we prove homomesy for statistics related to permutation properties other than inversions or descents and ascents.

\begin{prop}[Statistic 457]\label{prop:RC_457}
  The number of total occurrences of one of the patterns $132$, $213$, or $321$ in a permutation is $\frac{\binom{n}{3}}{2}$-mesic for both the complement and reverse maps.
\end{prop}

\begin{proof}
  These patterns are half of the patterns of length $3$ and, for each of the patterns, the reverse patterns and complement patterns are not included.

  So for a triple of positions $a<b<c$ and a permutation $\sigma$, either $\sigma_a\sigma_b\sigma_c$ or $\R(\sigma)_{n+1-c}\R(\sigma)_{n+1-b}\R(\sigma)_{n+1-a}$ is a pattern in the list. Similarly, for a triple of positions $a<b<c$, either $\sigma_a\sigma_b \sigma_c$ or $\C(\sigma)_a\C(\sigma)_b\C(\sigma)_c$ is a pattern in the list.

  Hence, $\frac{\binom{n}{3}}{2}$ is the average number of occurrences of these patterns in each orbit.
\end{proof}

\begin{definition}\cite{https://doi.org/10.48550/arxiv.1106.1995}
  The \bb{cosine of $\sigma$} is defined as $\cos(\sigma)=\displaystyle \sum_{i = 1}^{n} i\sigma_i$.
\end{definition}

The name of the statistic is due to the following construction, found in \cite{https://doi.org/10.48550/arxiv.1106.1995}: we treat the permutations $\sigma$ and $e$, the identity, as vectors. Then the dot product of the two vectors is calculated as
\[
  e\cdot \sigma =\displaystyle \sum_{i = 1}^{n} i\sigma_i
\]
or, alternatively,
\[
  e\cdot \sigma = |e||\sigma|\cos(\theta)=|\sigma|^2 \cos(\theta) = \frac{n(n+1)(2n+1)}{6} \cos(\theta),
\]
where $\theta$ is the angle between the vectors. Thus the dot product only relies on the cosine of the angle between the vectors, which is where the statistic derives its name.

\begin{prop}[Statistic 342]\label{prop:RC_342}
  The cosine of a permutation is $\frac{(n+1)^2n}{4}$-mesic
\end{prop}

\begin{proof} By definition,
  \[
    \cos(\sigma) + \cos(\C(\sigma)) = (n+1) \sum_{i = 1}^{n} i = (n+1) \frac{(n+1)n}{2},
  \]
  and
  \[
    \cos(\sigma) + \cos(\R(\sigma)) = \sum_{i = 1}^{n} i\sigma_i + \sum_{i = 1}^n (n + 1 - i) \sigma_i = \sum_{i = 1}^n (n+1) \sigma_i = (n+1) \frac{(n+1)n}{2}.
  \]
  Therefore, the average is
  \[
    \frac{(n+1)^2n}{4}.
  \]
\end{proof}

In statistics, Spearman's rho is used as a test to determine the relationship between two variables. In the study of permutations, it can be used as a measure for a distance between $\sigma$ and the identity permutation.

\begin{definition}\cite{ChatterjeeDiaconis}
  The \textbf{Spearman's rho of a permutation and the identity permutation} is given by $\displaystyle \sum_{i = 1}^n (\sigma_i - i)^2$.
\end{definition}

\begin{prop}[Statistic 828]\label{prop:RC_828}
  The Spearman's rho of a permutation and the identity permutation is $\binom{n+1}{3}$-mesic for both the complement and the reverse maps.
\end{prop}

\begin{proof}
  Under the reverse map, the average is calculated by

  \begin{align*}
     & \frac{1}{2}\sum_{i = 1}^n \left((\sigma_i - i)^2 + (\sigma_{n + 1 - i} - i)^2\right)                          \\
     & = \frac{1}{2}\sum_{i = 1}^n \left((\sigma_i - i)^2 + (\sigma_{i} - (n + 1 - i))^2\right)                      \\
     & = \frac{1}{2}\sum_{i = 1}^n \left((\sigma_i - i)^2 + (\sigma_i + i)^2 - 2(n+1)(\sigma_i + i) + (n+1)^2\right) \\
     & = \frac{1}{2}\sum_{i = 1}^n \left(2\sigma_i^2 + 2i^2- 2(n+1)(\sigma_i + i) + (n+1)^2\right)                   \\
     & = \frac{1}{2}\sum_{i = 1}^n \left(4i^2- 4(n+1)(i) + (n+1)^2\right)                                            \\
     & = 2\sum_{i = 1}^n i^2 - 2(n+1)\sum_{i = 1}^n i + \frac{n(n+1)^2}{2}                                           \\
     & = \frac{n(n+1)(2n+1)}{3} - n(n+1)^2 + \frac{n(n+1)^2}{2}                                                      \\
     & = \frac{n(n+1)(2n+1)}{3} - \frac{n(n+1)^2}{2}                                                                 \\
     & = \frac{n(n+1)(n-1)}{6}                                                                                       \\
     & = \binom{n+1}{3}.
  \end{align*}

  For the complement, we find that
  \[
    \sum_{i = 1}^n (\C(\sigma)_i - i)^2 = \sum_{i = 1}^n (n+1-\sigma_i - i)^2 = \sum_{i = 1}^n (\sigma_i - (n-i+1))^2.
  \]

  So the average is
  \[
    \frac{1}{2} \left( \sum_{i = 1}^n (\sigma_i - i)^2 + \sum_{i = 1}^n (\sigma_i - (n-i+1))^2\right),
  \]
  which means that the average for both the reverse and the complement is given by $\binom{n+1}{3}$.
\end{proof}

This concludes the proof of Theorem \ref{thmboth}, showing the statistics listed there exhibit homomesy for both the reverse and the complement maps.

\subsection{Statistics homomesic for the complement but not the reverse}
\label{sec:complement}

In this subsection, we prove that the statistics listed in Theorem \ref{onlycomp} are homomesic under the complement map and provide examples illustrating that they are not homomesic under the reverse map. Note that it is enough to provide an example of an orbit whose average under the statistic does not match that of the global average, following Remark~\ref{global_avg}.

\begin{prop}[Statistic 4]\label{prop:C_4}
  The major index is $\frac{n(n-1)}{4}$-mesic for the complement, but is not homomesic for the reverse.
\end{prop}

\begin{proof}
  For $\sigma$, recall from Definition \ref{def:basic_stats} that
  \[
    \maj (\sigma) = \sum_{\sigma_i>\sigma_{i+1}} i.
  \]
  For the complement, we know from Lemma \ref{lem:positionofdes_C&R} that $\C(\sigma)_i<\C(\sigma)_{i+1}$ whenever $i$ is a descent for $\sigma$, and  $\C(\sigma)_j>\C(\sigma)_{j+1}$ whenever $j$ is an ascent for $\sigma$. Thus
  \[
    \maj (\sigma) +\maj (\C(\sigma)) = \sum_{i=1}^{n-1} i = \frac{n(n-1)}{2}.
  \]
  Thus the average over an orbit is $\frac{n(n-1)}{4}$.

  To see that it is not homomesic under the reverse, we exhibit an orbit with an average that differs from the global average of $\frac{n(n-1)}{4}$. Consider the orbit $(\sigma, \R(\sigma))$ where $\sigma = 132$ and $\R(\sigma) = 231$. The major index of $\sigma$ is 2 and the major index of $\R(\sigma)$ is 2, so the average over the orbit is $2$, not $\frac{3(3-1)}{4} = \frac{3}{2}.$
\end{proof}

\begin{cor}[Statistics 1377, 1379]\label{prop:C_1377_1379}
  The major index minus the number of inversions of a permutation is $0$-mesic, and the number of inversions plus the major index of a permutation is $\frac{n(n-1)}{2}$-mesic under the complement, but  is not homomesic for the reverse.
\end{cor}

\begin{proof}
  Both of these are combinations of the major index and the number of inversions, which are both homomesic under the complement. The major index is $\frac{n(n-1)}{4}$-mesic, as is the number of inversions. This means that their difference is 0-mesic, and their sum is $\frac{n(n-1)}{2}$-mesic.

  The reverse map is homomesic under the number of inversions but is not under the major index, so it cannot be homomesic under the major index minus the number of inversions of a permutation, or under the number of inversions plus the major index of a permutation.
\end{proof}

\begin{thm}\label{thm:inversion_positions_RC}
  The number of inversions of the $i$-th entry of a permutation is $\frac{n-i}{2}$-mesic under the complement.
\end{thm}

\begin{proof}
  In general, if $\sigma_i>\sigma_j$ when $i<j$, we have $\C(\sigma)_i<\C(\sigma)_j$.
  There are $n-i$ possible inversions for the $i$-th entry. Each of those $n-i$ inversions is present in either $\sigma$ or $\C(\sigma)$. Thus we have an average of $\frac{n-i}{2}$ over one orbit.
\end{proof}

\begin{cor}[Statistics 1557, 1556]\label{prop:C_1557_1556}
  The number of inversions of the second entry of a permutation is $\frac{n-2}{2}$-mesic, and the number of inversions of the third entry of a permutation is $\frac{n-3}{2}$-mesic under the complement, but is not homomesic for  the reverse.
\end{cor}

\begin{proof}
  From Theorem \ref{thm:inversion_positions_RC}, we have the desired homomesies for the complement.

  To see that the number of inversions of the second entry is not homomesic under the reverse, we exhibit an orbit with an average that differs from the global average of $\frac{n-2}{2}$. Consider the orbit $(\sigma, \R(\sigma))$ where $\sigma = 132$ and $\R(\sigma) = 231$. The number of inversions of the second entry of $\sigma$ is 1 and the number of inversions of the second entry of $\R(\sigma)$ is 1, so the average over the orbit is 1, not $\frac{3-2}{2} = \frac{1}{2}$.

  To see that the number of inversions of the third entry is not homomesic under the reverse, we exhibit an orbit with an average that differs from the global average of $\frac{n-3}{2}$. Consider the orbit $(\sigma, \R(\sigma))$ where $\sigma = 1243$ and $\R(\sigma) = 3421$. The number of inversions of the third entry of $\sigma$ is 1 and the number of inversions of the third entry of $\R(\sigma)$ is 1, so the average over the orbit is 1, not $\frac{4-3}{2} = \frac{1}{2}$.
\end{proof}

\begin{prop}[Statistics 1114, 1115]\label{prop:C_1114_1115}
  The number of odd descents of a permutation and the number of even descents of a permutation are both homomesic under the complement, but not the reverse. The average number of odd descents over one orbit is $\frac{1}{2}\lceil \frac{n-1}{2}\rceil$, and the average number of even descents is $\frac{1}{2}\lfloor\frac{n-1}{2}\rfloor$.
\end{prop}

\begin{proof}
  An odd descent in $\sigma$ is an odd ascent in $\C(\sigma)$ and vice versa. The sum of odd descents in $\sigma$ and $\C(\sigma)$ is the number of possible odd descents in a string $n$, so that the average is $\frac{1}{2}\lceil \frac{n-1}{2}\rceil$.

  To see that it is not homomesic under the reverse, we exhibit an orbit with an average that differs from the global average of $\frac{1}{2}\lceil \frac{n-1}{2}\rceil$. Consider the orbit $(\sigma, \R(\sigma))$ where $\sigma = 132$ and $\R(\sigma) = 231$. Then the number of odd descents of $\sigma$ is 0 and the number of odd descents of $\R(\sigma)$ is 0, so the average over the orbit is $0$, not $\frac{1}{2}\lceil \frac{3-1}{2}\rceil = \frac{1}{2}$.

  We can think of even descents as the complement of the set of all odd descents. The average is $\frac{1}{2}\lfloor\frac{n-1}{2}\rfloor$. As the number of descents is homomesic for the reverse map, but the number of odd descents is not, the number of even descents is not homomesic.

\end{proof}

\begin{prop}\label{thm:ith_entry_comp}
  The $i$-th entry of the permutation is $\frac{n+1}{2}$-mesic under the complement.
\end{prop}

\begin{proof}
  Since $\sigma_i + \C(\sigma)_i = n+1$, the average of the $i$-th entry is $\frac{n+1}{2}$.
\end{proof}

\begin{cor}[Statistics 54, 740]\label{prop:C_54_740}
  The first entry of the permutation and the last entry of a permutation is $\frac{n+1}{2}$-mesic under the complement, but not the reverse.
\end{cor}

\begin{proof}
  From Proposition \ref{thm:ith_entry_comp}, we have the desired homomesies under the complement.

  To see that it is not homomesic under the reverse, we exhibit an orbit with an average that differs from the global average of $\frac{n+1}{2}$. Consider the orbit $( \sigma, \R(\sigma))$ where $\sigma = 132$ and $\R(\sigma) = 231$. The average over the orbit for the first entry (and last entry) is $\frac{3}{2}$, not $\frac{3+1}{2} = 2$.
\end{proof}

\begin{prop}[Statistic 20] \label{prop:C_20}
  The rank of a permutation is $\frac{n!+1}{2}$-mesic under the complement, but not the reverse.
\end{prop}

\begin{proof}
  From Lemma \ref{lem:rank_and_lehmer_code}, we know that the rank of a permutation can be found by
  \[
    \rank(\sigma) = 1 + \sum_{i=1}^{n-1} L(\sigma)_i (n-i)!
  \]
  Since $\C(\sigma)_i = n+1-\sigma_i$, the definition of the Lehmer code implies that the sum of  $i$-th entries of the Lehmer codes of $\sigma$ and its complement is $L(\sigma)_i+L(\C(\sigma))_i = n-i$. This allows us to find the following.
  \begin{align*}
    \rank (\sigma) + \rank (\C(\sigma)) & =2+ \sum_{i=1}^{n-1} L(\sigma)_i (n-i)!  + \sum_{i=1}^{n-1} L(\C(\sigma))_i (n-i)! \\
                                        & = 2+ \sum_{i=1}^{n-1}(n-i)! \left ( L(\sigma)_i +  L(\C(\sigma))_i\right)          \\
                                        & =2+ \sum_{i=1}^{n-1}(n-i)! (n-i).
  \end{align*}
  This means that, as seen in Proposition \ref{20_LC}, the average is
  \[
    1+\frac{1}{2} \sum_{i=1}^{n-1}(n-i)! (n-i) = \frac{n!+1}{2}.
  \]

  To see that it is not homomesic under the reverse, we exhibit an orbit with an average that differs from the global average of $\frac{n!+1}{2}$. Consider the orbit $( \sigma, \R(\sigma))$ where $\sigma = 132$ and $\R(\sigma) = 231$. The rank of $\sigma$ is 2 and the rank of $\R(\sigma)$ is 4, so the average over the orbit for the rank is $3$, not $\frac{3!+1}{2} = \frac{7}{2}$.
\end{proof}

The next proposition examines Babson and Steingr\'imsson's statistic stat, which is defined in Definition \ref{BabsonSteingr\'imsson}.

\begin{prop}[Statistic 692]\label{prop:C_692}
  Babson and Steingr\'imsson's statistic stat of a permutation is $\frac{n(n-1)}{4}$-mesic under the complement, but not the reverse.
\end{prop}

\begin{proof}
  In terms of generalized patterns, this statistic is given by the sum of the number of occurrences of each of the patterns $13-2$, $21-3$, $32-1$ and $21$.
  Numbers in the pattern which are not separated by a dash must appear consecutively, as explained in Definition \ref{def:consecutive_patterns}.

  The patterns $13-2$, $21-3$, $32-1$ and $21$ have complement, respectively, $31-2$, $23-1$, $12-3$ and $12$. The sum of the number of occurrences of each of the patterns $13-2$, $21-3$, $32-1$, $21$, $31-2$, $23-1$, $12-3$ and $12$ is the total number of pairs of adjacent entries, plus the number of triples made of two adjacent entries, plus another entry to their right. This is also $\text{stat}(\sigma)+\text{stat}(\C(\sigma))$. Hence, over one orbit, the statistics has average
  \[
    \frac{1}{2}\left(\text{stat}(\sigma)+\text{stat}(\C(\sigma))\right) = \frac{1}{2}\left( (n-1) + \sum_{i=1}^{n-2} (n-1-i) \right) = \frac{n(n-1)}{4}.
  \]
  The statistic is hence $\frac{n(n-1)}{4}$-mesic.

  To see that it is not homomesic under the reverse, we exhibit an orbit with an average that differs from the global average of $\frac{n(n-1)}{4}$. Consider the orbit $(\sigma, \R(\sigma))$ where $\sigma = 2314$ and $\R(\sigma) = 4132$. The statistic on $\sigma$ is 2 and the statistic on $\R(\sigma)$ is 3, so the average over the orbit is $\frac{5}{2}$, not $\frac{4(4-1)}{4} = 3$.
\end{proof}

\begin{prop}[Statistic 1332]\label{prop:C_1332}
  The number of steps on the non-negative side of the walk associated with a permutation is $\frac{n-1}{2}$-mesic under the complement, but not the reverse.
\end{prop}

\begin{proof}
  Consider the walk taking an up step for each ascent, and a down step for each descent of the permutation. Then this statistic is the number of steps that begin and end at non-negative height. The complement of a permutation flips the path upside down. Since the path takes $n-1$ steps, the statistic is $\frac{n-1}{2}$-mesic.

  To see that it is not homomesic under the reverse, we exhibit an orbit with an average that differs from the global average of $\frac{n-1}{2}$. Consider the orbit $( \sigma, \R(\sigma))$ where $\sigma = 132$ and $\R(\sigma) = 231$. The statistic on $\sigma$ is 2 and the statistic on $\R(\sigma)$ is 2, so the average over the orbit is $2$, not $\frac{3-1}{2} = 1$.
\end{proof}

\begin{definition}
  The \textbf{staircase size} of a permutation $\sigma$ is the largest index $k$ for which there exist indices \mbox{$i_k < i_{k-1}< \ldots < i_1$} with  $L(\sigma)_{i_j}\geq j$, where $L(\sigma)_i$ is the $i$-th entry of the Lehmer code of $\sigma$ (see Section \ref{sec:lehmer} for a definition of the Lehmer code).
\end{definition}

\begin{example}\label{ex:staircase_size}
  In this example, we give the Lehmer code of a permutation and its complement, in which we highlight (in bold) the entries of the Lehmer code at the positions that form the staircases. Notice that the set of highlighted positions in $L(\C(\sigma))$ is the complement (in $[n-1]$) of the set of highlighted positions in $L(\sigma)$. This fact will be used in the proof of the next proposition.
  \begin{align*}
    \sigma = 15286347,\qquad L(\sigma) =       & (0, \mathbf{3}, 0, \mathbf{4}, \mathbf{2}, 0, 0, 0), \quad \text{staircase size is 3}           \\
    \C(\sigma) = 84713652,\quad L(\C(\sigma)) = & (\mathbf{7}, 3, \mathbf{5}, 0, 1, \mathbf{2}, \mathbf{1}, 0), \quad \text{staircase size is 4}.
  \end{align*}
\end{example}

\begin{prop}[Statistic 662]\label{prop:C_662}
  The staircase size of the code of a permutation is $\frac{n-1}{2}$-mesic under the complement, but not the reverse.
\end{prop}

\begin{proof}
  Following Lemma \ref{inversion_pairs}, we know that $(i,j)$ is an inversion of $\sigma$ exactly when it is a noninversion of $\C(\sigma)$. Therefore, $L(\C(\sigma))_i = n-i-L(\sigma)_i$.

  To show homomesy, it is enough to prove the following claim:

  \underline{Claim}: There exists a subset $I = \{i_k<\ldots<i_1\}\subseteq [n-1]$ such that $L(\sigma)_{i_j} >j$ for all $j \in [k]$, and the numbers $[n-1]-I$ form a sequence $l_{n-1-k} < \ldots < l_1$ with $L(\C(\sigma))_{l_j} > j$ for all $j\in [n-1-k]$. In other terms, we partition the numbers $[n-1]$ into two sets that form the ``staircases'' of $\sigma$ and $\C(\sigma)$.

  An example of a partition into two staircases appear in Example \ref{ex:staircase_size}.

  We prove the claim by induction on $n$, the number of items in the permutations.

  The base case is when $n=2$. There is only one orbit, formed of $12$, with Lehmer code $(0,0)$, and $21$, that has Lehmer code $(\mathbf{1},0)$. The former has staircase size $0$, whereas the latter's staircase is $\{1\}$.

  For the induction step, we let $\sigma$ be a permutation of $[n+1]$ with Lehmer code $L(\sigma)$, and we define $\sigma'$ to be the unique permutation of $n$ elements with Lehmer code $(L(\sigma)_2, L(\sigma)_3, \ldots, L(\sigma)_{n+1})$. Let $I$ be the staircase of $\sigma'$, with size $|I|=k$. We know that $\C(\sigma')$ has Lehmer code $(n-1-L(\sigma)_1, n-2-L(\sigma)_2, \ldots, n-n-L(\sigma)_n)$, which correspond to the whole Lehmer code of $\C(\sigma)$ except the first entry. By induction hypothesis, there is a set $I \subseteq [n-1]$ of size $k$ that correspond to the staircase of $\sigma'$, and $[n-1]-I$ correspond to the staircase of $\C(\sigma')$. Denote $I\oplus 1 = \{i+1 \mid i \in I\}$.

  For $\sigma$, there are two cases:\begin{itemize}
    \sloppy
    \item If $L(\sigma)_1 \geq k+1$, then the staircase of $\sigma$ is $(I\oplus 1) \cup \{1\}$, and has size $k+1$. Then \mbox{$L(\C(\sigma))_1 = (n+1) - 1 - L(\sigma)_1 \leq n-k-1$}, and $\C(\sigma)$ has staircase, $I\oplus 1$, of size $n-k-1$. The union of the staircases of $\sigma$ and $\C(\sigma)$ is $[n]$.
    \item If $L(\sigma)_1 \leq k$, then the staircase of $\sigma$ is $I\oplus 1$, and has size $k$, but $L(\C(\sigma))_1 = (n+1) - 1 - L(\sigma)_1 \geq n-k$, so the staircase of $\C(\sigma)$ is $\{1\}\cup (I\oplus 1)$, of size $n-k$, and the staircases' union over the orbit is again $[n]$. \end{itemize}
  This concludes the proof of the claim, which means that the staircase size of a permutation is $\frac{n-1}{2}$-mesic under the complement.

  To see that it is not homomesic under the reverse, we exhibit an orbit with an average that differs from the global average of $\frac{n-1}{2}$. Consider the orbit $( \sigma, \R(\sigma))$ where $\sigma = 21453$ and $\R(\sigma) = 35412$. The statistic on $\sigma$ is 1 and the statistic on $\R(\sigma)$ is 2, so the average over the orbit is $\frac{3}{2}$, not $\frac{5-1}{2} = 2$.
\end{proof}

This concludes the proofs of Theorem~\ref{onlycomp}, showing the statistics listed there are homomesic for the complement map but not for the reverse map.

\subsection{Statistics homomesic for the reverse but not the complement}
\label{sec:rev}
In this subsection, we prove the statistics listed in Theorem \ref{onlyrev} are homomesic under the reverse map and provide examples illustrating that they are not homomesic under the complement map. Note that it is enough to provide an example of an orbit whose average under the statistic does not match that of the global average, following Remark \ref{global_avg}.

We begin with the inversion index, which is defined based on inversion pairs (see Definition \ref{def:basic_stats} for the definition of an inversion pair).

\begin{definition} The \textbf{inversion index} of a permutation $\sigma$ is given by summing all $\sigma_i$ where $(\sigma_i, \sigma_j)$ is an inversion pair for $\sigma$.
\end{definition}

\begin{prop}[Statistic 616]\label{prop:R_616}
  The inversion index is $\binom{n+1}{3}$-mesic for the reverse, but not homomesic for the complement.
\end{prop}

\begin{proof} Since any pair $(\sigma_i, \sigma_j)$ with $\sigma_i > \sigma_j$ is either an inversion pair for $\sigma$ or $\R(\sigma)$, the inversion index of $\sigma$ added to the inversion index of $\R(\sigma)$ is $n(n - 1) + (n-1)(n - 2) + \ldots + 2(1)$, and the average over the orbit is the sum divided by 2, which is $\binom{n+1}{3}$.

  To see that it is not homomesic under the complement, we exhibit an orbit with an average that differs from the global average of $\binom{n+1}{3}$. Consider the orbit $( \sigma, \C(\sigma))$ where $\sigma = 132$ and $\C(\sigma) = 312$. The inversion index for $\sigma$ is 3 and the inverse index for $\C(\sigma)$ is 6, so the average over the orbit is $\frac{9}{2}$, not $\binom{3+1}{3} = 4$.\end{proof}

The disorder of a permutation is defined by Emeric~Deutsch in the comments of the OEIS page for sequence A008302~\cite{OEIS}.

\begin{definition}
  Given a permutation $\sigma = \sigma_1\sigma_2 \ldots \sigma_n$, cyclically pass through the permutation left to right and remove the numbers $1, 2, \ldots, n$ in order. The \textbf{disorder} of the permutation is then defined by counting the number of times a position is not selected and summing that over all the positions.
\end{definition}

\begin{example}
  Let $\sigma = 12543$. In the first pass, $54$ remains. In the second pass only $5$ remains. In the third pass, nothing remains. Thus the disorder of $\sigma$ is 3.
\end{example}

\begin{prop}[Statistic 446]\label{prop:R_446} The disorder of a permutation is $\frac{n(n-1)}{4}$-mesic for the reverse, but not homomesic for the complement.
\end{prop}

\begin{proof}
  Each pass through the permutation ends when encountering an inversion pair of the form $(i + 1, i)$ (meaning $i + 1$ is to the left of $i$ in $\sigma = \sigma_1\sigma_2 \ldots \sigma_n$). To count the disorder, we break the sum into parts based on those inversion pairs. Any inversion pair $(i + 1, i)$ contributes $n - i$ to the disorder because when $i$ is removed, $i + 1, i + 2, \ldots, n$ remain. As we are summing the disorder over both $\sigma$ and $\R(\sigma)$, we encounter every possible inversion pair of the form $(i + 1, i)$ exactly once. Thus the disorder over $\sigma$ and $\R(\sigma)$ can be found by the sum $\sum_{i = 1}^{n-1} n - i = 1 + 2 + \ldots + (n-1)$, and the average over the orbit is $\frac{n(n-1)}{4}$.

  To see that it is not homomesic under the complement, we exhibit an orbit with an average that differs from the global average of $\frac{n(n-1)}{4}$.  Consider the orbit $( \sigma, \C(\sigma))$ where $\sigma = 132$ and $\C(\sigma) = 312$. The disorder of $\sigma$ is 1 and the disorder of $\C(\sigma)$ is 1, so the average over the orbit is $1$, not $\frac{3(3-1)}{4} = \frac{3}{2}$.\end{proof}

The makl of a permutation was first defined in \cite{ClarkeSteingrimssonZeng} as the sum of the descent bottoms of the permutation with the left embracing sum of the permutation. In this paper, we use the alternative definition in \cite{BabsonSteingrimsson} that defines the makl of the permutation in terms of summing the occurrence of certain patterns.

\begin{definition}\cite{BabsonSteingrimsson} The \textbf{makl} of a permutation is the sum of the number of occurrences of the patterns $1-32, 31-2, 32-1, 21,$ where letters without a dash appear side by side in the pattern, as explained in Definition \ref{def:consecutive_patterns}.
\end{definition}

\begin{example}
  Let $\sigma = 12543$. The pattern $1-32$ appears 4 times, the pattern $31-2$ appears 0 times, the pattern $32-1$ appears 1 time, and the pattern $21$ appears 2 times. Thus the makl of $\sigma$ is 7.
\end{example}

\begin{prop}[Statistic 798]\label{prop:R_798}
  The makl of a permutation is $\frac{n(n-1)}{4}$-mesic for the reverse, but not homomesic for the complement.
\end{prop}

\begin{proof}
  Summing the number of occurrence of the patterns $1-32, 31-2, 32-1, 21$ in the reverse permutation is the same as summing the number of occurrences of the patterns $23-1, 2-13, 1-23, 12$ in the original permutation. Summing all of these patterns is equal to $1 + 2 + \ldots + n-1 = \frac{n(n-1)}{2}$, as each pair $(\sigma_i, \sigma_j)$ in the permutation falls under one of these patterns, as explained below. Let $\sigma = \sigma_1 \sigma_2 \ldots \sigma_n$ a permutation of $[n]$ and $i, j$ such that $1\leq i < j \leq n$.

  \begin{enumerate}
    \item If $j = i + 1$, then $(\sigma_i, \sigma_j)$ either has the pattern $12$ or $21$.
    \item If $j > i + 1$ and $\sigma_i < \sigma_j$, then either
          \begin{itemize}
            \item $\sigma_i < \sigma_{j-1} <\sigma_j$ and $(\sigma_i, \sigma_j)$ has pattern $1-23.$
            \item $\sigma_{j-1} < \sigma_i < \sigma_j$ and $(\sigma_i, \sigma_j)$ has pattern $2-13.$
            \item $\sigma_i < \sigma_j < \sigma_{j-1}$ and $(\sigma_i, \sigma_j)$ has pattern $1-32.$
          \end{itemize}
    \item If $j > i + 1$ and $\sigma_j < \sigma_i$, then either
          \begin{itemize}
            \item $\sigma_i > \sigma_{i + 1} > \sigma_j$ and $(\sigma_i, \sigma_j)$ has pattern $32-1.$
            \item $\sigma_i > \sigma_j > \sigma_{i+1}$ and $(\sigma_i, \sigma_j)$ has pattern $31-2.$
            \item $\sigma_{i+1} > \sigma_i > \sigma_j$ and $(\sigma_i, \sigma_j)$ has pattern $23-1.$
          \end{itemize}
  \end{enumerate}
  So the average over the orbit is $\frac{n(n-1)}{4}$.

  To see that it is not homomesic under the complement, we exhibit an orbit with an average that differs from the global average of $\frac{n(n-1)}{4}$.  Consider the orbit $(\sigma, \C(\sigma))$ where $\sigma = 132$ and $\C(\sigma) = 312$. The makl of $\sigma$ is 2 and the makl of $\C(\sigma)$ is 2, so the average over the orbit is $2$, not $\frac{3(3-1)}{4} = \frac{3}{2}$.
\end{proof}

For the next proposition, we examine the inverse major index of a permutation, $\sigma$, which is the major index for $\sigma^{-1}$ (see Definition \ref{def:basic_stats} for the definition of major index).

\begin{prop}[Statistic 305]\label{prop:R_305}
  The inverse major index is $\frac{n(n-1)}{4}$-mesic for the reverse, but not homomesic for the complement.
\end{prop}

\begin{proof}

  As $\R(\sigma)^{-1} = \C(\sigma^{-1})$ (see Lemma \ref{lem:C&R_relation}), the average over an orbit is $\frac{1}{2}(\maj(\sigma^{-1})+\maj(\R(\sigma)^{-1})) = \frac{1}{2}(\maj(\sigma^{-1})+\maj(\C(\sigma^{-1}))) = \frac{n(n-1)}{4}$, where the last equality is obtained since the major index is homomesic for the complement (proven in Proposition \ref{prop:C_4}).

  To see that it is not homomesic under the complement, we exhibit an orbit with an average that differs from the global average of $\frac{n(n-1)}{4}$. Consider the orbit $(\sigma, \C(\sigma))$ where $\sigma = 132$ and $\C(\sigma) = 312$. The inverse major index of $\sigma$ is 2 and the inverse major index of $\C(\sigma)$ is 2, so the average over the orbit is $2$, not $\frac{3(3-1)}{4} = \frac{3}{2}$.\end{proof}

Lastly, we examine, the load of a permutation, which is defined in \cite{LascouxSchutzenberger} for finite words in a totally ordered alphabet, but for a permutation $\sigma$ it reduces to the major index for $\R(\sigma^{-1})$.

\begin{prop}[Statistic 304]\label{prop:R_304}
  The load of a permutation is $\frac{n(n-1)}{4}$-mesic under the reverse, but not homomesic for the complement.
\end{prop}

\begin{proof} The load of $\sigma$ is given by taking the major index of $\R(\sigma^{-1})$. So the average over an orbit under $\R$ sums the major index of $\R(\sigma^{-1})$ and the major index of $\R(\R(\sigma)^{-1})$ and divides by 2.

  Let $\beta^{-1} = \R(\sigma^{-1})$. We will show that the average above is the same as taking the average of the inverse major index of $\beta$ and $\R(\beta)$. As we know the inverse major index is homomesic for the reverse map, this will prove that the load of a permutation is also homomesic for the reverse map and attains the same average over each orbit as is attained by the inverse major index, which is $\frac{n(n-1)}{4}$.

  Note that $(\mathcal{I} \circ \R \circ \mathcal{I})(\beta^{-1}) = (\R(\beta))^{-1}$ and $(\R \circ \mathcal{I} \circ \R \circ \mathcal{I} \circ \R)(\R(\sigma^{-1})) = \R(\R(\sigma)^{-1})$. Using Lemma \ref{lem:C&R_relation}, this shows that if $\beta^{-1} = \R(\sigma^{-1})$, then $(\R(\beta))^{-1} = \R(\R(\sigma)^{-1})$.

  So summing the major index of $\R(\sigma^{-1})$ and the major index of $\R(\R(\sigma)^{-1})$ is equal to summing the major index of $\beta^{-1}$ and the major index of $(\R(\beta))^{-1}$, which is the same as summing the inverse major index of $\beta$ and $\R(\beta)$.

  To see that it is not homomesic under the complement, we exhibit an orbit with an average that differs from the global average of $\frac{n(n-1)}{4}$. Consider the orbit $(\sigma, \C(\sigma))$ where $\sigma = 132$ and $\C(\sigma) = 312$. The load of $\sigma$ is 2 and the load of $\C(\sigma)$ is 2, so the average over the orbit is $2$, not $\frac{3(3-1)}{4} = \frac{3}{2}$.
\end{proof}

This concludes the proof of Theorem \ref{onlyrev}, showing the statistics listed there are homomesic for the reverse map but not for the complement map.
Thus we have proven Theorems \ref{thmboth}, \ref{onlycomp}, \ref{onlyrev}, illustrating homomesy for the reverse map with 27 of the statistics found in FindStat and for the complement map with 35.

\section{Foata bijection and variations}
\label{sec:foata}
This section examines homomesies of the following statistic under the bijection of Foata \cite{Foata} (also appearing in Foata and Sch\"utzenberger~\cite{FoataSchutzenberger1978}) and related maps.  Recall the inversion number and major index from Definition~\ref{def:basic_stats}.

\begin{definition}(Statistic 1377) $\maj-\inv$ denotes the statistic equal to the difference of the major index and the inversion number: $(\maj-\inv)(\sigma)=\maj(\sigma)-\inv(\sigma)$.
\end{definition}

We turn our attention to defining the maps under which this statistic is homomesic, starting with the Foata bijection.

\begin{definition}
  The \textbf{Foata bijection} $\F$ (Map 67) is defined recursively on $n$:

  Given a permutation $\sigma=\sigma_1 \sigma_2 \ldots \sigma_n$, compute the image inductively by starting with $\F(\sigma_1) = \sigma_1$.
  At the $i$-th step, if $\F(\sigma_1 \sigma_2 \ldots \sigma_i) = \tau_1 \tau_2 \ldots \tau_i$, define $\F(\sigma_1 \sigma_2 \ldots \sigma_i \sigma_{i+1})$ by placing $\sigma_{i+1}$ at the end of $\tau_1 \tau_2 \ldots \tau_i$ and breaking into blocks as follows:
  \begin{itemize}
    \item Place a vertical line to the left of $\tau_1$.
    \item If $\sigma_{i+1} \geq \tau_i$, place a vertical line to the right of each $\tau_k$ for which $\sigma_{i+1} > \tau_k$.
    \item If $\sigma_{i+1} < \tau_i$, place a vertical line to the right of each $\tau_k$ for which $\sigma_{i+1} < \tau_k$.
  \end{itemize}
  Now, within each block between vertical lines, cyclically shift the entries one place to the right.
\end{definition}

\begin{example}
  To compute $\F(31542)$, the sequence of words is:
  \begin{align*}
    3         & \to 3      \\
    |3|1      & \to 31     \\
    |3|1|5    & \to 315    \\
    |315|4    & \to 5314   \\
    |5|3|14|2 & \to 53412.
  \end{align*}
  In total, this gives $\F(31542) = 53412$.
\end{example}

The Foata bijection also behaves nicely with respect to major index and inversions; see the reference \cite{Foata} for the proof.
\begin{lem}[\protect{\cite[Theorem 4.3]{Foata}}]
  \label{lem:Foata}
  The Foata bijection sends the major index to the number of inversions. That is, $\maj(\sigma)=\inv(\F(\sigma))$.
\end{lem}

Other maps relevant to Theorem~\ref{thm:foata} include the Lehmer code to major code bijection and its inverse.  The major code is defined below.  See Definition~\ref{def:lehmercode} for the definition of the Lehmer code.

\begin{definition}\label{def:foata}
  Given $\sigma \in S_n$, the  \textbf{major code} is the sequence $M(\sigma)=(M(\sigma)_1, M(\sigma)_2, \ldots, M(\sigma)_n)$ where $M(\sigma)_i$ is defined as follows. Let $\operatorname{del}_i(\sigma)$ be the permutation obtained by removing all $\sigma_j < i$ from $\sigma$ and then normalizing. $M(\sigma)_i$ is  given by
  $M(\sigma)_i = \operatorname{maj}(\operatorname{del}_i(\sigma)) - \operatorname{maj}(\operatorname{del}_{i-1}(\sigma))$.
\end{definition}

\begin{example}
  For the permutation $31542$, $\operatorname{del}_1(31542)=31542$. To obtain $\operatorname{del}_2(31542)$, we first remove $1$, obtaining $3542$. Then we normalize so that the values are in the interval $[1,4]$, obtaining $\operatorname{del}_2(31542)=2431$. Similarly, $\operatorname{del}_3(31542)=132$, $\operatorname{del}_4(31542)=21$, $\operatorname{del}_5(31542)=1$. Thus, $31542$
  has major code $(3,3,1,1)$, since:
  \[\maj(31542)=8, \quad \maj(2431)=5, \quad \maj(132)=2, \quad \maj(21)=1, \quad \maj(1) = 0.\]
\end{example}

One can recover the permutation from the major code by inverting the process in the Definition \ref{def:foata}.
The sum of the major code of $\sigma$ equals the major index of $\sigma$, that is, $\sum_i M(\sigma)_i=\maj(\sigma)$. Analogously, the sum of the Lehmer code of $\sigma$ equals the inversion number of $\sigma$, that is, $\sum_i L(\sigma)_i=\inv(\sigma)$ (see Proposition~\ref{prop:LC_num_inv_is_sum}).

\begin{definition}
  The \textbf{Lehmer-code-to-major-code} map $\M$ (Map 62) sends a permutation to the unique permutation such that the Lehmer code is sent to the major code. The \textbf{major-index-to-inversion-number} map (Map 73) is its inverse.
\end{definition}

The following lemma is clear from construction.

\begin{lem}
  \label{lem:invtomaj}
  The Lehmer-code-to-major-code map sends the number of inversions to the major index. That is, $\inv(\sigma)=\maj(\M(\sigma))$.
\end{lem}

\begin{thm}
  \label{thm:foata}
  The statistic $\maj-\inv$ \textnormal{(Stat 1377)} is 0-mesic with respect to each of the following maps:
  \begin{itemize} \rm
    \item \textnormal{Map} $62$: The Lehmer-code-to-major-code bijection,
    \item \textnormal{Map} $73$: The major-code-to-Lehmer-code bijection,
    \item \textnormal{Map} $67$: The Foata bijection, and
    \item \textnormal{Map} $175$: The inverse Foata bijection.
  \end{itemize}
\end{thm}

\begin{proof}
  By Lemma~\ref{lem:invtomaj}, $\inv(\sigma)=\maj(\M(\sigma))$, that is, the Lehmer-code-to-major-code bijection $\M$ sends inversion number to major index. Therefore, the sum over an orbit of $\M$ is:
  \[\sum_{\sigma} (\maj-\inv)(\sigma)=\sum_{\sigma} \maj(\sigma)- \sum_{\sigma}\inv(\sigma) = \sum_{\sigma} \maj(\sigma)- \sum_{\sigma}\maj(\M(\sigma)) = \sum_{\sigma} \maj(\sigma)- \sum_{\sigma}\maj(\sigma)=0.\]

  By Lemma~\ref{lem:Foata}, $\maj(\sigma)=\inv(\F(\sigma))$, that is, the Foata bijection $\F$ sends major index to inversion number. Therefore, the sum over an orbit of $\F$ is:
  \[\sum_{\sigma} (\maj-\inv)(\sigma)=\sum_{\sigma} \maj(\sigma)- \sum_{\sigma}\inv(\sigma) = \sum_{\sigma} \inv(\F(\sigma))- \sum_{\sigma}\inv(\sigma) = \sum_{\sigma} \inv(\sigma)- \sum_{\sigma}\inv(\sigma)=0.\]

  Thus, $\maj-\inv$ is $0$-mesic with respect to $\M$ (Map 62) and $\F$ (Map 67), and their inverses (Map 73 and Map 175).
\end{proof}

\section{Kreweras and inverse Kreweras complements.}
\label{sec:krew}
The Kreweras complement was introduced in 1972  as a bijection on noncrossing partitions \cite{kreweras1972partitions}.
In this section, we first describe the Kreweras complement map. We then state Theorem~\ref{Thm:Kreweras} which lists
the statistics we show are homomesic, three from the FindStat database and one not in FindStat at the time of the investigation. Before proving this theorem we describe in detail the orbit structure of the Kreweras complement in Subsection~\ref{sec:KrewOrb}. The homomesies are then proved in Subsection~\ref{sec:KrewHom}.

Given a noncrossing partition on $n$ elements, the Kreweras complement $\K$ can be understood geometrically as rotating an associated noncrossing matching on $2n$ elements and finding the resulting noncrossing partition~
\cite{gobet2016noncrossing}.
The action of the Kreweras complement may be extended to all permutations as follows.
\begin{definition}\label{def:krew}
  Let $\sigma$ be a permutation of $n$ elements.  The \textbf{Kreweras complement} of $\sigma$ and its inverse (Maps 88 and 89 in the FindStat database) are defined as $$\K(\sigma)=c\circ\sigma^{-1}\qquad\mbox{ and }\qquad \K^{-1}(\sigma)=\sigma^{-1}\circ c,$$ where $c$ is the long cycle $234\ldots 1$.
\end{definition}
Here composition is understood from right to left.
Note that in the literature, the definitions of $\K$ and $\K^{-1}$ are often swapped as compared to the above. But since all our results describe orbit sizes and homomesy, which are both invariant under taking the inverse map (see Lemma~\ref{lem:inverse}), this convention choice is immaterial. We chose the above convention to match the code for these maps in FindStat.

\begin{example}
  Consider $\sigma=43152.$  By definition, $\K(43152)=23451\circ 35214=41325$.
\end{example}

In \cite{EinsteinFGJMPR16}, the number of disjoint sets in a noncrossing partition of $n$ elements is shown to be $\frac{n+1}{2}$-mesic under a large class of operations which can be realized as compositions of toggles, including the Kreweras complement.  In this section, we study the generalized action of the Kreweras complement on permutations, proving the following homomesy results.

\begin{thm}
  \label{Thm:Kreweras}
  The Kreweras complement and its inverse exhibit homomesy for the following statistics
  \begin{itemize}
    \item \hyperref[KHom exc]{$\Stat$ $155$}: The number of exceedances of a permutation $(${\small average: $\frac{n-1}{2}$}$)$
    \item \hyperref[KHom exc]{$\Stat$ $702$}: The number of weak deficiencies of a permutation $(${\small average: $\frac{n+1}{2}$}$)$
    \item \hyperref[Khom lastentry]{$\Stat$ $740$}: The last entry of the permutation $(${\small average: $\frac{n+1}{2}$}$)$
    \item \hyperref[Khom lastentry]{$\frac{n}{2}$-th element}: When $n$ is even, the $\frac{n}{2}$-th element of the permutation $(${\small average: $\frac{n+1}{2}$}$)$
  \end{itemize}
\end{thm}
As we prove in Corollary \ref{KC: no entry but}, the last two homomesies are the only $i$-th entry homomesies possible for the Kreweras complement.

\subsection{Kreweras complement orbit structure}
\label{sec:KrewOrb}
In this subsection, we examine the action of the Kreweras complement and its orbit structure, finding the order of the map in Theorem  ~\ref{thm:K_order}, and completely characterizing the distribution of orbits in Theorems ~\ref{thm:orbit_count_odd_sizes}, and \ref{Prop:K even orbs}.  We give explicit generators for orbits of certain sizes in Theorem \ref{thm:Orbit_generators_K}.

The following lemma will be used to prove several results in this section.
\begin{lem}\label{ithentry}
  Let $\sigma = \sigma_1\sigma_2 \ldots \sigma_n$. Then for all integer values of $j$, the $i$-th entry of $\K^j(\sigma)$ is given by:
  \[\K^{j}(\sigma)_i =\begin{cases}
      \sigma_{i - \frac{j}{2}} + \frac{j}{2} \pmod{n}            & \textnormal{if } j \textnormal{ is an even integer}, \\
      \sigma^{-1}_{i - \frac{j-1}{2}} + \frac{j + 1}{2} \pmod{n} & \textnormal{if } j \textnormal{ is an odd integer}.
    \end{cases}\]
  Note, the operation in the subscripts is also modulo $n$, and in both cases $n$ is used for the $0$-th equivalence class representative.
\end{lem}

\begin{proof}
  From the definition of $\K$ and $\K^{-1}$,
  \begin{equation}\label{Eqn:k power m}
    \K^j(\sigma)=
    \begin{cases}
      c^{\frac{j}{2}}\circ\sigma \circ c^{-\frac{j}{2}}                      & \mbox{if $j$ is an even integer,} \\
      c^{\frac{j+1}{2}}\circ\sigma^{-1}\circ c^{-\left(\frac{j-1}{2}\right)} & \mbox{if $j$ is an odd integer.}
    \end{cases}
  \end{equation}
  Thus, if $j$ is even, $\K^j(\sigma)$ is found by rotating $\sigma$ cyclically $\frac{j}{2}$ units and adding $\frac{j}{2}$ to each entry modulo $n$, while if $j$ is odd, $\K^j(\sigma)$ is found by rotating $\sigma^{-1}$ cyclically $\frac{j-1}{2}$ units and adding $\frac{j+1}{2}$ to each entry modulo $n$.
\end{proof}

\begin{thm}[Order]
  \label{thm:K_order}
  For all $n>2,$ $\K$ and $\K^{-1}$ have order $2n$ as elements of $S_{S_n}.$
\end{thm}

\begin{proof}
  Let $e$ be the identity permutation $123\ldots n$ in $S_n$.  By equation (\ref{Eqn:k power m}), when $j$ is odd, $\K^j(e)=c\neq e$. Thus, as an element of $S_{S_n},$ the order of $\K$ must be even.  Also by equation (\ref{Eqn:k power m}), when $j$ is even, $\K^j$ is an element of $\mbox{Inn}(S_n),$ the group of inner automorphisms of $S_n$ (automorphisms defined by conjugation). Since for all $n>2$, $\mbox{Inn}(S_n)\cong S_n$, it follows that $\K^j$ acts as the identity automorphism only when $c^{\frac{j}{2}}=e,$ i.e.\ when $j$ is a multiple of $2n$. Thus the order of $\K$, and equivalently $\K^{-1},$ is $2n.$
\end{proof}

It follows by the orbit-stabilizer theorem that the orbit sizes under the action of $\K$ must be divisors of the order of $\K.$  The distribution of orbit sizes for $n\le 10$ are listed below.  The exponents denote the number of orbits of a given size.
\begin{itemize}
  \item $n=2:$ $[2]$
  \item $n=3:$ $[1,2,3]$
  \item $n=4:$  $[2^{(2)},4,8^{(2)}]$
  \item $n=5: [1, 2^{(2)},5^{(5)},10^{(9)}]$
  \item $n=6: [2^{(3)}, 4^{(3)}, 6^{(7)}, 12^{(55)}]$
  \item $n=7: [1, 2^{(3)}, 7^{(33)}, 14^{(343)}]$
  \item $n=8: [2^{(4)}, 4^{(6)}, 8^{(44)}, 16^{(2496)}]$
  \item $n=9: [1, 2^{(4)}, 3^{(3)}, 6^{(24)}, 9^{(290)}, 18^{(20006)}]$
  \item $n=10: [2^{(5)}, 4^{(10)}, 10^{(383)}, 20^{(181246)}]$
\end{itemize}
Since the orbit of a given permutation under the action of $\K$ is the same as that under the action of $\K^{-1}$, just generated in the opposite order, these results also hold for $\K^{-1}$.  The patterns observed in these results are captured in Theorems ~\ref{thm:orbit_count_odd_sizes}, and \ref{Prop:K even orbs}.  

We start by examining elements of $S_n$ belonging to orbits of odd size, and thank Joel Brewster Lewis for contributing the proofs of Proposition~\ref{prop:orbits_of_odd_length}, Proposition~\ref{prop:orbits_of_size_d_odd}, and Theorem ~\ref{thm:orbit_count_odd_sizes}.

\begin{prop}\label{prop:orbits_of_odd_length}
  A permutation $\sigma \in S_n$ with $n$ odd belongs to an orbit of odd size under the Kreweras complement if and only if $\sigma = c^{\frac{n + 1}{2}} \tau$ for some involution $\tau$ in $S_n$. Hence, there is a bijection between permutations of $S_n$ that are involutions and permutations that belong to an orbit of odd size.
\end{prop}
\begin{proof}
  Suppose $n$ is odd.  Since the order of $\K$ is $2n$, an element $\sigma$ belongs to an orbit of odd size under the action of $\K$ if and only if $\K^n(\sigma) = \sigma$.  
  For any $\sigma$,
  \[
    \K^n(\sigma) = \K\left(\K^{2 \cdot \frac{n - 1}{2}}(\sigma)\right) = \K(c^{\frac{n - 1}{2}} \sigma c^{-\frac{n - 1}{2}}) = c^{\frac{n +1}{2}} \sigma^{-1} c^{-\frac{n - 1}{2}}.
  \]
  Therefore, $\sigma = \K^n(\sigma)$ if and only if
  \[
    \sigma = c^{\frac{n + 1}{2}} \sigma^{-1} c^{-\frac{n - 1}{2}}.
  \]
  Multiplying through by $c^{\frac{n - 1}{2}} = c^{-\frac{n + 1}{2}}$ on the left, we have that $ \sigma = \K^n(\sigma)$ if and only if
  \[
    c^{\frac{n - 1}{2}} \sigma = \sigma^{-1} c^{-\frac{n - 1}{2}}.
  \]
  Setting $\tau := c^{\frac{n - 1}{2}} \sigma$, we have $\sigma = \K^n(\sigma)$ if and only if $\tau = \tau^{-1}$, in other words exactly when $\tau$ is an involution.  Multiplying through by $c^{-\frac{n - 1}{2}} = c^{\frac{n + 1}{2}}$ on the left gives the result.
\end{proof}

It is possible to refine the last result, by describing the permutations in orbits of size $d$ for any divisor $d$ of $n$.

\begin{prop}\label{prop:orbits_of_size_d_odd}
Let $d$ be an (odd) divisor of $n,$ with $n$ odd. A permutation $\sigma \in S_n$ belongs to an orbit of size that divides $d$ if and only if $\sigma = c^{\frac{n - d}{2}} \tau = \tau c^{\frac{n-d}{2}}$ for some involution $\tau$ in $S_n$. Furthermore, there is a bijection between permutations of $S_n$ that belong to an orbit of odd size that divides $d$ and pairs made of an involution $\pi$ of $S_d$ and a word on the alphabet $\{1,\ldots, \frac{n}{d}\}$ of length $\frac{d-\#\fix(\pi)}{2}$, where $\fix(\pi)$ is the set of fixed points of $\pi$: $\{x \in [d] \mid \pi(x) = x\}$.
\end{prop}
\begin{proof}
Let $n$ be odd, $d$ be a divisor of $n$ and $\sigma \in S_n$ be in an orbit of odd size that divides $d$ for $\K$. Hence, $\K^n(\sigma) = \K^d(\sigma) = \sigma$. Therefore, we know from Proposition \ref{prop:orbits_of_odd_length} that there exists an involution $\tau \in S_n$ for which $\sigma = c^{\frac{n+1}{2}} \tau$ . Repeating the same trick as before, we have
  \begin{equation}\label{eq:K_orbit_q}
    \sigma = \K^d(\sigma) = c^{\frac{d+1}{2}} \sigma^{-1} c^{-\frac{d-1}{2}}.
  \end{equation}
Using that $\sigma = c^{\frac{n+1}{2}} \tau$ for some involution $\tau$, we rewrite Equation \eqref{eq:K_orbit_q} as
  \[
    \sigma = c^{\frac{n+1}{2}} \tau = c^{\frac{d+1}{2}} \tau c^{-\frac{n+1}{2}} c^{-\frac{d-1}{2}} = c^{\frac{d+1}{2}} \tau c^{-\frac{n+d}{2}} = c^{\frac{d+1}{2}} \tau c^{\frac{n-d}{2}}
  \]
  (where in the last step we use $c^n = e$).
  Multiplying $c^{\frac{n+1}{2}} \tau =c^{\frac{d+1}{2}} \tau c^{\frac{n-d}{2}}$ through by $c^{-\frac{d+1}{2}}$ on the left, this becomes
  \[
    c^{\frac{n-d}{2}}   \tau = \tau c^{\frac{n-d}{2}},
  \] 
  meaning that we need to count involutions that commute with the power $c^{\frac{n-d}{2}}$ of our original $n$-cycle. We will show that the number of such involutions is the number of pairs made of an involution $\pi$ of $S_d$ and a word on the alphabet $\{1,\ldots, \frac{n}{d}\}$ of length $\frac{d-\#\fix(\pi)}{2}$.\\
  
Since $\gcd(n, \frac{n-d}{2}) = \gcd(n, n - d) = d$ and $c$ is an $n$-cycle, the element $c' := c^{\frac{n-d}{2}}$ has $d$ cycles, each of order $\frac{n}{d}$.  Suppose that the involution $\tau$ commutes with $c'$, and let $a \in [n]$.  We consider an exhaustive list of cases for $\tau(a)$: $\tau(a) = a$; $\tau(a) \neq a$ and $\tau(a)$ belongs to the same cycle in $c'$ as $a$; or $\tau(a)$ belongs to a different cycle of $c'$ than $a$ does.

  In the first case, we have $\tau(c'(a)) = c'(\tau(a)) = c'(a)$ and so $\tau$ also fixes $c'(a)$; and thus $\tau$ fixes the entire cycle of $a$ pointwise.

  In the second case, let $b := \tau(a) = (c')^k(a)$ for some $k$.  Applying $\tau$ to this equation and using the hypothesis gives $a = \tau(b) = \tau( (c')^k(a)) = (c')^k(\tau(a)) = (c')^k(b)$, and so $a = (c')^{2k}(a)$.  However, since all cycles of $c'$ are of odd length, this can only happen if actually $a = (c')^k(a)$ -- which is the previous case $\tau(a) = a$.  Thus, this case never occurs.

  Finally, in the third case, suppose $(a_1 \cdots a_{n/d})$ and $(b_1 \cdots b_{n/d})$ are two cycles of $c'$, and that $\tau(a_1) = b_m$.  Then (by the hypothesis) $\tau(a_{i + 1}) = \tau( (c')^i(a_1)) = (c')^i(\tau(a_1)) = (c')^i(b_m) = b_{m + i}$ -- so the value of $\tau$ is determined on the entire cycle by the choice of the image of $a = a_1$. Because $\tau$ is an involution, $\tau(b_m) = a$, and the the value of $\tau$ on the cycle $(b_1 \cdots b_{n/d})$ is also determined by the image of $a$.  

  The preceding case analysis becomes the following enumeration: for each of the $d$ cycles of $c'$, we find the cycle that is its image under $\tau$. Listing all the cycles in some order $(c_1, \ldots, c_d)$, there exists an involution $\pi \in S_d$ for which  the following equality of sets holds: $\{\tau(x) \mid x \in c_i\} = c_{\pi(i)}$. For each fixed point of $\pi$, no further choice is needed as the value of $\tau$ is fixed on the whole cycle. For each transposition $(i,j)$ of $\pi$ with $i < j$, it is sufficient to choose the image in $c_j$ of the smallest element of $c_i$ to determine the image of both cycles under $\tau$. Since $\pi$ is an involution in $S_d$, the number of transpositions of $\pi$ is $\frac{d-\#\fix(\pi)}{2}$. For each transposition $(i,j)$, there are $\frac{n}{d}$ possible choices for the image of the smallest element of $c_i.$ Putting it all together, the number of involutions $\tau$ commuting with $c'$ in which each cycle $c_i$ is sent to $c_{\pi(i)}$ under $\tau$ is exactly $\left(\frac{n}{d}\right)^{\frac{d-\#\fix(\pi)}{2}}$.
\end{proof}

This allows us to compute the orbit sizes when $n$ is odd, as follows.

\begin{thm}\label{thm:orbit_count_odd_sizes}
  Let $n$ be an odd number and $d$ be an (odd) divisor of $n$.  Then the number of orbits of $\K$ of order $d$ is
  \[\frac{1}{d} \sum_{q \mid d} \mu\left(\frac{d}{q}\right)\sum_{j = 0}^{\lfloor q/2 \rfloor} \binom{q}{2j} \cdot (2j - 1)!! \cdot \left(\frac{n}{q}\right)^j,
  \]
  where the double factorial $k!!$ is the product of the integers from $1$ to $k$ that have the same parity as $k$, and $\mu$ is the number-theoretic M\"obius function
\end{thm}

\begin{proof}

The proof technique is to compute all permutations $\sigma$ for which $\K^d(\sigma) = \sigma$ and then use inclusion-exclusion to get those that belong to an orbit of size $d$ for each divisor $d$ of $n$.

From Proposition \ref{prop:orbits_of_size_d_odd}, we know that there is a bijection between permutations of $S_n$ that belong to an orbit of odd size that divides $d$ (what we are trying to count) and pairs made of an involution $\pi$ of $S_d$ and a word on the alphabet $\{1,\ldots, \frac{n}{d}\}$ of length $\frac{d-\#\fix(\pi)}{2}$, where $\fix(\pi)$ is the set of fixed points of $\pi$. The latter is easier to compute.

First, the number of words of length $\frac{d-\#\fix(\pi)}{2}$ with letters in $\left[\frac{n}{d}\right]$ is $\left(\frac{n}{d}\right)^{\frac{d-\#\fix(\pi)}{2}}$. Therefore, the number of permutations $\sigma$ for which $\K^d(\sigma)=\sigma$ is 
\[  \sum_{\pi \in S_d,\ \pi^2=e} \left(\frac{n}{d}\right)^{\frac{d-\#\fix(\pi)}{2}}. \]

We next rewrite this sum so permutations of $S_d$ are grouped by their number of fixed points, and let $j = \frac{d-\#\fix(\pi)}{2}$:
\[  \sum_{j=0}^{\lfloor d/2 \rfloor} \left(\frac{n}{d}\right)^{j} \#\left\{\pi \in S_d \mid \pi^2=e, \frac{d-\#\fix(\pi)}{2} = j \right\}.\]

Involutions of $d$ with $d-2j$ fixed points are counted in the following way: we first pick the non-fixed points (there are $\binom{d}{2j}$ of them), and then match them so they come in pairs. The pairs are counted by $(2j-1)(2j-3)(2j-5)\ldots1 = (2j-1)!!$.

  Consequently, our sum becomes
  \[
    \sum_{j = 0}^{\lfloor d/2 \rfloor} \binom{d}{2j} \cdot (2j - 1)!! \cdot \left(\frac{n}{d}\right)^j.
  \]
  This is the number of elements in $S_n$ that are fixed by $K^d$, i.e.\ that belong to orbits of $K$ of size \emph{dividing} $d$.  In order to find the number of elements that belong to orbits of size \emph{exactly} $d$, we do a M\"obius inversion, finding the number of such elements is
 \[
    \sum_{q \mid d} \mu\left(\frac{d}{q}\right) \sum_{j = 0}^{\lfloor q/2 \rfloor} \binom{q}{2j} \cdot (2j - 1)!! \cdot \left(\frac{n}{q}\right)^j,
  \]
where $\mu$ is the number-theoretic M\"obius function.  Finally, to get the number of orbits, we divide by the common size $d$ of the orbits, as needed.
\end{proof}

Calculating the number of orbits of even size is even more straight forward.

\begin{thm}[Even sized orbit cardinality]\label{Prop:K even orbs}
For each even divisor $2k$ of $2n$, the number of orbits of size $2k$ under the action of the Kreweras complement is equal to $$\frac{\left(\frac{n}{k}\right)^kk!-T}{2k},$$ where $\left(\frac{n}{k}\right)^kk!$ is the number of elements in $S_n$ fixed by $\K^{2k}$ and $T$ is the number of elements in orbits of size $q$ for each proper divisor $q$ of $2k.$ 
\end{thm}

\begin{proof}
  Let $\sigma\in S_n$ be an arbitrary element fixed under the action of $\K^{2k}$.  By Equation \eqref{Eqn:k power m}, $c^k\circ\sigma\circ c^{-k}=\sigma,$ or equivalently $\sigma\circ c^k\circ \sigma^{-1}=c^k.$  Thus $\sigma$ is fixed by $\K^{2k}$ if and only if it is in the centralizer of $c^k$ under the action of the inner automorphism group $\mbox{Inn}(S_n).$  By the orbit-stabilizer theorem, the number of elements in the centralizer of $c^k$ is equal to $\frac{|\mbox{Inn}(S_n)|}{|\mbox{Orb}(c^k)|}$ where $|\mbox{Orb}(c^k)|$ is the number of elements in the orbit of $c^k$ under conjugation, or equivalently the number of elements in $S_n$ with the same cycle structure as $c^k$.  As $c^k$ consists of $k$ cycles, each of length $\frac{n}{k}$, and there are exactly $\frac{n!}{(\frac{n}{k})^kk!}$ permutations in $S_n$ with this cycle structure, it follows that the number of elements in the centralizer of $c^k$ is $\left(\frac{n}{k}\right)^kk!.$
  
Elements fixed by the action of $K^{2k}$ include those in orbits of size $q$ for all $q\mid 2k.$  To find the number of elements in orbits of size $2k$ we need to subtract out elements in orbits of size equal to a proper divisor of $2k$.  Dividing the result by $2k$ gives the number of distinct orbits of size $2k$ under the action of $\K.$

\end{proof}

\begin{remark}
Consider $\sigma$ in the centralizer of $c^k$ under the action of the inner automorphism group.  Since $\sigma\circ c^k \circ\sigma^{-1}=c^k$, conjugation by $\sigma$ has the effect of permuting the disjoint cycles $c_1,c_2,\ldots, c_k$ of $c^k.$ Furthermore, since $\sigma\circ c^k\circ \sigma^{-1}=\sigma c_1c_2\ldots c_k \sigma^{-1}=\sigma c_1\sigma^{-1}\sigma c_2\sigma^{-1}\sigma\ldots \sigma^{-1}\sigma c_k \sigma^{-1}$ and writing $c_i=(c_{i1}c_{i2}\ldots c_{i\frac{n}{k}})$, we have $\sigma c_i\sigma^{-1}=(\sigma(c_{i1})\sigma(c_{i2})\ldots \sigma(c_{i\frac{n}{k}})),$ it follows that elements of the centralizer of $c^k$ are completely determined by how they permute the disjoint cycles of $c^k$ and where they map $c_{i1}$ for each $1\leq i\leq k.$  Put another way, for each permutation $\pi$ in $S_k$, we can construct an element of the centralizer of $c^k$ from a word $w$ of length $k$ on the alphabet $\{1,2,\ldots,\frac{n}{k}\}$ 
by setting $\sigma(c_{i1})=c_{\pi(i)w_i}.$  Thus, similar to the proof of Proposition ~\ref{prop:orbits_of_size_d_odd}, we can count the number of elements in the centralizer of $c^k$ using a bijection between these elements and pairs consisting of a permutation in $S_k$ and a word of length $k$ on the alphabet $\{1,2,\ldots,\frac{n}{k}\}.$  
\end{remark}

\begin{remark}
  As we prove in Corollary \ref{Korb odd divisor even n}, there are no orbits of odd size when $n$ is even. Thus, Theorem \ref{Prop:K even orbs} completely characterizes the distribution of orbits when $n$ is even, allowing us to calculate the number of orbits of size $2k$ using the M\"obius inversion formula.  By Theorem \ref{Prop:K even orbs}, the number of elements fixed by $2k$, or equivalently in orbits of size dividing $2k$, is $\left(\frac{n}{k}\right)^kk!.$ When $n$ is even, the number of elements in orbits of size exactly $2k$ (when $k$ divides $n$) is $$\sum_{q|k}\mu\left(\frac{k}{q}\right)\cdot\left(\frac{n}{q}\right)^q\cdot q!.$$ 
\end{remark}

Below we examine orbits of size $1$, $2$, $n$, and $2n$, providing explicit generators in each case.

\begin{thm}\label{thm:Orbit_generators_K} The following permutations generate orbits of a given size for the Kreweras complement and its inverse:
  \begin{itemize}
    \item When $n$ is odd  $\left\{\frac{n+3}{2}\frac{n+5}{2}\ldots n123\ldots\frac{n+1}{2}\right\}$ is the unique orbit of size $1$. There is no orbit of size $1$ when $n$ is even.
    \item For all $n$, there are $\lfloor \frac{n}{2}\rfloor$ orbits of size $2,$ one of which is generated by the identity permutation.
    \item For all $n$, an orbit of size $n$ is generated by $n(n-1)\ldots 321$.
    \item For even values of $n>3$, an orbit of size $2n$ is generated by $13\ldots(n-1)24\ldots n$, and for odd values of $n>3$ an orbit of size $2n$ is generated by $13\ldots n24\ldots (n-1).$
  \end{itemize}
\end{thm}

We prove each part of this theorem separately through Propositions \ref{Prop:K fixed points} --\ref{Prop:K_orbit_size_2n}.

\begin{prop}\label{Prop:K fixed points} There are no fixed points in $S_n$ under the action of the Kreweras complement when $n$ is even and exactly one at
  \[\sigma = \frac{n+3}{2}\frac{n+5}{2}\ldots n123\ldots\frac{n+1}{2}\] when $n$ is odd.
\end{prop}

\begin{proof}
  Let $\sigma\in S_n$ be a fixed point under the action of the Kreweras complement, i.e.\ $\K(\sigma)=c\circ\sinv=\sigma,$ and $\sigma^2=c.$  Thus, if $\sigma_1=i,$ then $\sigma_i=2$ and $\sigma_2=i+1$ and so on with $\sigma_j=i+j-1$ and $\sigma_{(i+j-1)}=j+1$, in other words $\sigma=i(i+1)\ldots n 1 2 \ldots (i-1)$. Examining $\sigma$ we find $n$ appears in the $n-(i-1)$ spot, that is, $\sigma_{n-(i-1)}=n$. However since $\sigma_i=2$ and $n$ appears two positions to the left of $2,$ $n$ is in the $(i-2)$ spot, i.e.\ $\sigma_{i-2}=n.$ Setting the indices $n-(i-1)$ and $i-2$ equal and solving for $i$ we find $i=\frac{n+3}{2}$.
\end{proof}

\begin{prop}
  The Kreweras complement has $\lfloor \frac{n}{2} \rfloor $ orbits of size $2$. In particular, the identity permutation generates an orbit of size $2$.
\end{prop}

\begin{proof}
  By direct computation, we find that $\{1234\ldots n, 234\ldots n1\}$ is an orbit of size 2 under the Kreweras complement. By Theorem \ref{Prop:K even orbs}, there are $\frac{n-T}{2}$ distinct orbits of size 2, where $T$ is the number of orbits of size 1.  By Proposition \ref{Prop:K fixed points}, $T$ is zero when $n$ is even and $1$ when $n$ is odd, thus $\frac{n-T}{2}=\lfloor \frac{n}{2} \rfloor.$
\end{proof}

The following proposition shows the reverse of the identity permutation is in an orbit of size~$n$.

\begin{prop}\label{prop:orbit_size_n}
  The Kreweras complement has an orbit of size $n$ of the form $\{ n(n-1) \ldots 321, 1n(n-1) \ldots 32, 21n \ldots 43, 321n \ldots 54, \ldots, (n-1)(n-2)\ldots 21n\}$.
\end{prop}

\begin{proof}
  Let $\sigma=n(n-1) \ldots 321$.  We start by noting that the $i$-th entry of $\sigma$ is given by $\sigma_i=n-(i-1)$ and $\sigma=\sigma^{-1}$.  By Lemma \ref{ithentry}, $\K^j(\sigma)_i$ is equal to $n-(i-\frac{j}{2}-1)+\frac{j}{2}\pmod{n}$ when $j$ is even, and $n-(i-\frac{j-1}{2}-1)+\frac{j+1}{2}\pmod{n}$ when $j$ is odd.
  In both cases $\K^j(\sigma)_i=n-(i-1)+j\pmod{n},$ which is only equal to $\sigma_i$ when $j$ is a multiple of $n$, and gives the specified orbit when evaluated at $0\le j<n$.
\end{proof}

Finally, the proposition below exhibits an orbit of size $2n$, completing the proof of Theorem~\ref{thm:Orbit_generators_K}.
\begin{prop}\label{Prop:K_orbit_size_2n}
  For $n > 3$, the permutation $\sigma = 1 3 \ldots (n-1)2 4 \ldots n$ when $n$ is even, and $\sigma = 1 3 \ldots (n)2 4 \ldots (n-1)$ when $n$ is odd,
  generates an orbit of size $2n$ under the Kreweras complement.
\end{prop}

\begin{proof}
  Let $\sigma$ be the permutation on $n$ elements given in the proposition, i.e.\
  \[\sigma_i=
    \begin{cases}
      2(i-1)+1                     & \mbox{for } 1\leq i\leq\lceil\frac{n}{2}\rceil \\
      2(i-\lceil\frac{n}{2}\rceil) & \mbox{for } \lceil\frac{n}{2}\rceil<i\leq n
    \end{cases}.
  \]
  Direct calculation shows that the orbit of $\sigma$ under the Kreweras complement has size $n$ for $n=2,$ and $3$ and size $2n$ for $n=4,$ and $5.$
  For $n>5$, we use the fact that all odd entries of $\sigma$ are followed by all even entries to show that $\K^{j}(\sigma) \neq \sigma$ for $0 < j < 2n.$

  Consider $\sigma^{-1}=1(\frac{n}{2}+1)2(\frac{n}{2}+2)3(\frac{n}{2}+3)\ldots \frac{n}{2} n$ for $n$ even, and $\sigma^{-1}=1(\lceil\frac{n}{2}\rceil+1)2 (\lceil\frac{n}{2}\rceil+2)3( \lceil\frac{n}{2}\rceil+3)\ldots \lceil\frac{n}{2}\rceil$ for $n$ odd.  Unlike in $\sigma$, even and odd entries of $\sigma^{-1}$ appear in alternating pairs, potentially bookended on either or both sides by a single even or odd entry. To illustrate this point, when $n=4$, $\sigma^{-1}=1324$, when $n=5$, $\sigma^{-1}=14253$, when $n=6$, $\sigma^{-1}=142536$, and when $n=7$, $\sigma^{-1}=1526374.$  As conveyed in Lemma \ref{ithentry}, when $j$ is odd, $\K^j(\sigma)$ is equivalent to rotating $\sigma^{-1}$ some set amount and then adding a constant amount to each entry modulo $n$. This process preserves the grouping of even and odd entries of $\sigma^{-1}$ when viewed as a cyclic ordering, and for $n>5$  never results in all even entries followed by all odd entries.  Thus, for $n>5,$ $\K^j(\sigma)\neq \sigma$ when $j$ is odd.

  Next, consider $j$ even, in which case the action of $\K^j$ can be realized as acting on $\sigma$ by $\K^2$ some number of times.  By Lemma \ref{ithentry}, the action of $\K^2$ is the same as cyclically shifting each entry of $\sigma$ one place to the right and adding one modulo $n$.  When $n$ is even, this process alternates the parity of every entry and thus preserves the string of $\frac{n}{2}$ even and $\frac{n}{2}$ odd entries (though cyclically shifted). For the permutation to return to a state where the first $\frac{n}{2}$ entries (and thus also the last $\frac{n}{2}$ entries) have the same parity, we have to shift the entries some multiple of $\frac{n}{2}$ times, i.e.\ $j$ must be some multiple of $n$.
  For $n$ even, $\K^n(\sigma) = (2 + \frac{n}{2})(4 + \frac{n}{2}) \ldots (n + \frac{n}{2})(1 + \frac{n}{2}) \ldots (n - 1 + \frac{n}{2})$, where the entries are calculated modulo $n$.  If $\K^n(\sigma) = \sigma$, then $(2 + \frac{n}{2}) \mod{n} = 1$. But as $n > 5$, this is not possible. Thus $\K^n(\sigma) \neq \sigma$ and the smallest $j$ such that $\K^j(\sigma) = \sigma$ is $2n$.

  For $n$ odd, the case where $j$ is even is slightly more complicated.  $\K^j(\sigma)$ can still be thought of as repeatedly acting on $\sigma$ by $\K^2,$ however each time we shift and add one to the entries of $\sigma$, every entry \emph{except $n$} alternates in parity.  To show that $\K^j(\sigma)\neq\sigma$ for $0< j<2n$ when $n$ is odd, we examine the position of $1$ in $\K^{j}(\sigma).$  Suppose $1$ appears in the $i$-th position of $\K^j(\sigma).$  By Lemma \ref{ithentry}, $\K^j(\sigma)_i=\sigma_{i-\frac{j}{2}}+\frac{j}{2}\pmod{n}.$  Thus, if $\K^j(\sigma)_i=1$, it follows that $\sigma_{i-\frac{j}{2}}=1-\frac{j}{2}\pmod{n}.$  Let $j=2k$, then $\sigma_{i-k}=1-k=n+1-k\pmod{n}.$  When $k$ is even, $n+1-k$ is even and appears in the $n+1-\frac{k}{2}$ position by the definition of $\sigma$.  Setting the indices equal we have $i-k=n+1-\frac{k}{2}\pmod{n}$, which implies $i=1+\frac{k}{2}\pmod{n}.$  Thus, when $k$ is even, the first time $i=1$ is when $k=2n.$  By a similar argument, if $k$ is odd, $\sigma_{i-k}=n+1-k$ is odd and appears at position $\frac{n-k}{2}+1$ in $\sigma$. Again, setting indices equal and solving for $i$ we have $i=\frac{n+k}{2}+1.$  In this case the first time $i=1$ is when $k=n.$  Thus, if $n$ is odd, the first time $1$ returns to the first position (i.e.\ $i=1$), and therefore $\K^{j}(\sigma)$ might equal $\sigma,$ is when $k=n$, or equivalently $j=2n.$
\end{proof}

\subsection{Kreweras complement homomesies}
\label{sec:KrewHom}
Empirical investigations using the FindStat database suggested three potential homomesies for the Kreweras complement and its inverse. In this subsection, we prove those homomesies, as well as a fourth that was not included in the FindStat database at the time of our investigation.

We begin with the relevant definitions and lemmas that will be useful in proving Theorem~\ref{Thm:Kreweras}.

\begin{definition}
\label{def:exceedance}
  Given a permutation $\sigma\in S_n$, an index $i$ is said to be an \textbf{exceedance} of $\sigma$ if $\sigma_i>i$, a \textbf{deficiency} of $\sigma$ if $\sigma_i < i$, and a \textbf{weak deficiency} or \textbf{anti-exceedance} if $\sigma_i\le i.$  The number of exceedances $\#\{i:\sigma_i> i\}$ is denoted $\exc(\sigma)$, while the number of points fixed by $\sigma$ is denoted $\fp$.
\end{definition}

\begin{lem}\label{exc_sinv} Given a permutation $\sigma\in S_n$,   $$\exc(\sigma^{-1})=n-\exc(\sigma)-\fp.$$
\end{lem}
\begin{proof}
  Examining the weak deficiencies of $\sigma^{-1}$ we note that  $\sigma^{-1}_i=m\le i$ if and only if $\sigma_m=i\ge m$.  Thus the number of weak deficiencies of $\sigma^{-1}$ is the sum of the number of exceedances and fixed points of $\sigma$.  As the exceedances of $\sigma^{-1}$ are the complement of the weak deficiencies of $\sigma^{-1}$, the result follows.
\end{proof}

\begin{lem}\label{exc_K}
  Given $\sigma\in S_n$, the following formula gives the number of exceedances in the image under $\K$ and $\K^{-1}.$ $$\exc(\K(\sigma))=n-\exc(\sigma)-1=\exc(\K^{-1}(\sigma)).$$
\end{lem}
\begin{proof}
  For all $i<n,$ if $i$ is a fixed point of $\sigma^{-1}$, then $(c\circ\sigma^{-1})_i=i+1$ and $i$ is an exceedance of $\K(\sigma)=c\circ\sigma^{-1}.$ If $n$ is a fixed point, $(c\circ\sinv)_n=1$ and $n$ is not an exceedance of $\K(\sigma).$ Similarly, if $i$ is an exceedance of $\sigma^{-1}$ with $\sinv_i\neq n$, then $(c\circ\sigma^{-1})_i=\sigma^{-1}_i+1$ is an exceedance.  If $\sigma^{-1}_i=n,$ then  $(c\circ\sinv)_i=1$ and $i$ is not an exceedance of $\K(\sigma).$
  Thus, if $n$ is a fixed point of $\sigma^{-1}$, $$\exc(\K(\sigma))=\exc(c\circ\sinv)=\exc(\sinv)+(\fp-1),$$ while if $n$ is not a fixed point $$\exc(\K(\sigma))=(\exc(\sinv)-1)+\fp.$$ Therefore by Lemma \ref{exc_sinv}, $\exc(\K(\sigma))=n-\exc(\sigma)-1.$

  To prove the result for $\K^{-1}$, let $\tau=\K^{-1}(\sigma)$. Then $\K(\tau) = \sigma$ and we know that $\exc(\sigma) = \exc(\K(\tau)) = n-\exc(\tau)-1 =n-1-\exc(\K^{-1}(\sigma))$. Solving for $\exc(\K^{-1}(\sigma))$, we find $\exc(\K^{-1}(\sigma)) = n-\exc(\sigma)-1.$
\end{proof}

We are now ready to prove the homomesy results stated in Theorem \ref{Thm:Kreweras}.

\begin{prop}[Statistics 155 and 702]\label{KHom exc}
  For the Kreweras complement and its inverse acting on $S_n$, the number of exceedances is $\frac{n-1}{2}$-mesic while the number of weak deficiencies is $\frac{n+1}{2}$-mesic.  Furthermore, the number of exceedances and weak deficiencies is constant over orbits of odd size.
\end{prop}

\begin{proof}  We start by observing that, as a  consequence of Lemma \ref{exc_K},$$\exc(\K^{-2}(\sigma))=\exc(\K^2(\sigma))=n-\exc(\K(\sigma))-1=n-(n-\exc(\sigma)-1)-1=\exc(\sigma).$$  More generally,
  $$\exc(\K^{\pm m}(\sigma))= \left\{
    \begin{array}{lr}
      \exc(\sigma)     & \mbox{when $m$ is even} \\
      n-\exc(\sigma)-1 & \mbox{when $m$ is odd}
    \end{array}\right .
  $$
  Let $\ell$ be the size of the orbit of $\sigma$ under the action of $\K$, i.e.\ $\K^{\ell}(\sigma)=\sigma.$ Thus $\exc(\K^{\ell}(\sigma))=\exc(\sigma).$  If $\ell$ is odd, it follows that $$\exc(\K^{\ell}(\sigma))=n-\exc(\sigma)-1=\exc(\sigma)$$ and solving for $\exc(\sigma)$ we find $$\exc(\sigma)=\frac{n-1}{2}.$$
  Thus, when $\ell$ is odd, the number of exceedances is constant over the orbit and equal to $\frac{n-1}{2}$.

  When $\ell$ is even, the average number of exceedances over the orbit is
  \begin{eqnarray*}
    \frac{1}{\ell}\sum_{i=0}^{\ell-1} \exc(\K^i(\sigma))&=&\frac{1}{\ell}\sum_{j=0}^{\frac{\ell}{2}-1}\exc(\K^{2j}(\sigma))+\exc(\K^{2j+1}(\sigma))\\
    &=&\frac{1}{\ell}\sum_{j=0}^{\frac{\ell}{2}-1}\exc(\sigma)+(n-\exc(\sigma)-1)\\
    &=&\frac{1}{\ell}\sum_{j=0}^{\frac{\ell}{2}-1}n-1=\frac{n-1}{2}.
  \end{eqnarray*}
  Thus, the number of exceedances is $\frac{n-1}{2}$-mesic under the action of the Kreweras complement. The result for $\K^{-1}$ follows from Lemma \ref{lem:inverse}.

  By definition, the number of weak deficiencies is equal to $n-\exc(\sigma)$ for all $\sigma\in S_n.$   Thus, since $\exc$ is $\frac{n-1}{2}$-mesic under the action of the Kreweras complement and its inverse, the number of weak deficiencies is $n-\frac{n-1}{2}=\frac{n+1}{2}$-mesic.
\end{proof}

Our homomesy result for the number of exceedances explains the observed lack of odd sized orbits under the action of the Kreweras complement on $S_n$ when $n$ is even.

\begin{cor}\label{Korb odd divisor even n}
  If $n$ is even, there are no orbits of odd size under the Kreweras complement or its inverse acting on $S_n$.
\end{cor}

\begin{proof}
  Consider orbits of odd size under the Kreweras complement and its inverse.  As shown in the proof of Proposition \ref{KHom exc}, the number of exceedances for each element of such an orbit is the constant value $\frac{n-1}{2}$.  Since the number of exceedances must be an integer, we arrive at a contradiction when $n$ is even.
\end{proof}

\begin{definition}
  Given a permutation $\sigma\in S_n$, define the lower middle element to be $\sigma_{\frac{n}{2}}$ when $n$ is even, and $\sigma_{\frac{n+1}{2}}$ when $n$ is odd.
\end{definition}

\begin{prop}[Statistic 740]\label{Khom lastentry}
  The last entry of a permutation, and when $n$ is even, the lower middle element, are $\frac{n+1}{2}$-mesic under the Kreweras complement and its inverse.
\end{prop}

\begin{proof}
  We start by showing that the average of the set  $\{ \K^j(\sigma)_i \ | \ j = 0, 1, \ldots , 2n-1\}$ is $\frac{n + 1}{2}$ when $i=n$ and when $i=\frac{n}{2}$ for even $n$.

  By Lemma \ref{ithentry}, running over even values of $j$, the last entry of $\K^j(\sigma)$, i.e.\ $\K^j(\sigma)_n,$ takes the form $\sigma_n, \sigma_{n-1} + 1, \dots, \sigma_{1} + (n - 1) $; and running over odd values of $j$, $\K^j(\sigma)_n$ takes the forms $\sigma^{-1}_n + 1, \sigma^{-1}_{n-1} + 2 , \dots , \sigma^{-1}_1$.  In each case, the entries are calculated modulo $n$ with $n$ used for the $0$-th equivalence class.

  To find the sum of $\K^j(\sigma)_n$ for $0\leq j\leq 2n-1$, we show that these last terms can be partitioned into pairs which sum to $n+1.$
  Suppose $\sigma_{n-k}=m.$  Breaking up the sum as follows $$\sum_{j=0}^{2n-1}\K^j(\sigma)_n=\sum_{k=0}^{n-1}[\sigma_{n-k}+k]_n+[\sigma^{-1}_m+(n-m+1)]_n,$$
  we observe that
  \begin{itemize}
    \item if $\sigma_{n-k}+k=m+k>n$, then $\sigma^{-1}_m+(n-m+1)=n-k+(n-m+1)\leq n$.  Since both terms are between $1$ and $2n$, it follows $[\sigma_{n-k}+k]_n+[\sigma^{-1}_m+(n-m+1)]_n=(m+k-n)+(n-k+(n-m+1))=n+1.$
    \item Similarly, if $\sigma_{n-k}+k\leq n$, then $\sigma^{-1}_m+(n-m+1)> n$, and $[\sigma_{n-k}+k]_n+[\sigma^{-1}_m+(n-m+1)]_n=(m+k)+(n-k+(n-m+1)-n)=n+1.$
  \end{itemize} Thus, the average of the last entries is
  $$\frac{1}{2n}\sum_{j=0}^{2n-1}\K^j(\sigma)_n=\frac{1}{2n}\sum_{k=0}^{n-1}(n+1)=\frac{n+1}{2}.$$

  An analogous argument works to find the sum of the lower middle elements, $\{\K^j(\sigma)_{\frac{n}{2}}| 0\leq j\leq 2n-1\},$ when $n$ is even.  By Lemma \ref{ithentry}, running over the even values of $j,$ $\K^j(\sigma)_{\frac{n}{2}}$ takes the form $\sigma_{[\frac{n}{2}]_n}, \sigma_{[\frac{n}{2}-1]_n} + 1, \sigma_{[\frac{n}{2}-2]_n} + 2, \ldots, \sigma_{[\frac{n}{2}-(n-1)]_n} + (n-1).$ Running over the odd values of $j,$ $\K^j(\sigma)_{\frac{n}{2}}$ takes the form  $\sigma^{-1}_{[\frac{n}{2}]_n} + 1, \sigma^{-1}_{[\frac{n}{2}-1]_n} + 2, \ldots, \sigma^{-1}_{[\frac{n}{2}-(n-1)]_n} + n$. As before, values are calculated modulo $n$ with $n$ used for the $0$-th equivalence class.

  Suppose $\sigma_{[\frac{n}{2}-k]_n}=m$, and break up the sum  $$\sum_{j=0}^{2n-1}\K^j(\sigma)_{\frac{n}{2}}=\sum_{k=0}^{n-1}[\sigma_{[\frac{n}{2}-k]_n}+k]_n+[\sigma^{-1}_m+\frac{n}{2}-m+1]_n,$$
  observing that
  \begin{itemize}
    \item if $\sigma_{[\frac{n}{2}-k]_n}+k=m+k>n,$ then $\sigma^{-1}_m+(\frac{n}{2}-m+1)=\frac{n}{2}-k+(\frac{n}{2}-m+1)\leq 0$, and since both terms are between $-n$ and $n$, $[\sigma_{[\frac{n}{2}-k]_n}+k]_n+[\sigma^{-1}_m+(\frac{n}{2}-m+1)]_n=(m+k-n)+(n-k-m+1+n)=n+1.$
    \item If $\sigma_{[\frac{n}{2}-k]_n}+k\leq n,$ then  $n \geq \sigma^{-1}_m+(\frac{n}{2}-m+1)\geq 1$, and $[\sigma_{[\frac{n}{2}-k]_n}+k]_n+[\sigma^{-1}_m+(\frac{n}{2}-m+1)]_n=(m+k)+(n-k-m+1)=n+1.$

  \end{itemize} Thus, the average of the $\frac{n}{2}$-th entries is
  $$\frac{1}{2n}\sum_{j=0}^{2n-1}\K^j(\sigma)_{\frac{n}{2}}=\frac{1}{2n}\sum_{k=0}^{n-1}(n+1)=\frac{n+1}{2}.$$

  Next, we consider an orbit with size $l$ where $0 < l \leq 2n$. As $l$ must be a divisor of $2n$, there is a positive integer $k$ such that $2n = kl$. Therefore

  $$\displaystyle \sum_{j = 0}^{2n - 1} \K^{j}(\sigma)_i = k \sum_{j = 0}^{l - 1} \K^{j}(\sigma)_i.$$

  So, the average over the orbit when $i=n,$ and when $i=\frac{n}{2}$ for even $n$, is

  $$\frac{\sum_{j = 0}^{l - 1} \K^{j}(\sigma)_i}{l} = \frac{k \sum_{j = 0}^{l - 1} \K^{j}(\sigma)_i}{k l} = \frac{\sum_{j = 0}^{2n - 1} \K^{j}(\sigma)_i}{2n} = \frac{n + 1}{2}.$$  The same results for $\K^{-1}$ follow by Lemma \ref{lem:inverse}.
\end{proof}

When we ran the experiment, the only entries of a permutation that were statistics in FindStat were Statistic 740, the last entry, and Statistic 54, the first entry. Our homomesy result for the lower middle element was found analytically.  Since then, Statistic 1806 - the upper middle entry of a permutation ($\sigma_{\lceil\frac{n+1}{2}\rceil}$), and Statistic 1807 - the lower middle entry ($\sigma_{\lfloor\frac{n+1}{2}\rfloor}$), have been added to the FindStat database.
As noted in the following corollary, the homomesies from Proposition \ref{Khom lastentry} are the only $i$-th entry homomesies possible for the Kreweras complement.

\begin{cor}\label{KC: no entry but}
  No entry except the last entry, and the $\frac{n}{2}$-th entry when $n$ is even, is homomesic with respect to Kreweras complement.
\end{cor}
\begin{proof}
  Consider the set of $i$-th entries of each permutation in $S_n$. As each number between $1$ and $n$ is equally likely to appear in the $i$-th position, the global average of the $i$-th entry is $\frac{1+2+\ldots+n}{n}=\frac{n+1}{2}$.

  By Theorem~\ref{thm:Orbit_generators_K}, for all $n$, $\{1234\ldots n, 234\ldots n1\}$ is an orbit of Kreweras complement. In this orbit, the entries in a given position have to sum to $n+1$ in order for the orbit-average to equal the global average. This is clearly true for the last entry. The sum of the $i$-th entry over these two permutations is $i+(i+1)=2i+1$ for all $1\leq i\leq n-1$. As $2i+1=n+1$ implies $i=\frac{n}{2}$, this orbit shows that only the $n$-th and $\frac{n}{2}$-th entries can be homomesic.
\end{proof}

\bibliographystyle{plain}
\bibliography{master.bib}

\end{document}

%% file: pythonlisting.tex
\usepackage[procnames]{listings}
\usepackage{textcomp}  
\usepackage{setspace}
\usepackage{inconsolata}

\lstnewenvironment{python}[1][]{
\lstset{
xleftmargin=0.0in,
language=Python,
basicstyle=\ttfamily\footnotesize\setstretch{0.75},
basewidth=0.5em,
stringstyle=\color{Purple},
showstringspaces=false,
alsoletter={1234567890},
otherkeywords={1, 2, 3, 4, 5, 6, 7, 8 ,9 , 0, -, =, +, [, ], (, ), \{, \}, :, *, !},
keywordstyle=\color{RoyalBlue},
emph={access,and,as,break,class,continue,def,del,elif,else,%
except,exec,finally,for,from,global,if,import,in,is,%
lambda,not,or,pass,raise,return,try,while,assert},
emphstyle=\color{Orange}\bfseries,
emph={[2]self},
emphstyle=[2]\color{Gray},
emph={[4]abs,all,any,apply,%
basestring,bool,buffer,callable,chr,classmethod,cmp,coerce,compile,%
complex,copyright,credits,delattr,dict,dir,divmod,enumerate,eval,%
execfile,exit,file,filter,float,frozenset,getattr,globals,hasattr,%
hash,help,hex,id,input,int,intern,isinstance,issubclass,iter,len,%
license,list,locals,long,map,max,min,object,oct,open,ord,print,pow,property,%
quit,range,raw_input,reduce,reload,repr,reversed,round,set,setattr,%
slice,sorted,staticmethod,str,sum,super,tuple,type,unichr,unicode,%
vars,xrange,zip},
emphstyle=[4]\color{Green}\bfseries,
emph={[5]sage},
emphstyle=[5]\color{RoyalBlue}\bfseries,
emph={[6]ArithmeticError,AssertionError,AttributeError,BaseException,%
DeprecationWarning,EOFError,Ellipsis,EnvironmentError,Exception,%
FloatingPointError,FutureWarning,GeneratorExit,IOError,%
ImportError,ImportWarning,IndentationError,IndexError,KeyError,%
KeyboardInterrupt,LookupError,MemoryError,NameError,None,%
NotImplemented,NotImplementedError,OSError,OverflowError,%
PendingDeprecationWarning,ReferenceError,RuntimeError,RuntimeWarning,%
StandardError,StopIteration,SyntaxError,SyntaxWarning,SystemError,%
SystemExit,TabError,TypeError,UnboundLocalError,UnicodeDecodeError,%
UnicodeEncodeError,UnicodeError,UnicodeTranslateError,UnicodeWarning,%
UserWarning,ValueError,Warning,ZeroDivisionError,TraceBack},
emphstyle=[6]\color{Red}\bfseries,
emph={[7]True},
emphstyle=[7]\color{Red}\bfseries,
emph={[8]False},
emphstyle=[8]\color{Red}\bfseries,
upquote=true,
morecomment=[s][\color{Gray}]{r"""}{"""},
commentstyle=\color{BrickRed}\slshape,
literate={|->}{{$\longrightarrow$}}{3},
procnamekeys={def,class},
procnamestyle=\color{RoyalBlue}\textbf,
framexleftmargin=1mm,
framextopmargin=1mm,
frame=single,
frameround=tttt,
rulecolor=\color{ForestGreen},
backgroundcolor=\color{White}
}}{}